\documentclass[hidelinks,onefignum,onetabnum]{siamart251216}

\usepackage[a4paper]{geometry}
\usepackage{xspace}
\usepackage{amssymb}
\usepackage{amsopn}
\usepackage{amsfonts}
\usepackage{graphicx}
\usepackage{shuffle}
\usepackage{booktabs}

\usepackage{tikz}
\usetikzlibrary{arrows.meta,positioning}
\usepackage{planarforest}
\newcommand{\forestA}{
\tikz[planar forest default, planar forest, ] {
  \node [  b,  label={[label distance=-1mm]0:{\scriptsize{}}}] at (0.0, 0.0) {}
   ;
 }}

\newcommand{\forestB}{
\tikz[planar forest default, planar forest, ] {
  \node [  b,  label={[label distance=-1mm]0:{\scriptsize{}}}] at (0.0, 0.0) {}
  child {   node [    b,    label={[label distance=-1mm]0:{\scriptsize{}}}  ] at (0.0, 1.0) {}
     edge from parent[    -,    solid, solid,    draw=black  ]   node [!l,right] {\scriptsize{}} }
 ;
 }}

\newcommand{\forestC}{
\tikz[planar forest default, planar forest, ] {
  \node [  b,  label={[label distance=-1mm]0:{\scriptsize{}}}] at (0.0, 0.0) {}
  child {   node [    b,    label={[label distance=-1mm]0:{\scriptsize{}}}  ] at (-0.5, 1.0) {}
     edge from parent[    -,    solid, solid,    draw=black  ]   node [!l,right] {\scriptsize{}} }
child {   node [    b,    label={[label distance=-1mm]0:{\scriptsize{}}}  ] at (0.5, 1.0) {}
     edge from parent[    -,    solid, solid,    draw=black  ]   node [!l,right] {\scriptsize{}} }
 ;
 }}

\newcommand{\forestD}{
\tikz[planar forest default, planar forest, ] {
  \node [  b,  label={[label distance=-1mm]0:{\scriptsize{}}}] at (0.0, 0.0) {}
  child {   node [    b,    label={[label distance=-1mm]0:{\scriptsize{}}}  ] at (0.0, 1.0) {}
  child {   node [    b,    label={[label distance=-1mm]0:{\scriptsize{}}}  ] at (0.0, 1.0) {}
     edge from parent[    -,    solid, solid,    draw=black  ]   node [!l,right] {\scriptsize{}} }
   edge from parent[    -,    solid, solid,    draw=black  ]   node [!l,right] {\scriptsize{}} }
 ;
 }}

\newcommand{\forestE}{
\tikz[planar forest default, planar forest, ] {
  \node [  b,  label={[label distance=-1mm]0:{\scriptsize{}}}] at (0.0, 0.0) {}
  child {   node [    b,    label={[label distance=-1mm]0:{\scriptsize{}}}  ] at (-1.0, 1.0) {}
     edge from parent[    -,    solid, solid,    draw=black  ]   node [!l,right] {\scriptsize{}} }
child {   node [    b,    label={[label distance=-1mm]0:{\scriptsize{}}}  ] at (0.0, 1.0) {}
     edge from parent[    -,    solid, solid,    draw=black  ]   node [!l,right] {\scriptsize{}} }
child {   node [    b,    label={[label distance=-1mm]0:{\scriptsize{}}}  ] at (1.0, 1.0) {}
     edge from parent[    -,    solid, solid,    draw=black  ]   node [!l,right] {\scriptsize{}} }
 ;
 }}

\newcommand{\forestF}{
\tikz[planar forest default, planar forest, ] {
  \node [  b,  label={[label distance=-1mm]0:{\scriptsize{}}}] at (0.0, 0.0) {}
  child {   node [    b,    label={[label distance=-1mm]0:{\scriptsize{}}}  ] at (-0.5, 1.0) {}
     edge from parent[    -,    solid, solid,    draw=black  ]   node [!l,right] {\scriptsize{}} }
child {   node [    b,    label={[label distance=-1mm]0:{\scriptsize{}}}  ] at (0.5, 1.0) {}
  child {   node [    b,    label={[label distance=-1mm]0:{\scriptsize{}}}  ] at (0.0, 1.0) {}
     edge from parent[    -,    solid, solid,    draw=black  ]   node [!l,right] {\scriptsize{}} }
   edge from parent[    -,    solid, solid,    draw=black  ]   node [!l,right] {\scriptsize{}} }
 ;
 }}

\newcommand{\forestG}{
\tikz[planar forest default, planar forest, ] {
  \node [  b,  label={[label distance=-1mm]0:{\scriptsize{}}}] at (0.0, 0.0) {}
  child {   node [    b,    label={[label distance=-1mm]0:{\scriptsize{}}}  ] at (0.0, 1.0) {}
  child {   node [    b,    label={[label distance=-1mm]0:{\scriptsize{}}}  ] at (-0.5, 1.0) {}
     edge from parent[    -,    solid, solid,    draw=black  ]   node [!l,right] {\scriptsize{}} }
child {   node [    b,    label={[label distance=-1mm]0:{\scriptsize{}}}  ] at (0.5, 1.0) {}
     edge from parent[    -,    solid, solid,    draw=black  ]   node [!l,right] {\scriptsize{}} }
   edge from parent[    -,    solid, solid,    draw=black  ]   node [!l,right] {\scriptsize{}} }
 ;
 }}

\newcommand{\forestH}{
\tikz[planar forest default, planar forest, ] {
  \node [  b,  label={[label distance=-1mm]0:{\scriptsize{}}}] at (0.0, 0.0) {}
  child {   node [    b,    label={[label distance=-1mm]0:{\scriptsize{}}}  ] at (0.0, 1.0) {}
  child {   node [    b,    label={[label distance=-1mm]0:{\scriptsize{}}}  ] at (0.0, 1.0) {}
  child {   node [    b,    label={[label distance=-1mm]0:{\scriptsize{}}}  ] at (0.0, 1.0) {}
     edge from parent[    -,    solid, solid,    draw=black  ]   node [!l,right] {\scriptsize{}} }
   edge from parent[    -,    solid, solid,    draw=black  ]   node [!l,right] {\scriptsize{}} }
   edge from parent[    -,    solid, solid,    draw=black  ]   node [!l,right] {\scriptsize{}} }
 ;
 }}

\newcommand{\forestI}{
\tikz[planar forest default, planar forest, ] {
  \node [  b,  label={[label distance=-1mm]0:{\scriptsize{}}}] at (0.0, 0.0) {}
  child {   node [    b,    label={[label distance=-1mm]0:{\scriptsize{}}}  ] at (-0.5, 1.0) {}
     edge from parent[    -,    solid, solid,    draw=black  ]   node [!l,right] {\scriptsize{}} }
child {   node [    b,    label={[label distance=-1mm]0:{\scriptsize{}}}  ] at (0.5, 1.0) {}
  child {   node [    b,    label={[label distance=-1mm]0:{\scriptsize{}}}  ] at (0.0, 1.0) {}
     edge from parent[    -,    solid, solid,    draw=black  ]   node [!l,right] {\scriptsize{}} }
   edge from parent[    -,    solid, solid,    draw=black  ]   node [!l,right] {\scriptsize{}} }
 ;
 }}

\newcommand{\forestJ}{
\tikz[planar forest default, planar forest, ] {
  \node [  b,  label={[label distance=-1mm]0:{\scriptsize{}}}] at (0.0, 0.0) {}
  child {   node [    b,    label={[label distance=-1mm]0:{\scriptsize{}}}  ] at (-0.5, 1.0) {}
  child {   node [    b,    label={[label distance=-1mm]0:{\scriptsize{}}}  ] at (0.0, 1.0) {}
     edge from parent[    -,    solid, solid,    draw=black  ]   node [!l,right] {\scriptsize{}} }
   edge from parent[    -,    solid, solid,    draw=black  ]   node [!l,right] {\scriptsize{}} }
child {   node [    b,    label={[label distance=-1mm]0:{\scriptsize{}}}  ] at (0.5, 1.0) {}
     edge from parent[    -,    solid, solid,    draw=black  ]   node [!l,right] {\scriptsize{}} }
 ;
 }}

\newcommand{\forestK}{
\tikz[planar forest default, planar forest, ] {
  \node [  b,  label={[label distance=-1mm]0:{\scriptsize{}}}] at (0.0, 0.0) {}
   ;
 }}

\newcommand{\forestL}{
\tikz[planar forest default, planar forest, ] {
  \node [  b,  label={[label distance=-1mm]0:{\scriptsize{}}}] at (0.0, 0.0) {}
  child {   node [    b,    label={[label distance=-1mm]0:{\scriptsize{}}}  ] at (-0.5, 1.0) {}
     edge from parent[    -,    solid, solid,    draw=black  ]   node [!l,right] {\scriptsize{}} }
child {   node [    b,    label={[label distance=-1mm]0:{\scriptsize{}}}  ] at (0.5, 1.0) {}
     edge from parent[    -,    solid, solid,    draw=black  ]   node [!l,right] {\scriptsize{}} }
 ;
 }}

\newcommand{\forestM}{
\tikz[planar forest default, planar forest, ] {
  \node [  b,  label={[label distance=-1mm]0:{\scriptsize{}}}] at (0.0, 0.0) {}
  child {   node [    b,    label={[label distance=-1mm]0:{\scriptsize{}}}  ] at (0.0, 1.0) {}
  child {   node [    b,    label={[label distance=-1mm]0:{\scriptsize{}}}  ] at (-0.5, 1.0) {}
     edge from parent[    -,    solid, solid,    draw=black  ]   node [!l,right] {\scriptsize{}} }
child {   node [    b,    label={[label distance=-1mm]0:{\scriptsize{}}}  ] at (0.5, 1.0) {}
     edge from parent[    -,    solid, solid,    draw=black  ]   node [!l,right] {\scriptsize{}} }
   edge from parent[    -,    solid, solid,    draw=black  ]   node [!l,right] {\scriptsize{}} }
 ;
 }}

\newcommand{\forestN}{
\tikz[planar forest default, planar forest, ] {
  \node [  b,  label={[label distance=-1mm]0:{\scriptsize{}}}] at (0.0, 0.0) {}
   ;
\node [  b,  label={[label distance=-1mm]0:{\scriptsize{}}}] at (1.0, 0.0) {}
  child {   node [    b,    label={[label distance=-1mm]0:{\scriptsize{}}}  ] at (0.0, 1.0) {}
     edge from parent[    -,    solid, solid,    draw=black  ]   node [!l,right] {\scriptsize{}} }
 ;
\node [  b,  label={[label distance=-1mm]0:{\scriptsize{}}}] at (2.0, 0.0) {}
   ;
 }}

\newcommand{\forestO}{
\tikz[planar forest default, planar forest, ] {
  \node [  b,  label={[label distance=-1mm]0:{\scriptsize{}}}] at (0.0, 0.0) {}
  child {   node [    b,    label={[label distance=-1mm]0:{\scriptsize{}}}  ] at (-1.0, 1.0) {}
     edge from parent[    -,    solid, solid,    draw=black  ]   node [!l,right] {\scriptsize{}} }
child {   node [    b,    label={[label distance=-1mm]0:{\scriptsize{}}}  ] at (0.0, 1.0) {}
  child {   node [    b,    label={[label distance=-1mm]0:{\scriptsize{}}}  ] at (0.0, 1.0) {}
     edge from parent[    -,    solid, solid,    draw=black  ]   node [!l,right] {\scriptsize{}} }
   edge from parent[    -,    solid, solid,    draw=black  ]   node [!l,right] {\scriptsize{}} }
child {   node [    b,    label={[label distance=-1mm]0:{\scriptsize{}}}  ] at (1.0, 1.0) {}
     edge from parent[    -,    solid, solid,    draw=black  ]   node [!l,right] {\scriptsize{}} }
 ;
 }}

\newcommand{\forestP}{
\tikz[planar forest default, planar forest, ] {
  \node [  b,  label={[label distance=-1mm]0:{\scriptsize{}}}] at (0.0, 0.0) {}
   ;
 }}

\newcommand{\forestQ}{
\tikz[planar forest default, planar forest, ] {
  \node [  b,  label={[label distance=-1mm]0:{\scriptsize{}}}] at (0.0, 0.0) {}
  child {   node [    b,    label={[label distance=-1mm]0:{\scriptsize{}}}  ] at (-0.5, 1.0) {}
     edge from parent[    -,    solid, solid,    draw=black  ]   node [!l,right] {\scriptsize{}} }
child {   node [    b,    label={[label distance=-1mm]0:{\scriptsize{}}}  ] at (0.5, 1.0) {}
  child {   node [    b,    label={[label distance=-1mm]0:{\scriptsize{}}}  ] at (0.0, 1.0) {}
     edge from parent[    -,    solid, solid,    draw=black  ]   node [!l,right] {\scriptsize{}} }
   edge from parent[    -,    solid, solid,    draw=black  ]   node [!l,right] {\scriptsize{}} }
 ;
 }}

\newcommand{\forestR}{
\tikz[planar forest default, planar forest, ] {
  \node [  b,  label={[label distance=-1mm]0:{\scriptsize{}}}] at (0.0, 0.0) {}
   ;
 }}

\newcommand{\forestS}{
\tikz[planar forest default, planar forest, ] {
  \node [  b,  label={[label distance=-1mm]0:{\scriptsize{}}}] at (0.0, 0.0) {}
  child {   node [    b,    label={[label distance=-1mm]0:{\scriptsize{}}}  ] at (-0.5, 1.0) {}
     edge from parent[    -,    solid, solid,    draw=black  ]   node [!l,right] {\scriptsize{}} }
child {   node [    b,    label={[label distance=-1mm]0:{\scriptsize{}}}  ] at (0.5, 1.0) {}
     edge from parent[    -,    solid, solid,    draw=black  ]   node [!l,right] {\scriptsize{}} }
 ;
 }}

\newcommand{\forestT}{
\tikz[planar forest default, planar forest, ] {
  \node [  b,  label={[label distance=-1mm]0:{\scriptsize{}}}] at (0.0, 0.0) {}
  child {   node [    b,    label={[label distance=-1mm]0:{\scriptsize{}}}  ] at (-1.0, 1.0) {}
     edge from parent[    -,    solid, solid,    draw=black  ]   node [!l,right] {\scriptsize{}} }
child {   node [    b,    label={[label distance=-1mm]0:{\scriptsize{}}}  ] at (0.0, 1.0) {}
     edge from parent[    -,    solid, solid,    draw=black  ]   node [!l,right] {\scriptsize{}} }
child {   node [    b,    label={[label distance=-1mm]0:{\scriptsize{}}}  ] at (1.0, 1.0) {}
  child {   node [    b,    label={[label distance=-1mm]0:{\scriptsize{}}}  ] at (-0.5, 1.0) {}
     edge from parent[    -,    solid, solid,    draw=black  ]   node [!l,right] {\scriptsize{}} }
child {   node [    b,    label={[label distance=-1mm]0:{\scriptsize{}}}  ] at (0.5, 1.0) {}
     edge from parent[    -,    solid, solid,    draw=black  ]   node [!l,right] {\scriptsize{}} }
   edge from parent[    -,    solid, solid,    draw=black  ]   node [!l,right] {\scriptsize{}} }
 ;
 }}

\newcommand{\forestU}{
\tikz[planar forest default, planar forest, ] {
  \node [  b,  label={[label distance=-1mm]0:{\scriptsize{}}}] at (0.0, 0.0) {}
  child {   node [    b,    label={[label distance=-1mm]0:{\scriptsize{}}}  ] at (-1.0, 1.0) {}
     edge from parent[    -,    solid, solid,    draw=black  ]   node [!l,right] {\scriptsize{}} }
child {   node [    b,    label={[label distance=-1mm]0:{\scriptsize{}}}  ] at (0.0, 1.0) {}
     edge from parent[    -,    solid, solid,    draw=black  ]   node [!l,right] {\scriptsize{}} }
child {   node [    b,    label={[label distance=-1mm]0:{\scriptsize{}}}  ] at (1.0, 1.0) {}
     edge from parent[    -,    solid, solid,    draw=black  ]   node [!l,right] {\scriptsize{}} }
 ;
 }}

\newcommand{\forestV}{
\tikz[planar forest default, planar forest, ] {
  \node [  b,  label={[label distance=-1mm]0:{\scriptsize{}}}] at (0.0, 0.0) {}
  child {   node [    b,    label={[label distance=-1mm]0:{\scriptsize{}}}  ] at (-0.5, 1.0) {}
     edge from parent[    -,    solid, solid,    draw=black  ]   node [!l,right] {\scriptsize{}} }
child {   node [    b,    label={[label distance=-1mm]0:{\scriptsize{}}}  ] at (0.5, 1.0) {}
  child {   node [    b,    label={[label distance=-1mm]0:{\scriptsize{}}}  ] at (0.0, 1.0) {}
     edge from parent[    -,    solid, solid,    draw=black  ]   node [!l,right] {\scriptsize{}} }
   edge from parent[    -,    solid, solid,    draw=black  ]   node [!l,right] {\scriptsize{}} }
 ;
 }}

\newcommand{\forestW}{
\tikz[planar forest default, planar forest, ] {
  \node [  b,  label={[label distance=-1mm]0:{\scriptsize{}}}] at (0.0, 0.0) {}
   ;
 }}

\newcommand{\forestX}{
\tikz[planar forest default, planar forest, ] {
  \node [  b,  label={[label distance=-1mm]0:{\scriptsize{}}}] at (0.0, 0.0) {}
  child {   node [    b,    label={[label distance=-1mm]0:{\scriptsize{}}}  ] at (0.0, 1.0) {}
     edge from parent[    -,    solid, solid,    draw=black  ]   node [!l,right] {\scriptsize{}} }
 ;
 }}

\newcommand{\forestY}{
\tikz[planar forest default, planar forest, ] {
  \node [  b,  label={[label distance=-1mm]0:{\scriptsize{}}}] at (0.0, 0.0) {}
  child {   node [    b,    label={[label distance=-1mm]0:{\scriptsize{}}}  ] at (-0.5, 1.0) {}
     edge from parent[    -,    solid, solid,    draw=black  ]   node [!l,right] {\scriptsize{}} }
child {   node [    b,    label={[label distance=-1mm]0:{\scriptsize{}}}  ] at (0.5, 1.0) {}
     edge from parent[    -,    solid, solid,    draw=black  ]   node [!l,right] {\scriptsize{}} }
 ;
 }}

\newcommand{\forestAB}{
\tikz[planar forest default, planar forest, ] {
  \node [  b,  label={[label distance=-1mm]0:{\scriptsize{}}}] at (0.0, 0.0) {}
  child {   node [    b,    label={[label distance=-1mm]0:{\scriptsize{}}}  ] at (-1.0, 1.0) {}
  child {   node [    b,    label={[label distance=-1mm]0:{\scriptsize{}}}  ] at (0.0, 1.0) {}
     edge from parent[    -,    solid, solid,    draw=black  ]   node [!l,right] {\scriptsize{}} }
   edge from parent[    -,    solid, solid,    draw=black  ]   node [!l,right] {\scriptsize{}} }
child {   node [    b,    label={[label distance=-1mm]0:{\scriptsize{}}}  ] at (0.0, 1.0) {}
     edge from parent[    -,    solid, solid,    draw=black  ]   node [!l,right] {\scriptsize{}} }
child {   node [    b,    label={[label distance=-1mm]0:{\scriptsize{}}}  ] at (1.0, 1.0) {}
     edge from parent[    -,    solid, solid,    draw=black  ]   node [!l,right] {\scriptsize{}} }
 ;
 }}

\newcommand{\forestBB}{
\tikz[planar forest default, planar forest, ] {
  \node [  b,  label={[label distance=-1mm]0:{\scriptsize{}}}] at (0.0, 0.0) {}
   ;
 }}

\newcommand{\forestCB}{
\tikz[planar forest default, planar forest, ] {
  \node [  b,  label={[label distance=-1mm]0:{\scriptsize{}}}] at (0.0, 0.0) {}
  child {   node [    b,    label={[label distance=-1mm]0:{\scriptsize{}}}  ] at (-0.5, 1.0) {}
     edge from parent[    -,    solid, solid,    draw=black  ]   node [!l,right] {\scriptsize{}} }
child {   node [    b,    label={[label distance=-1mm]0:{\scriptsize{}}}  ] at (0.5, 1.0) {}
     edge from parent[    -,    solid, solid,    draw=black  ]   node [!l,right] {\scriptsize{}} }
 ;
 }}

\newcommand{\forestDB}{
\tikz[planar forest default, planar forest, ] {
  \node [  b,  label={[label distance=-1mm]0:{\scriptsize{}}}] at (0.0, 0.0) {}
  child {   node [    b,    label={[label distance=-1mm]0:{\scriptsize{}}}  ] at (-1.0, 1.0) {}
     edge from parent[    -,    solid, solid,    draw=black  ]   node [!l,right] {\scriptsize{}} }
child {   node [    b,    label={[label distance=-1mm]0:{\scriptsize{}}}  ] at (0.0, 1.0) {}
     edge from parent[    -,    solid, solid,    draw=black  ]   node [!l,right] {\scriptsize{}} }
child {   node [    b,    label={[label distance=-1mm]0:{\scriptsize{}}}  ] at (1.0, 1.0) {}
     edge from parent[    -,    solid, solid,    draw=black  ]   node [!l,right] {\scriptsize{}} }
 ;
 }}

\newcommand{\forestEB}{
\tikz[planar forest default, planar forest, ] {
  \node [  b,  label={[label distance=-1mm]0:{\scriptsize{}}}] at (0.0, 0.0) {}
  child {   node [    b,    label={[label distance=-1mm]0:{\scriptsize{}}}  ] at (-1.5, 1.0) {}
     edge from parent[    -,    solid, solid,    draw=black  ]   node [!l,right] {\scriptsize{}} }
child {   node [    b,    label={[label distance=-1mm]0:{\scriptsize{}}}  ] at (-0.5, 1.0) {}
     edge from parent[    -,    solid, solid,    draw=black  ]   node [!l,right] {\scriptsize{}} }
child {   node [    b,    label={[label distance=-1mm]0:{\scriptsize{}}}  ] at (0.5, 1.0) {}
     edge from parent[    -,    solid, solid,    draw=black  ]   node [!l,right] {\scriptsize{}} }
child {   node [    b,    label={[label distance=-1mm]0:{\scriptsize{}}}  ] at (1.5, 1.0) {}
     edge from parent[    -,    solid, solid,    draw=black  ]   node [!l,right] {\scriptsize{}} }
 ;
 }}

\newcommand{\forestFB}{
\tikz[planar forest default, planar forest, ] {
  \node [  b,  label={[label distance=-1mm]0:{\scriptsize{}}}] at (0.0, 0.0) {}
  child {   node [    b,    label={[label distance=-1mm]0:{\scriptsize{}}}  ] at (-1.0, 1.0) {}
     edge from parent[    -,    solid, solid,    draw=black  ]   node [!l,right] {\scriptsize{}} }
child {   node [    b,    label={[label distance=-1mm]0:{\scriptsize{}}}  ] at (0.0, 1.0) {}
  child {   node [    b,    label={[label distance=-1mm]0:{\scriptsize{}}}  ] at (0.0, 1.0) {}
     edge from parent[    -,    solid, solid,    draw=black  ]   node [!l,right] {\scriptsize{}} }
   edge from parent[    -,    solid, solid,    draw=black  ]   node [!l,right] {\scriptsize{}} }
child {   node [    b,    label={[label distance=-1mm]0:{\scriptsize{}}}  ] at (1.0, 1.0) {}
     edge from parent[    -,    solid, solid,    draw=black  ]   node [!l,right] {\scriptsize{}} }
 ;
 }}

\newcommand{\forestGB}{
\tikz[planar forest default, planar forest, ] {
  \node [  b,  label={[label distance=-1mm]0:{\scriptsize{}}}] at (0.0, 0.0) {}
  child {   node [    b,    label={[label distance=-1mm]0:{\scriptsize{}}}  ] at (-0.5, 1.0) {}
  child {   node [    b,    label={[label distance=-1mm]0:{\scriptsize{}}}  ] at (0.0, 1.0) {}
     edge from parent[    -,    solid, solid,    draw=black  ]   node [!l,right] {\scriptsize{}} }
   edge from parent[    -,    solid, solid,    draw=black  ]   node [!l,right] {\scriptsize{}} }
child {   node [    b,    label={[label distance=-1mm]0:{\scriptsize{}}}  ] at (0.5, 1.0) {}
  child {   node [    b,    label={[label distance=-1mm]0:{\scriptsize{}}}  ] at (0.0, 1.0) {}
     edge from parent[    -,    solid, solid,    draw=black  ]   node [!l,right] {\scriptsize{}} }
   edge from parent[    -,    solid, solid,    draw=black  ]   node [!l,right] {\scriptsize{}} }
 ;
 }}

\newcommand{\forestHB}{
\tikz[planar forest default, planar forest, ] {
  \node [  b,  label={[label distance=-1mm]0:{\scriptsize{}}}] at (0.0, 0.0) {}
   ;
 }}

\newcommand{\forestIB}{
\tikz[planar forest default, planar forest, ] {
  \node [  b,  label={[label distance=-1mm]0:{\scriptsize{}}}] at (0.0, 0.0) {}
  child {   node [    b,    label={[label distance=-1mm]0:{\scriptsize{}}}  ] at (-0.5, 1.0) {}
     edge from parent[    -,    solid, solid,    draw=black  ]   node [!l,right] {\scriptsize{}} }
child {   node [    b,    label={[label distance=-1mm]0:{\scriptsize{}}}  ] at (0.5, 1.0) {}
     edge from parent[    -,    solid, solid,    draw=black  ]   node [!l,right] {\scriptsize{}} }
 ;
 }}

\newcommand{\forestJB}{
\tikz[planar forest default, planar forest, ] {
  \node [  b,  label={[label distance=-1mm]0:{\scriptsize{}}}] at (0.0, 0.0) {}
  child {   node [    b,    label={[label distance=-1mm]0:{\scriptsize{}}}  ] at (-0.5, 1.0) {}
     edge from parent[    -,    solid, solid,    draw=black  ]   node [!l,right] {\scriptsize{}} }
child {   node [    b,    label={[label distance=-1mm]0:{\scriptsize{}}}  ] at (0.5, 1.0) {}
     edge from parent[    -,    solid, solid,    draw=black  ]   node [!l,right] {\scriptsize{}} }
 ;
 }}

\newcommand{\forestKB}{
\tikz[planar forest default, planar forest, ] {
  \node [  b,  label={[label distance=-1mm]0:{\scriptsize{}}}] at (0.0, 0.0) {}
  child {   node [    b,    label={[label distance=-1mm]0:{\scriptsize{}}}  ] at (-1.0, 1.0) {}
     edge from parent[    -,    solid, solid,    draw=black  ]   node [!l,right] {\scriptsize{}} }
child {   node [    b,    label={[label distance=-1mm]0:{\scriptsize{}}}  ] at (0.0, 1.0) {}
     edge from parent[    -,    solid, solid,    draw=black  ]   node [!l,right] {\scriptsize{}} }
child {   node [    b,    label={[label distance=-1mm]0:{\scriptsize{}}}  ] at (1.0, 1.0) {}
     edge from parent[    -,    solid, solid,    draw=black  ]   node [!l,right] {\scriptsize{}} }
 ;
 }}

\newcommand{\forestLB}{
\tikz[planar forest default, planar forest, ] {
  \node [  b,  label={[label distance=-1mm]0:{\scriptsize{}}}] at (0.0, 0.0) {}
  child {   node [    b,    label={[label distance=-1mm]0:{\scriptsize{}}}  ] at (-1.0, 1.0) {}
     edge from parent[    -,    solid, solid,    draw=black  ]   node [!l,right] {\scriptsize{}} }
child {   node [    b,    label={[label distance=-1mm]0:{\scriptsize{}}}  ] at (0.0, 1.0) {}
     edge from parent[    -,    solid, solid,    draw=black  ]   node [!l,right] {\scriptsize{}} }
child {   node [    b,    label={[label distance=-1mm]0:{\scriptsize{}}}  ] at (1.0, 1.0) {}
     edge from parent[    -,    solid, solid,    draw=black  ]   node [!l,right] {\scriptsize{}} }
 ;
 }}

\newcommand{\forestMB}{
\tikz[planar forest default, planar forest, ] {
  \node [  b,  label={[label distance=-1mm]0:{\scriptsize{}}}] at (0.0, 0.0) {}
  child {   node [    b,    label={[label distance=-1mm]0:{\scriptsize{}}}  ] at (-0.5, 1.0) {}
     edge from parent[    -,    solid, solid,    draw=black  ]   node [!l,right] {\scriptsize{}} }
child {   node [    b,    label={[label distance=-1mm]0:{\scriptsize{}}}  ] at (0.5, 1.0) {}
     edge from parent[    -,    solid, solid,    draw=black  ]   node [!l,right] {\scriptsize{}} }
 ;
 }}

\newcommand{\forestNB}{
\tikz[planar forest default, planar forest, ] {
  \node [  b,  label={[label distance=-1mm]0:{\scriptsize{}}}] at (0.0, 0.0) {}
  child {   node [    b,    label={[label distance=-1mm]0:{\scriptsize{}}}  ] at (-1.5, 1.0) {}
     edge from parent[    -,    solid, solid,    draw=black  ]   node [!l,right] {\scriptsize{}} }
child {   node [    b,    label={[label distance=-1mm]0:{\scriptsize{}}}  ] at (-0.5, 1.0) {}
     edge from parent[    -,    solid, solid,    draw=black  ]   node [!l,right] {\scriptsize{}} }
child {   node [    b,    label={[label distance=-1mm]0:{\scriptsize{}}}  ] at (0.5, 1.0) {}
     edge from parent[    -,    solid, solid,    draw=black  ]   node [!l,right] {\scriptsize{}} }
child {   node [    b,    label={[label distance=-1mm]0:{\scriptsize{}}}  ] at (1.5, 1.0) {}
     edge from parent[    -,    solid, solid,    draw=black  ]   node [!l,right] {\scriptsize{}} }
 ;
 }}

\newcommand{\forestOB}{
\tikz[planar forest default, planar forest, ] {
  \node [  b,  label={[label distance=-1mm]0:{\scriptsize{}}}] at (0.0, 0.0) {}
  child {   node [    b,    label={[label distance=-1mm]0:{\scriptsize{}}}  ] at (-1.5, 1.0) {}
     edge from parent[    -,    solid, solid,    draw=black  ]   node [!l,right] {\scriptsize{}} }
child {   node [    b,    label={[label distance=-1mm]0:{\scriptsize{}}}  ] at (-0.5, 1.0) {}
     edge from parent[    -,    solid, solid,    draw=black  ]   node [!l,right] {\scriptsize{}} }
child {   node [    b,    label={[label distance=-1mm]0:{\scriptsize{}}}  ] at (0.5, 1.0) {}
     edge from parent[    -,    solid, solid,    draw=black  ]   node [!l,right] {\scriptsize{}} }
child {   node [    b,    label={[label distance=-1mm]0:{\scriptsize{}}}  ] at (1.5, 1.0) {}
     edge from parent[    -,    solid, solid,    draw=black  ]   node [!l,right] {\scriptsize{}} }
 ;
 }}

\newcommand{\forestPB}{
\tikz[planar forest default, planar forest, ] {
  \node [  b,  label={[label distance=-1mm]0:{\scriptsize{}}}] at (0.0, 0.0) {}
  child {   node [    b,    label={[label distance=-1mm]0:{\scriptsize{}}}  ] at (-1.0, 1.0) {}
     edge from parent[    -,    solid, solid,    draw=black  ]   node [!l,right] {\scriptsize{}} }
child {   node [    b,    label={[label distance=-1mm]0:{\scriptsize{}}}  ] at (0.0, 1.0) {}
     edge from parent[    -,    solid, solid,    draw=black  ]   node [!l,right] {\scriptsize{}} }
child {   node [    b,    label={[label distance=-1mm]0:{\scriptsize{}}}  ] at (1.0, 1.0) {}
     edge from parent[    -,    solid, solid,    draw=black  ]   node [!l,right] {\scriptsize{}} }
 ;
 }}

\newcommand{\forestQB}{
\tikz[planar forest default, planar forest, ] {
  \node [  b,  label={[label distance=-1mm]0:{\scriptsize{}}}] at (0.0, 0.0) {}
  child {   node [    b,    label={[label distance=-1mm]0:{\scriptsize{}}}  ] at (-0.5, 1.0) {}
     edge from parent[    -,    solid, solid,    draw=black  ]   node [!l,right] {\scriptsize{}} }
child {   node [    b,    label={[label distance=-1mm]0:{\scriptsize{}}}  ] at (0.5, 1.0) {}
     edge from parent[    -,    solid, solid,    draw=black  ]   node [!l,right] {\scriptsize{}} }
 ;
 }}

\newcommand{\forestRB}{
\tikz[planar forest default, planar forest, ] {
  \node [  b,  label={[label distance=-1mm]0:{\scriptsize{}}}] at (0.0, 0.0) {}
  child {   node [    b,    label={[label distance=-1mm]0:{\scriptsize{}}}  ] at (-1.0, 1.0) {}
     edge from parent[    -,    solid, solid,    draw=black  ]   node [!l,right] {\scriptsize{}} }
child {   node [    b,    label={[label distance=-1mm]0:{\scriptsize{}}}  ] at (0.0, 1.0) {}
  child {   node [    b,    label={[label distance=-1mm]0:{\scriptsize{}}}  ] at (0.0, 1.0) {}
     edge from parent[    -,    solid, solid,    draw=black  ]   node [!l,right] {\scriptsize{}} }
   edge from parent[    -,    solid, solid,    draw=black  ]   node [!l,right] {\scriptsize{}} }
child {   node [    b,    label={[label distance=-1mm]0:{\scriptsize{}}}  ] at (1.0, 1.0) {}
     edge from parent[    -,    solid, solid,    draw=black  ]   node [!l,right] {\scriptsize{}} }
 ;
 }}

\newcommand{\forestSB}{
\tikz[planar forest default, planar forest, ] {
  \node [  b,  label={[label distance=-1mm]0:{\scriptsize{}}}] at (0.0, 0.0) {}
  child {   node [    b,    label={[label distance=-1mm]0:{\scriptsize{}}}  ] at (-1.0, 1.0) {}
     edge from parent[    -,    solid, solid,    draw=black  ]   node [!l,right] {\scriptsize{}} }
child {   node [    b,    label={[label distance=-1mm]0:{\scriptsize{}}}  ] at (0.0, 1.0) {}
  child {   node [    b,    label={[label distance=-1mm]0:{\scriptsize{}}}  ] at (0.0, 1.0) {}
     edge from parent[    -,    solid, solid,    draw=black  ]   node [!l,right] {\scriptsize{}} }
   edge from parent[    -,    solid, solid,    draw=black  ]   node [!l,right] {\scriptsize{}} }
child {   node [    b,    label={[label distance=-1mm]0:{\scriptsize{}}}  ] at (1.0, 1.0) {}
     edge from parent[    -,    solid, solid,    draw=black  ]   node [!l,right] {\scriptsize{}} }
 ;
 }}

\newcommand{\forestTB}{
\tikz[planar forest default, planar forest, ] {
  \node [  b,  label={[label distance=-1mm]0:{\scriptsize{}}}] at (0.0, 0.0) {}
  child {   node [    b,    label={[label distance=-1mm]0:{\scriptsize{}}}  ] at (-1.0, 1.0) {}
     edge from parent[    -,    solid, solid,    draw=black  ]   node [!l,right] {\scriptsize{}} }
child {   node [    b,    label={[label distance=-1mm]0:{\scriptsize{}}}  ] at (0.0, 1.0) {}
     edge from parent[    -,    solid, solid,    draw=black  ]   node [!l,right] {\scriptsize{}} }
child {   node [    b,    label={[label distance=-1mm]0:{\scriptsize{}}}  ] at (1.0, 1.0) {}
     edge from parent[    -,    solid, solid,    draw=black  ]   node [!l,right] {\scriptsize{}} }
 ;
 }}

\newcommand{\forestUB}{
\tikz[planar forest default, planar forest, ] {
  \node [  b,  label={[label distance=-1mm]0:{\scriptsize{}}}] at (0.0, 0.0) {}
  child {   node [    b,    label={[label distance=-1mm]0:{\scriptsize{}}}  ] at (-0.5, 1.0) {}
     edge from parent[    -,    solid, solid,    draw=black  ]   node [!l,right] {\scriptsize{}} }
child {   node [    b,    label={[label distance=-1mm]0:{\scriptsize{}}}  ] at (0.5, 1.0) {}
     edge from parent[    -,    solid, solid,    draw=black  ]   node [!l,right] {\scriptsize{}} }
 ;
 }}

\newcommand{\forestVB}{
\tikz[planar forest default, planar forest, ] {
  \node [  b,  label={[label distance=-1mm]0:{\scriptsize{}}}] at (0.0, 0.0) {}
  child {   node [    b,    label={[label distance=-1mm]0:{\scriptsize{}}}  ] at (-0.5, 1.0) {}
  child {   node [    b,    label={[label distance=-1mm]0:{\scriptsize{}}}  ] at (0.0, 1.0) {}
     edge from parent[    -,    solid, solid,    draw=black  ]   node [!l,right] {\scriptsize{}} }
   edge from parent[    -,    solid, solid,    draw=black  ]   node [!l,right] {\scriptsize{}} }
child {   node [    b,    label={[label distance=-1mm]0:{\scriptsize{}}}  ] at (0.5, 1.0) {}
  child {   node [    b,    label={[label distance=-1mm]0:{\scriptsize{}}}  ] at (0.0, 1.0) {}
     edge from parent[    -,    solid, solid,    draw=black  ]   node [!l,right] {\scriptsize{}} }
   edge from parent[    -,    solid, solid,    draw=black  ]   node [!l,right] {\scriptsize{}} }
 ;
 }}

\newcommand{\forestWB}{
\tikz[planar forest default, planar forest, ] {
  \node [  b,  label={[label distance=-1mm]0:{\scriptsize{}}}] at (0.0, 0.0) {}
  child {   node [    b,    label={[label distance=-1mm]0:{\scriptsize{}}}  ] at (-0.5, 1.0) {}
  child {   node [    b,    label={[label distance=-1mm]0:{\scriptsize{}}}  ] at (0.0, 1.0) {}
     edge from parent[    -,    solid, solid,    draw=black  ]   node [!l,right] {\scriptsize{}} }
   edge from parent[    -,    solid, solid,    draw=black  ]   node [!l,right] {\scriptsize{}} }
child {   node [    b,    label={[label distance=-1mm]0:{\scriptsize{}}}  ] at (0.5, 1.0) {}
  child {   node [    b,    label={[label distance=-1mm]0:{\scriptsize{}}}  ] at (0.0, 1.0) {}
     edge from parent[    -,    solid, solid,    draw=black  ]   node [!l,right] {\scriptsize{}} }
   edge from parent[    -,    solid, solid,    draw=black  ]   node [!l,right] {\scriptsize{}} }
 ;
 }}

\newcommand{\forestXB}{
\tikz[planar forest default, planar forest, ] {
  \node [  b,  label={[label distance=-1mm]0:{\scriptsize{}}}] at (0.0, 0.0) {}
  child {   node [    b,    label={[label distance=-1mm]0:{\scriptsize{}}}  ] at (-0.5, 1.0) {}
     edge from parent[    -,    solid, solid,    draw=black  ]   node [!l,right] {\scriptsize{}} }
child {   node [    b,    label={[label distance=-1mm]0:{\scriptsize{}}}  ] at (0.5, 1.0) {}
     edge from parent[    -,    solid, solid,    draw=black  ]   node [!l,right] {\scriptsize{}} }
 ;
 }}

\newcommand{\forestYB}{
\tikz[planar forest default, planar forest, ] {
  \node [  b,  label={[label distance=-1mm]0:{\scriptsize{}}}] at (0.0, 0.0) {}
  child {   node [    b,    label={[label distance=-1mm]0:{\scriptsize{}}}  ] at (-1.0, 1.0) {}
     edge from parent[    -,    solid, solid,    draw=black  ]   node [!l,right] {\scriptsize{}} }
child {   node [    b,    label={[label distance=-1mm]0:{\scriptsize{}}}  ] at (0.0, 1.0) {}
     edge from parent[    -,    solid, solid,    draw=black  ]   node [!l,right] {\scriptsize{}} }
child {   node [    b,    label={[label distance=-1mm]0:{\scriptsize{}}}  ] at (1.0, 1.0) {}
     edge from parent[    -,    solid, solid,    draw=black  ]   node [!l,right] {\scriptsize{}} }
 ;
 }}

\newcommand{\forestAC}{
\tikz[planar forest default, planar forest, ] {
  \node [  b,  label={[label distance=-1mm]0:{\scriptsize{}}}] at (0.0, 0.0) {}
  child {   node [    b,    label={[label distance=-1mm]0:{\scriptsize{}}}  ] at (-1.5, 1.0) {}
     edge from parent[    -,    solid, solid,    draw=black  ]   node [!l,right] {\scriptsize{}} }
child {   node [    b,    label={[label distance=-1mm]0:{\scriptsize{}}}  ] at (-0.5, 1.0) {}
     edge from parent[    -,    solid, solid,    draw=black  ]   node [!l,right] {\scriptsize{}} }
child {   node [    b,    label={[label distance=-1mm]0:{\scriptsize{}}}  ] at (0.5, 1.0) {}
     edge from parent[    -,    solid, solid,    draw=black  ]   node [!l,right] {\scriptsize{}} }
child {   node [    b,    label={[label distance=-1mm]0:{\scriptsize{}}}  ] at (1.5, 1.0) {}
     edge from parent[    -,    solid, solid,    draw=black  ]   node [!l,right] {\scriptsize{}} }
 ;
 }}

\newcommand{\forestBC}{
\tikz[planar forest default, planar forest, ] {
  \node [  b,  label={[label distance=-1mm]0:{\scriptsize{}}}] at (0.0, 0.0) {}
  child {   node [    b,    label={[label distance=-1mm]0:{\scriptsize{}}}  ] at (-1.0, 1.0) {}
     edge from parent[    -,    solid, solid,    draw=black  ]   node [!l,right] {\scriptsize{}} }
child {   node [    b,    label={[label distance=-1mm]0:{\scriptsize{}}}  ] at (0.0, 1.0) {}
  child {   node [    b,    label={[label distance=-1mm]0:{\scriptsize{}}}  ] at (0.0, 1.0) {}
     edge from parent[    -,    solid, solid,    draw=black  ]   node [!l,right] {\scriptsize{}} }
   edge from parent[    -,    solid, solid,    draw=black  ]   node [!l,right] {\scriptsize{}} }
child {   node [    b,    label={[label distance=-1mm]0:{\scriptsize{}}}  ] at (1.0, 1.0) {}
     edge from parent[    -,    solid, solid,    draw=black  ]   node [!l,right] {\scriptsize{}} }
 ;
 }}

\newcommand{\forestCC}{
\tikz[planar forest default, planar forest, ] {
  \node [  b,  label={[label distance=-1mm]0:{\scriptsize{}}}] at (0.0, 0.0) {}
  child {   node [    b,    label={[label distance=-1mm]0:{\scriptsize{}}}  ] at (-0.5, 1.0) {}
  child {   node [    b,    label={[label distance=-1mm]0:{\scriptsize{}}}  ] at (0.0, 1.0) {}
     edge from parent[    -,    solid, solid,    draw=black  ]   node [!l,right] {\scriptsize{}} }
   edge from parent[    -,    solid, solid,    draw=black  ]   node [!l,right] {\scriptsize{}} }
child {   node [    b,    label={[label distance=-1mm]0:{\scriptsize{}}}  ] at (0.5, 1.0) {}
  child {   node [    b,    label={[label distance=-1mm]0:{\scriptsize{}}}  ] at (0.0, 1.0) {}
     edge from parent[    -,    solid, solid,    draw=black  ]   node [!l,right] {\scriptsize{}} }
   edge from parent[    -,    solid, solid,    draw=black  ]   node [!l,right] {\scriptsize{}} }
 ;
 }}

\newcommand{\forestDC}{
\tikz[planar forest default, planar forest, ] {
  \node [  b,  label={[label distance=-1mm]0:{\scriptsize{}}}] at (0.0, 0.0) {}
   ;
 }}

\newcommand{\forestEC}{
\tikz[planar forest default, planar forest, ] {
  \node [  b,  label={[label distance=-1mm]0:{\scriptsize{}}}] at (0.0, 0.0) {}
  child {   node [    b,    label={[label distance=-1mm]0:{\scriptsize{}}}  ] at (-1.5, 1.0) {}
  child {   node [    b,    label={[label distance=-1mm]0:{\scriptsize{}}}  ] at (0.0, 1.0) {}
     edge from parent[    -,    solid, solid,    draw=black  ]   node [!l,right] {\scriptsize{}} }
   edge from parent[    -,    solid, solid,    draw=black  ]   node [!l,right] {\scriptsize{}} }
child {   node [    b,    label={[label distance=-1mm]0:{\scriptsize{}}}  ] at (-0.5, 1.0) {}
  child {   node [    b,    label={[label distance=-1mm]0:{\scriptsize{}}}  ] at (0.0, 1.0) {}
     edge from parent[    -,    solid, solid,    draw=black  ]   node [!l,right] {\scriptsize{}} }
   edge from parent[    -,    solid, solid,    draw=black  ]   node [!l,right] {\scriptsize{}} }
child {   node [    .l,    label={[label distance=-1mm]0:{\scriptsize{}}}  ] at (0.5, 1.0) {$\times$}
     edge from parent[    -,    solid, solid,    draw=none  ]   node [!l,right] {\scriptsize{}} }
child {   node [    b,    label={[label distance=-1mm]0:{\scriptsize{}}}  ] at (1.5, 1.0) {}
  child {   node [    b,    label={[label distance=-1mm]0:{\scriptsize{}}}  ] at (0.0, 1.0) {}
     edge from parent[    -,    solid, solid,    draw=black  ]   node [!l,right] {\scriptsize{}} }
   edge from parent[    -,    solid, solid,    draw=black  ]   node [!l,right] {\scriptsize{}} }
 ;
 }}

\newcommand{\forestFC}{
\tikz[planar forest default, planar forest, ] {
  \node [  b,  label={[label distance=-1mm]0:{\scriptsize{}}}] at (0.0, 0.0) {}
  child {   node [    b,    label={[label distance=-1mm]0:{\scriptsize{}}}  ] at (-0.5, 1.0) {}
     edge from parent[    -,    solid, solid,    draw=black  ]   node [!l,right] {\scriptsize{}} }
child {   node [    b,    label={[label distance=-1mm]0:{\scriptsize{}}}  ] at (0.5, 1.0) {}
     edge from parent[    -,    solid, solid,    draw=black  ]   node [!l,right] {\scriptsize{}} }
 ;
\node [  b,  label={[label distance=-1mm]0:{\scriptsize{}}}] at (1.5, 0.0) {}
  child {   node [    b,    label={[label distance=-1mm]0:{\scriptsize{}}}  ] at (0.0, 1.0) {}
     edge from parent[    -,    solid, solid,    draw=black  ]   node [!l,right] {\scriptsize{}} }
 ;
\node [  .l,  label={[label distance=-1mm]0:{\scriptsize{}}}] at (2.5, 0.0) {$\times$}
   ;
\node [  b,  label={[label distance=-1mm]0:{\scriptsize{}}}] at (3.5, 0.0) {}
   ;
 }}

\newcommand{\forestGC}{
\tikz[planar forest default, planar forest, ] {
  \node [  b,  label={[label distance=-1mm]0:{\scriptsize{}}}] at (0.0, 0.0) {}
  child {   node [    b,    label={[label distance=-1mm]0:{\scriptsize{}}}  ] at (-0.5, 1.0) {}
     edge from parent[    -,    solid, solid,    draw=black  ]   node [!l,right] {\scriptsize{}} }
child {   node [    b,    label={[label distance=-1mm]0:{\scriptsize{}}}  ] at (0.5, 1.0) {}
     edge from parent[    -,    solid, solid,    draw=black  ]   node [!l,right] {\scriptsize{}} }
 ;
\node [  b,  label={[label distance=-1mm]0:{\scriptsize{}}}] at (1.5, 0.0) {}
  child {   node [    b,    label={[label distance=-1mm]0:{\scriptsize{}}}  ] at (0.0, 1.0) {}
     edge from parent[    -,    solid, solid,    draw=black  ]   node [!l,right] {\scriptsize{}} }
 ;
\node [  .l,  label={[label distance=-1mm]0:{\scriptsize{}}}] at (2.5, 0.0) {$\times$}
   ;
\node [  b,  label={[label distance=-1mm]0:{\scriptsize{}}}] at (3.5, 0.0) {}
   ;
 }}

\newcommand{\forestHC}{
\tikz[planar forest default, planar forest, ] {
  \node [  b,  label={[label distance=-1mm]0:{\scriptsize{}}}] at (0.0, 0.0) {}
  child {   node [    b,    label={[label distance=-1mm]0:{\scriptsize{}}}  ] at (0.0, 1.0) {}
     edge from parent[    -,    solid, solid,    draw=black  ]   node [!l,right] {\scriptsize{}} }
 ;
\node [  b,  label={[label distance=-1mm]0:{\scriptsize{}}}] at (1.5, 0.0) {}
  child {   node [    b,    label={[label distance=-1mm]0:{\scriptsize{}}}  ] at (-0.5, 1.0) {}
     edge from parent[    -,    solid, solid,    draw=black  ]   node [!l,right] {\scriptsize{}} }
child {   node [    b,    label={[label distance=-1mm]0:{\scriptsize{}}}  ] at (0.5, 1.0) {}
     edge from parent[    -,    solid, solid,    draw=black  ]   node [!l,right] {\scriptsize{}} }
 ;
\node [  .l,  label={[label distance=-1mm]0:{\scriptsize{}}}] at (2.5, 0.0) {$\times$}
   ;
\node [  b,  label={[label distance=-1mm]0:{\scriptsize{}}}] at (3.5, 0.0) {}
   ;
 }}

\newcommand{\forestIC}{
\tikz[planar forest default, planar forest, ] {
  \node [  b,  label={[label distance=-1mm]0:{\scriptsize{}}}] at (0.0, 0.0) {}
  child {   node [    b,    label={[label distance=-1mm]0:{\scriptsize{}}}  ] at (-0.5, 1.0) {}
     edge from parent[    -,    solid, solid,    draw=black  ]   node [!l,right] {\scriptsize{}} }
child {   node [    b,    label={[label distance=-1mm]0:{\scriptsize{}}}  ] at (0.5, 1.0) {}
     edge from parent[    -,    solid, solid,    draw=black  ]   node [!l,right] {\scriptsize{}} }
 ;
\node [  b,  label={[label distance=-1mm]0:{\scriptsize{}}}] at (1.5, 0.0) {}
  child {   node [    b,    label={[label distance=-1mm]0:{\scriptsize{}}}  ] at (0.0, 1.0) {}
     edge from parent[    -,    solid, solid,    draw=black  ]   node [!l,right] {\scriptsize{}} }
 ;
\node [  .l,  label={[label distance=-1mm]0:{\scriptsize{}}}] at (2.5, 0.0) {$\times$}
   ;
\node [  b,  label={[label distance=-1mm]0:{\scriptsize{}}}] at (3.5, 0.0) {}
   ;
 }}

\newcommand{\forestJC}{
\tikz[planar forest default, planar forest, ] {
  \node [  b,  label={[label distance=-1mm]0:{\scriptsize{}}}] at (0.0, 0.0) {}
   ;
\node [  .l,  label={[label distance=-1mm]0:{\scriptsize{}}}] at (1.0, 0.0) {$\times$}
   ;
\node [  b,  label={[label distance=-1mm]0:{\scriptsize{}}}] at (2.0, 0.0) {}
  child {   node [    b,    label={[label distance=-1mm]0:{\scriptsize{}}}  ] at (-0.5, 1.0) {}
     edge from parent[    -,    solid, solid,    draw=black  ]   node [!l,right] {\scriptsize{}} }
child {   node [    b,    label={[label distance=-1mm]0:{\scriptsize{}}}  ] at (0.5, 1.0) {}
     edge from parent[    -,    solid, solid,    draw=black  ]   node [!l,right] {\scriptsize{}} }
 ;
\node [  b,  label={[label distance=-1mm]0:{\scriptsize{}}}] at (3.5, 0.0) {}
  child {   node [    b,    label={[label distance=-1mm]0:{\scriptsize{}}}  ] at (0.0, 1.0) {}
     edge from parent[    -,    solid, solid,    draw=black  ]   node [!l,right] {\scriptsize{}} }
 ;
 }}

\newcommand{\forestKC}{
\tikz[planar forest default, planar forest, ] {
  \node [  b,  label={[label distance=-1mm]0:{\scriptsize{}}}] at (0.0, 0.0) {}
  child {   node [    b,    label={[label distance=-1mm]0:{\scriptsize{}}}  ] at (0.0, 1.0) {}
     edge from parent[    -,    solid, solid,    draw=black  ]   node [!l,right] {\scriptsize{}} }
 ;
\node [  .l,  label={[label distance=-1mm]0:{\scriptsize{}}}] at (1.0, 0.0) {$\times$}
   ;
\node [  b,  label={[label distance=-1mm]0:{\scriptsize{}}}] at (2.0, 0.0) {}
   ;
 }}

\newcommand{\forestLC}{
\tikz[planar forest default, planar forest, ] {
  \node [  b,  label={[label distance=-1mm]0:{\scriptsize{}}}] at (0.0, 0.0) {}
  child {   node [    b,    label={[label distance=-1mm]0:{\scriptsize{}}}  ] at (0.0, 1.0) {}
     edge from parent[    -,    solid, solid,    draw=black  ]   node [!l,right] {\scriptsize{}} }
 ;
 }}

\newcommand{\forestMC}{
\tikz[planar forest default, planar forest, ] {
  \node [  b,  label={[label distance=-1mm]0:{\scriptsize{}}}] at (0.0, 0.0) {}
   ;
 }}

\newcommand{\forestNC}{
\tikz[planar forest default, planar forest, ] {
  \node [  b,  label={[label distance=-1mm]0:{\scriptsize{}}}] at (0.0, 0.0) {}
   ;
 }}

\newcommand{\forestOC}{
\tikz[planar forest default, planar forest, ] {
  \node [  b,  label={[label distance=-1mm]0:{\scriptsize{}}}] at (0.0, 0.0) {}
  child {   node [    b,    label={[label distance=-1mm]0:{\scriptsize{}}}  ] at (0.0, 1.0) {}
     edge from parent[    -,    solid, solid,    draw=black  ]   node [!l,right] {\scriptsize{}} }
 ;
 }}

\newcommand{\forestPC}{
\tikz[planar forest default, planar forest, ] {
  \node [  b,  label={[label distance=-1mm]0:{\scriptsize{}}}] at (0.0, 0.0) {}
  child {   node [    b,    label={[label distance=-1mm]0:{\scriptsize{}}}  ] at (-1.0, 1.0) {}
     edge from parent[    -,    solid, solid,    draw=black  ]   node [!l,right] {\scriptsize{}} }
child {   node [    .l,    label={[label distance=-1mm]0:{\scriptsize{}}}  ] at (0.0, 1.0) {$\times$}
     edge from parent[    -,    solid, solid,    draw=none  ]   node [!l,right] {\scriptsize{}} }
child {   node [    b,    label={[label distance=-1mm]0:{\scriptsize{}}}  ] at (1.0, 1.0) {}
     edge from parent[    -,    solid, solid,    draw=black  ]   node [!l,right] {\scriptsize{}} }
 ;
\node [  b,  label={[label distance=-1mm]0:{\scriptsize{}}}] at (1.0, 0.0) {}
   ;
\node [  .l,  label={[label distance=-1mm]0:{\scriptsize{}}}] at (2.0, 0.0) {$\times$}
   ;
\node [  b,  label={[label distance=-1mm]0:{\scriptsize{}}}] at (3.0, 0.0) {}
   ;
 }}

\newcommand{\forestQC}{
\tikz[planar forest default, planar forest, ] {
  \node [  b,  label={[label distance=-1mm]0:{\scriptsize{}}}] at (0.0, 0.0) {}
  child {   node [    b,    label={[label distance=-1mm]0:{\scriptsize{}}}  ] at (-1.0, 1.0) {}
     edge from parent[    -,    solid, solid,    draw=black  ]   node [!l,right] {\scriptsize{}} }
child {   node [    .l,    label={[label distance=-1mm]0:{\scriptsize{}}}  ] at (0.0, 1.0) {$\times$}
     edge from parent[    -,    solid, solid,    draw=none  ]   node [!l,right] {\scriptsize{}} }
child {   node [    b,    label={[label distance=-1mm]0:{\scriptsize{}}}  ] at (1.0, 1.0) {}
     edge from parent[    -,    solid, solid,    draw=black  ]   node [!l,right] {\scriptsize{}} }
 ;
\node [  .l,  label={[label distance=-1mm]0:{\scriptsize{}}}] at (1.5, 0.0) {$\times$}
   ;
\node [  b,  label={[label distance=-1mm]0:{\scriptsize{}}}] at (3.0, 0.0) {}
  child {   node [    b,    label={[label distance=-1mm]0:{\scriptsize{}}}  ] at (-1.0, 1.0) {}
     edge from parent[    -,    solid, solid,    draw=black  ]   node [!l,right] {\scriptsize{}} }
child {   node [    .l,    label={[label distance=-1mm]0:{\scriptsize{}}}  ] at (0.0, 1.0) {$\times$}
     edge from parent[    -,    solid, solid,    draw=none  ]   node [!l,right] {\scriptsize{}} }
child {   node [    b,    label={[label distance=-1mm]0:{\scriptsize{}}}  ] at (1.0, 1.0) {}
     edge from parent[    -,    solid, solid,    draw=black  ]   node [!l,right] {\scriptsize{}} }
 ;
 }}

\newcommand{\forestRC}{
\tikz[planar forest default, planar forest, ] {
  \node [  b,  label={[label distance=-1mm]0:{\scriptsize{}}}] at (0.0, 0.0) {}
  child {   node [    b,    label={[label distance=-1mm]0:{\scriptsize{}}}  ] at (-0.5, 1.0) {}
     edge from parent[    -,    solid, solid,    draw=black  ]   node [!l,right] {\scriptsize{}} }
child {   node [    b,    label={[label distance=-1mm]0:{\scriptsize{}}}  ] at (0.5, 1.0) {}
     edge from parent[    -,    solid, solid,    draw=black  ]   node [!l,right] {\scriptsize{}} }
 ;
\node [  .l,  label={[label distance=-1mm]0:{\scriptsize{}}}] at (1.0, 0.0) {$\times$}
   ;
\node [  b,  label={[label distance=-1mm]0:{\scriptsize{}}}] at (2.0, 0.0) {}
  child {   node [    b,    label={[label distance=-1mm]0:{\scriptsize{}}}  ] at (-0.5, 1.0) {}
     edge from parent[    -,    solid, solid,    draw=black  ]   node [!l,right] {\scriptsize{}} }
child {   node [    b,    label={[label distance=-1mm]0:{\scriptsize{}}}  ] at (0.5, 1.0) {}
     edge from parent[    -,    solid, solid,    draw=black  ]   node [!l,right] {\scriptsize{}} }
 ;
 }}

\newcommand{\forestSC}{
\tikz[planar forest default, planar forest, ] {
  \node [  b,  label={[label distance=-1mm]0:{\scriptsize{}}}] at (0.0, 0.0) {}
  child {   node [    b,    label={[label distance=-1mm]0:{\scriptsize{}}}  ] at (-1.5, 1.0) {}
     edge from parent[    -,    solid, solid,    draw=black  ]   node [!l,right] {\scriptsize{}} }
child {   node [    b,    label={[label distance=-1mm]0:{\scriptsize{}}}  ] at (-0.5, 1.0) {}
     edge from parent[    -,    solid, solid,    draw=black  ]   node [!l,right] {\scriptsize{}} }
child {   node [    .l,    label={[label distance=-1mm]0:{\scriptsize{}}}  ] at (0.5, 1.0) {$\times$}
     edge from parent[    -,    solid, solid,    draw=none  ]   node [!l,right] {\scriptsize{}} }
child {   node [    b,    label={[label distance=-1mm]0:{\scriptsize{}}}  ] at (1.5, 1.0) {}
  child {   node [    b,    label={[label distance=-1mm]0:{\scriptsize{}}}  ] at (0.0, 1.0) {}
     edge from parent[    -,    solid, solid,    draw=black  ]   node [!l,right] {\scriptsize{}} }
   edge from parent[    -,    solid, solid,    draw=black  ]   node [!l,right] {\scriptsize{}} }
 ;
\node [  .l,  label={[label distance=-1mm]0:{\scriptsize{}}}] at (1.0, 0.0) {$\times$}
   ;
\node [  b,  label={[label distance=-1mm]0:{\scriptsize{}}}] at (2.0, 0.0) {}
   ;
 }}

\newcommand{\forestTC}{
\tikz[planar forest default, planar forest, ] {
  \node [  b,  label={[label distance=-1mm]0:{\scriptsize{}}}] at (0.0, 0.0) {}
   ;
 }}

\newcommand{\forestUC}{
\tikz[planar forest default, planar forest, ] {
  \node [  b,  label={[label distance=-1mm]0:{\scriptsize{}}}] at (0.0, 0.0) {}
   ;
 }}

\newcommand{\forestVC}{
\tikz[planar forest default, planar forest, ] {
  \node [  b,  label={[label distance=-1mm]0:{\scriptsize{}}}] at (0.0, 0.0) {}
   ;
\node [  b,  label={[label distance=-1mm]0:{\scriptsize{}}}] at (1.0, 0.0) {}
   ;
 }}

\newcommand{\forestWC}{
\tikz[planar forest default, planar forest, ] {
  \node [  b,  label={[label distance=-1mm]0:{\scriptsize{}}}] at (0.0, 0.0) {}
  child {   node [    b,    label={[label distance=-1mm]0:{\scriptsize{}}}  ] at (0.0, 1.0) {}
     edge from parent[    -,    solid, solid,    draw=black  ]   node [!l,right] {\scriptsize{}} }
 ;
 }}

\newcommand{\forestXC}{
\tikz[planar forest default, planar forest, ] {
  \node [  b,  label={[label distance=-1mm]0:{\scriptsize{}}}] at (0.0, 0.0) {}
  child {   node [    b,    label={[label distance=-1mm]0:{\scriptsize{}}}  ] at (0.0, 1.0) {}
     edge from parent[    -,    solid, solid,    draw=black  ]   node [!l,right] {\scriptsize{}} }
 ;
 }}

\newcommand{\forestYC}{
\tikz[planar forest default, planar forest, ] {
  \node [  b,  label={[label distance=-1mm]0:{\scriptsize{}}}] at (0.0, 0.0) {}
   ;
\node [  b,  label={[label distance=-1mm]0:{\scriptsize{}}}] at (1.0, 0.0) {}
  child {   node [    b,    label={[label distance=-1mm]0:{\scriptsize{}}}  ] at (0.0, 1.0) {}
     edge from parent[    -,    solid, solid,    draw=black  ]   node [!l,right] {\scriptsize{}} }
 ;
 }}

\newcommand{\forestAD}{
\tikz[planar forest default, planar forest, ] {
  \node [  b,  label={[label distance=-1mm]0:{\scriptsize{}}}] at (0.0, 0.0) {}
  child {   node [    b,    label={[label distance=-1mm]0:{\scriptsize{}}}  ] at (-1.0, 1.0) {}
     edge from parent[    -,    solid, solid,    draw=black  ]   node [!l,right] {\scriptsize{}} }
child {   node [    .l,    label={[label distance=-1mm]0:{\scriptsize{}}}  ] at (0.0, 1.0) {$\times$}
     edge from parent[    -,    solid, solid,    draw=none  ]   node [!l,right] {\scriptsize{}} }
child {   node [    b,    label={[label distance=-1mm]0:{\scriptsize{}}}  ] at (1.0, 1.0) {}
     edge from parent[    -,    solid, solid,    draw=black  ]   node [!l,right] {\scriptsize{}} }
 ;
 }}

\newcommand{\forestBD}{
\tikz[planar forest default, planar forest, ] {
  \node [  b,  label={[label distance=-1mm]0:{\scriptsize{}}}] at (0.0, 0.0) {}
  child {   node [    b,    label={[label distance=-1mm]0:{\scriptsize{}}}  ] at (-0.5, 1.0) {}
     edge from parent[    -,    solid, solid,    draw=black  ]   node [!l,right] {\scriptsize{}} }
child {   node [    b,    label={[label distance=-1mm]0:{\scriptsize{}}}  ] at (0.5, 1.0) {}
     edge from parent[    -,    solid, solid,    draw=black  ]   node [!l,right] {\scriptsize{}} }
 ;
 }}

\newcommand{\forestCD}{
\tikz[planar forest default, planar forest, ] {
  \node [  b,  label={[label distance=-1mm]0:{\scriptsize{}}}] at (0.0, 0.0) {}
   ;
\node [  b,  label={[label distance=-1mm]0:{\scriptsize{}}}] at (1.0, 0.0) {}
   ;
\node [  b,  label={[label distance=-1mm]0:{\scriptsize{}}}] at (2.0, 0.0) {}
   ;
 }}

\newcommand{\forestDD}{
\tikz[planar forest default, planar forest, ] {
  \node [  b,  label={[label distance=-1mm]0:{\scriptsize{}}}] at (0.0, 0.0) {}
  child {   node [    b,    label={[label distance=-1mm]0:{\scriptsize{}}}  ] at (0.0, 1.0) {}
     edge from parent[    -,    solid, solid,    draw=black  ]   node [!l,right] {\scriptsize{}} }
 ;
\node [  b,  label={[label distance=-1mm]0:{\scriptsize{}}}] at (1.0, 0.0) {}
   ;
 }}

\newcommand{\forestED}{
\tikz[planar forest default, planar forest, ] {
  \node [  b,  label={[label distance=-1mm]0:{\scriptsize{}}}] at (0.0, 0.0) {}
  child {   node [    b,    label={[label distance=-1mm]0:{\scriptsize{}}}  ] at (-0.5, 1.0) {}
     edge from parent[    -,    solid, solid,    draw=black  ]   node [!l,right] {\scriptsize{}} }
child {   node [    b,    label={[label distance=-1mm]0:{\scriptsize{}}}  ] at (0.5, 1.0) {}
     edge from parent[    -,    solid, solid,    draw=black  ]   node [!l,right] {\scriptsize{}} }
 ;
 }}

\newcommand{\forestFD}{
\tikz[planar forest default, planar forest, ] {
  \node [  b,  label={[label distance=-1mm]0:{\scriptsize{}}}] at (0.0, 0.0) {}
  child {   node [    b,    label={[label distance=-1mm]0:{\scriptsize{}}}  ] at (0.0, 1.0) {}
  child {   node [    b,    label={[label distance=-1mm]0:{\scriptsize{}}}  ] at (0.0, 1.0) {}
     edge from parent[    -,    solid, solid,    draw=black  ]   node [!l,right] {\scriptsize{}} }
   edge from parent[    -,    solid, solid,    draw=black  ]   node [!l,right] {\scriptsize{}} }
 ;
 }}

\newcommand{\forestGD}{
\tikz[planar forest default, planar forest, ] {
  \node [  b,  label={[label distance=-1mm]0:{\scriptsize{}}}] at (0.0, 0.0) {}
  child {   node [    b,    label={[label distance=-1mm]0:{\scriptsize{}}}  ] at (0.0, 1.0) {}
  child {   node [    b,    label={[label distance=-1mm]0:{\scriptsize{}}}  ] at (0.0, 1.0) {}
     edge from parent[    -,    solid, solid,    draw=black  ]   node [!l,right] {\scriptsize{}} }
   edge from parent[    -,    solid, solid,    draw=black  ]   node [!l,right] {\scriptsize{}} }
 ;
 }}

\newcommand{\forestHD}{
\tikz[planar forest default, planar forest, ] {
  \node [  b,  label={[label distance=-1mm]0:{\scriptsize{}}}] at (0.0, 0.0) {}
   ;
 }}

\newcommand{\forestID}{
\tikz[planar forest default, planar forest, ] {
  \node [  b,  label={[label distance=-1mm]0:{\scriptsize{}}}] at (0.0, 0.0) {}
   ;
 }}

\newcommand{\forestJD}{
\tikz[planar forest default, planar forest, ] {
  \node [  b,  label={[label distance=-1mm]0:{\scriptsize{}}}] at (0.0, 0.0) {}
  child {   node [    .l,    label={[label distance=-1mm]0:{\scriptsize{}}}  ] at (0.0, 1.0) {$\vdots$}
  child {   node [    b,    label={[label distance=-1mm]0:{\scriptsize{}}}  ] at (0.0, 1.0) {}
     edge from parent[    -,    solid, solid,    draw=black  ]   node [!l,right] {\scriptsize{}} }
   edge from parent[    -,    solid, solid,    draw=black  ]   node [!l,right] {\scriptsize{}} }
 ;
 }}

\newcommand{\forestKD}{
\tikz[planar forest default, planar forest, ] {
  \node [  b,  label={[label distance=-1mm]0:{\scriptsize{}}}] at (0.0, 0.0) {}
  child {   node [    .l,    label={[label distance=-1mm]0:{\scriptsize{}}}  ] at (0.0, 1.0) {$\vdots$}
  child {   node [    b,    label={[label distance=-1mm]0:{\scriptsize{}}}  ] at (0.0, 1.0) {}
     edge from parent[    -,    solid, solid,    draw=black  ]   node [!l,right] {\scriptsize{}} }
   edge from parent[    -,    solid, solid,    draw=black  ]   node [!l,right] {\scriptsize{}} }
 ;
 }}

\newcommand{\forestLD}{
\tikz[planar forest default, planar forest, ] {
  \node [  b,  label={[label distance=-1mm]0:{\scriptsize{}}}] at (0.0, 0.0) {}
   ;
\node [  .l,  label={[label distance=-1mm]0:{\scriptsize{}}}] at (1.0, 0.0) {$\times$}
   ;
\node [  b,  label={[label distance=-1mm]0:{\scriptsize{}}}] at (2.0, 0.0) {}
   ;
 }}

\newcommand{\forestMD}{
\tikz[planar forest default, planar forest, ] {
  \node [  b,  label={[label distance=-1mm]0:{\scriptsize{}}}] at (0.0, 0.0) {}
  child {   node [    b,    label={[label distance=-1mm]0:{\scriptsize{}}}  ] at (0.0, 1.0) {}
     edge from parent[    -,    solid, solid,    draw=black  ]   node [!l,right] {\scriptsize{}} }
 ;
\node [  .l,  label={[label distance=-1mm]0:{\scriptsize{}}}] at (1.0, 0.0) {$\times$}
   ;
\node [  b,  label={[label distance=-1mm]0:{\scriptsize{}}}] at (2.0, 0.0) {}
  child {   node [    b,    label={[label distance=-1mm]0:{\scriptsize{}}}  ] at (0.0, 1.0) {}
     edge from parent[    -,    solid, solid,    draw=black  ]   node [!l,right] {\scriptsize{}} }
 ;
 }}

\newcommand{\forestND}{
\tikz[planar forest default, planar forest, ] {
  \node [  b,  label={[label distance=-1mm]0:{\scriptsize{}}}] at (0.0, 0.0) {}
  child {   node [    b,    label={[label distance=-1mm]0:{\scriptsize{}}}  ] at (0.0, 1.0) {}
     edge from parent[    -,    solid, solid,    draw=black  ]   node [!l,right] {\scriptsize{}} }
 ;
\node [  .l,  label={[label distance=-1mm]0:{\scriptsize{}}}] at (1.0, 0.0) {$\times$}
   ;
\node [  b,  label={[label distance=-1mm]0:{\scriptsize{}}}] at (2.0, 0.0) {}
   ;
\node [  b,  label={[label distance=-1mm]0:{\scriptsize{}}}] at (3.0, 0.0) {}
   ;
 }}

\newcommand{\forestOD}{
\tikz[planar forest default, planar forest, ] {
  \node [  b,  label={[label distance=-1mm]0:{\scriptsize{}}}] at (0.0, 0.0) {}
   ;
\node [  b,  label={[label distance=-1mm]0:{\scriptsize{}}}] at (1.0, 0.0) {}
   ;
\node [  .l,  label={[label distance=-1mm]0:{\scriptsize{}}}] at (2.0, 0.0) {$\times$}
   ;
\node [  b,  label={[label distance=-1mm]0:{\scriptsize{}}}] at (3.0, 0.0) {}
   ;
\node [  b,  label={[label distance=-1mm]0:{\scriptsize{}}}] at (4.0, 0.0) {}
   ;
 }}

\newcommand{\forestPD}{
\tikz[planar forest default, planar forest, ] {
  \node [  b,  label={[label distance=-1mm]0:{\scriptsize{}}}] at (0.0, 0.0) {}
   ;
\node [  .l,  label={[label distance=-1mm]0:{\scriptsize{}}}] at (1.0, 0.0) {$\otimes$}
   ;
\node [  b,  label={[label distance=-1mm]0:{\scriptsize{}}}] at (2.0, 0.0) {}
   ;
 }}

\newcommand{\forestQD}{
\tikz[planar forest default, planar forest, ] {
  \node [  b,  label={[label distance=-1mm]0:{\scriptsize{}}}] at (0.0, 0.0) {}
  child {   node [    b,    label={[label distance=-1mm]0:{\scriptsize{}}}  ] at (0.0, 1.0) {}
     edge from parent[    -,    solid, solid,    draw=black  ]   node [!l,right] {\scriptsize{}} }
 ;
\node [  .l,  label={[label distance=-1mm]0:{\scriptsize{}}}] at (1.0, 0.0) {$\otimes$}
   ;
\node [  b,  label={[label distance=-1mm]0:{\scriptsize{}}}] at (2.0, 0.0) {}
  child {   node [    b,    label={[label distance=-1mm]0:{\scriptsize{}}}  ] at (0.0, 1.0) {}
     edge from parent[    -,    solid, solid,    draw=black  ]   node [!l,right] {\scriptsize{}} }
 ;
 }}

\newcommand{\forestRD}{
\tikz[planar forest default, planar forest, ] {
  \node [  b,  label={[label distance=-1mm]0:{\scriptsize{}}}] at (0.0, 0.0) {}
  child {   node [    b,    label={[label distance=-1mm]0:{\scriptsize{}}}  ] at (0.0, 1.0) {}
     edge from parent[    -,    solid, solid,    draw=black  ]   node [!l,right] {\scriptsize{}} }
 ;
\node [  .l,  label={[label distance=-1mm]0:{\scriptsize{}}}] at (1.0, 0.0) {$\otimes$}
   ;
\node [  b,  label={[label distance=-1mm]0:{\scriptsize{}}}] at (2.0, 0.0) {}
   ;
\node [  b,  label={[label distance=-1mm]0:{\scriptsize{}}}] at (3.0, 0.0) {}
   ;
 }}

\newcommand{\forestSD}{
\tikz[planar forest default, planar forest, ] {
  \node [  b,  label={[label distance=-1mm]0:{\scriptsize{}}}] at (0.0, 0.0) {}
   ;
\node [  b,  label={[label distance=-1mm]0:{\scriptsize{}}}] at (1.0, 0.0) {}
   ;
\node [  .l,  label={[label distance=-1mm]0:{\scriptsize{}}}] at (2.0, 0.0) {$\otimes$}
   ;
\node [  b,  label={[label distance=-1mm]0:{\scriptsize{}}}] at (3.0, 0.0) {}
   ;
\node [  b,  label={[label distance=-1mm]0:{\scriptsize{}}}] at (4.0, 0.0) {}
   ;
 }}

\newcommand{\forestTD}{
\tikz[planar forest default, planar forest, ] {
  \node [  b,  label={[label distance=-1mm]0:{\scriptsize{}}}] at (0.0, 0.0) {}
   ;
\node [  .l,  label={[label distance=-1mm]0:{\scriptsize{}}}] at (1.0, 0.0) {$\times$}
   ;
\node [  b,  label={[label distance=-1mm]0:{\scriptsize{}}}] at (2.0, 0.0) {}
   ;
\node [  .l,  label={[label distance=-1mm]0:{\scriptsize{}}}] at (3.0, 0.0) {$\otimes$}
   ;
\node [  b,  label={[label distance=-1mm]0:{\scriptsize{}}}] at (4.0, 0.0) {}
   ;
 }}

\newcommand{\forestUD}{
\tikz[planar forest default, planar forest, ] {
  \node [  b,  label={[label distance=-1mm]0:{\scriptsize{}}}] at (0.0, 0.0) {}
   ;
\node [  .l,  label={[label distance=-1mm]0:{\scriptsize{}}}] at (1.0, 0.0) {$\times$}
   ;
\node [  b,  label={[label distance=-1mm]0:{\scriptsize{}}}] at (2.0, 0.0) {}
   ;
\node [  .l,  label={[label distance=-1mm]0:{\scriptsize{}}}] at (3.0, 0.0) {$\otimes$}
   ;
\node [  b,  label={[label distance=-1mm]0:{\scriptsize{}}}] at (4.0, 0.0) {}
   ;
\node [  .l,  label={[label distance=-1mm]0:{\scriptsize{}}}] at (5.0, 0.0) {$\times$}
   ;
\node [  b,  label={[label distance=-1mm]0:{\scriptsize{}}}] at (6.0, 0.0) {}
   ;
 }}

\newcommand{\forestVD}{
\tikz[planar forest default, planar forest, ] {
  \node [  b,  label={[label distance=-1mm]0:{\scriptsize{}}}] at (0.0, 0.0) {}
   ;
\node [  b,  label={[label distance=-1mm]0:{\scriptsize{}}}] at (1.0, 0.0) {}
   ;
\node [  .l,  label={[label distance=-1mm]0:{\scriptsize{}}}] at (2.0, 0.0) {$\times$}
   ;
\node [  b,  label={[label distance=-1mm]0:{\scriptsize{}}}] at (3.0, 0.0) {}
   ;
\node [  .l,  label={[label distance=-1mm]0:{\scriptsize{}}}] at (4.0, 0.0) {$\otimes$}
   ;
\node [  b,  label={[label distance=-1mm]0:{\scriptsize{}}}] at (5.0, 0.0) {}
   ;
 }}

\newcommand{\forestWD}{
\tikz[planar forest default, planar forest, ] {
  \node [  b,  label={[label distance=-1mm]0:{\scriptsize{}}}] at (0.0, 0.0) {}
   ;
\node [  b,  label={[label distance=-1mm]0:{\scriptsize{}}}] at (1.0, 0.0) {}
   ;
\node [  .l,  label={[label distance=-1mm]0:{\scriptsize{}}}] at (2.0, 0.0) {$\otimes$}
   ;
\node [  b,  label={[label distance=-1mm]0:{\scriptsize{}}}] at (3.0, 0.0) {}
   ;
\node [  .l,  label={[label distance=-1mm]0:{\scriptsize{}}}] at (4.0, 0.0) {$\times$}
   ;
\node [  b,  label={[label distance=-1mm]0:{\scriptsize{}}}] at (5.0, 0.0) {}
   ;
 }}

\newcommand{\forestXD}{
\tikz[planar forest default, planar forest, ] {
  \node [  b,  label={[label distance=-1mm]0:{\scriptsize{}}}] at (0.0, 0.0) {}
  child {   node [    b,    label={[label distance=-1mm]0:{\scriptsize{}}}  ] at (0.0, 1.0) {}
     edge from parent[    -,    solid, solid,    draw=black  ]   node [!l,right] {\scriptsize{}} }
 ;
\node [  .l,  label={[label distance=-1mm]0:{\scriptsize{}}}] at (1.0, 0.0) {$\otimes$}
   ;
\node [  b,  label={[label distance=-1mm]0:{\scriptsize{}}}] at (2.0, 0.0) {}
   ;
\node [  .l,  label={[label distance=-1mm]0:{\scriptsize{}}}] at (3.0, 0.0) {$\times$}
   ;
\node [  b,  label={[label distance=-1mm]0:{\scriptsize{}}}] at (4.0, 0.0) {}
   ;
 }}

\newcommand{\forestYD}{
\tikz[planar forest default, planar forest, ] {
  \node [  b,  label={[label distance=-1mm]0:{\scriptsize{}}}] at (0.0, 0.0) {}
  child {   node [    b,    label={[label distance=-1mm]0:{\scriptsize{}}}  ] at (0.0, 1.0) {}
     edge from parent[    -,    solid, solid,    draw=black  ]   node [!l,right] {\scriptsize{}} }
 ;
\node [  .l,  label={[label distance=-1mm]0:{\scriptsize{}}}] at (1.0, 0.0) {$\times$}
   ;
\node [  b,  label={[label distance=-1mm]0:{\scriptsize{}}}] at (2.0, 0.0) {}
   ;
\node [  .l,  label={[label distance=-1mm]0:{\scriptsize{}}}] at (3.0, 0.0) {$\otimes$}
   ;
\node [  b,  label={[label distance=-1mm]0:{\scriptsize{}}}] at (4.0, 0.0) {}
   ;
 }}

\newcommand{\forestAE}{
\tikz[planar forest default, planar forest, ] {
  \node [  b,  label={[label distance=-1mm]0:{\scriptsize{}}}] at (0.0, 0.0) {}
  child {   node [    b,    label={[label distance=-1mm]0:{\scriptsize{}}}  ] at (-0.5, 1.0) {}
     edge from parent[    -,    solid, solid,    draw=black  ]   node [!l,right] {\scriptsize{}} }
child {   node [    b,    label={[label distance=-1mm]0:{\scriptsize{}}}  ] at (0.5, 1.0) {}
     edge from parent[    -,    solid, solid,    draw=black  ]   node [!l,right] {\scriptsize{}} }
 ;
 }}

\newcommand{\forestBE}{
\tikz[planar forest default, planar forest, ] {
  \node [  b,  label={[label distance=-1mm]0:{\scriptsize{}}}] at (0.0, 0.0) {}
   ;
\node [  .l,  label={[label distance=-1mm]0:{\scriptsize{}}}] at (1.0, 0.0) {$\times$}
   ;
\node [  b,  label={[label distance=-1mm]0:{\scriptsize{}}}] at (2.0, 0.0) {}
   ;
 }}

\newcommand{\forestCE}{
\tikz[planar forest default, planar forest, ] {
  \node [  b,  label={[label distance=-1mm]0:{\scriptsize{}}}] at (0.0, 0.0) {}
   ;
\node [  .l,  label={[label distance=-1mm]0:{\scriptsize{}}}] at (1.0, 0.0) {$\otimes$}
   ;
\node [  b,  label={[label distance=-1mm]0:{\scriptsize{}}}] at (2.0, 0.0) {}
   ;
 }}

\newcommand{\forestDE}{
\tikz[planar forest default, planar forest, ] {
  \node [  b,  label={[label distance=-1mm]0:{\scriptsize{}}}] at (0.0, 0.0) {}
  child {   node [    b,    label={[label distance=-1mm]0:{\scriptsize{}}}  ] at (-1.0, 1.0) {}
     edge from parent[    -,    solid, solid,    draw=black  ]   node [!l,right] {\scriptsize{}} }
child {   node [    b,    label={[label distance=-1mm]0:{\scriptsize{}}}  ] at (0.0, 1.0) {}
     edge from parent[    -,    solid, solid,    draw=black  ]   node [!l,right] {\scriptsize{}} }
child {   node [    b,    label={[label distance=-1mm]0:{\scriptsize{}}}  ] at (1.0, 1.0) {}
     edge from parent[    -,    solid, solid,    draw=black  ]   node [!l,right] {\scriptsize{}} }
 ;
 }}

\newcommand{\forestEE}{
\tikz[planar forest default, planar forest, ] {
  \node [  b,  label={[label distance=-1mm]0:{\scriptsize{}}}] at (0.0, 0.0) {}
   ;
\node [  .l,  label={[label distance=-1mm]0:{\scriptsize{}}}] at (1.0, 0.0) {$\times$}
   ;
\node [  b,  label={[label distance=-1mm]0:{\scriptsize{}}}] at (2.0, 0.0) {}
   ;
\node [  .l,  label={[label distance=-1mm]0:{\scriptsize{}}}] at (3.0, 0.0) {$\otimes$}
   ;
\node [  b,  label={[label distance=-1mm]0:{\scriptsize{}}}] at (4.0, 0.0) {}
   ;
 }}

\newcommand{\forestFE}{
\tikz[planar forest default, planar forest, ] {
  \node [  b,  label={[label distance=-1mm]0:{\scriptsize{}}}] at (0.0, 0.0) {}
  child {   node [    b,    label={[label distance=-1mm]0:{\scriptsize{}}}  ] at (-0.5, 1.0) {}
  child {   node [    b,    label={[label distance=-1mm]0:{\scriptsize{}}}  ] at (0.0, 1.0) {}
     edge from parent[    -,    solid, solid,    draw=black  ]   node [!l,right] {\scriptsize{}} }
   edge from parent[    -,    solid, solid,    draw=black  ]   node [!l,right] {\scriptsize{}} }
child {   node [    b,    label={[label distance=-1mm]0:{\scriptsize{}}}  ] at (0.5, 1.0) {}
  child {   node [    b,    label={[label distance=-1mm]0:{\scriptsize{}}}  ] at (0.0, 1.0) {}
     edge from parent[    -,    solid, solid,    draw=black  ]   node [!l,right] {\scriptsize{}} }
   edge from parent[    -,    solid, solid,    draw=black  ]   node [!l,right] {\scriptsize{}} }
 ;
 }}

\newcommand{\forestGE}{
\tikz[planar forest default, planar forest, ] {
  \node [  b,  label={[label distance=-1mm]0:{\scriptsize{}}}] at (0.0, 0.0) {}
  child {   node [    b,    label={[label distance=-1mm]0:{\scriptsize{}}}  ] at (0.0, 1.0) {}
     edge from parent[    -,    solid, solid,    draw=black  ]   node [!l,right] {\scriptsize{}} }
 ;
\node [  .l,  label={[label distance=-1mm]0:{\scriptsize{}}}] at (1.0, 0.0) {$\times$}
   ;
\node [  b,  label={[label distance=-1mm]0:{\scriptsize{}}}] at (2.0, 0.0) {}
  child {   node [    b,    label={[label distance=-1mm]0:{\scriptsize{}}}  ] at (0.0, 1.0) {}
     edge from parent[    -,    solid, solid,    draw=black  ]   node [!l,right] {\scriptsize{}} }
 ;
 }}

\newcommand{\forestHE}{
\tikz[planar forest default, planar forest, ] {
  \node [  b,  label={[label distance=-1mm]0:{\scriptsize{}}}] at (0.0, 0.0) {}
  child {   node [    b,    label={[label distance=-1mm]0:{\scriptsize{}}}  ] at (0.0, 1.0) {}
     edge from parent[    -,    solid, solid,    draw=black  ]   node [!l,right] {\scriptsize{}} }
 ;
\node [  .l,  label={[label distance=-1mm]0:{\scriptsize{}}}] at (1.0, 0.0) {$\times$}
   ;
\node [  b,  label={[label distance=-1mm]0:{\scriptsize{}}}] at (2.0, 0.0) {}
   ;
\node [  b,  label={[label distance=-1mm]0:{\scriptsize{}}}] at (3.0, 0.0) {}
   ;
 }}

\newcommand{\forestIE}{
\tikz[planar forest default, planar forest, ] {
  \node [  b,  label={[label distance=-1mm]0:{\scriptsize{}}}] at (0.0, 0.0) {}
   ;
\node [  b,  label={[label distance=-1mm]0:{\scriptsize{}}}] at (1.0, 0.0) {}
   ;
\node [  .l,  label={[label distance=-1mm]0:{\scriptsize{}}}] at (2.0, 0.0) {$\times$}
   ;
\node [  b,  label={[label distance=-1mm]0:{\scriptsize{}}}] at (3.0, 0.0) {}
   ;
\node [  b,  label={[label distance=-1mm]0:{\scriptsize{}}}] at (4.0, 0.0) {}
   ;
 }}

\newcommand{\forestJE}{
\tikz[planar forest default, planar forest, ] {
  \node [  b,  label={[label distance=-1mm]0:{\scriptsize{}}}] at (0.0, 0.0) {}
  child {   node [    b,    label={[label distance=-1mm]0:{\scriptsize{}}}  ] at (0.0, 1.0) {}
     edge from parent[    -,    solid, solid,    draw=black  ]   node [!l,right] {\scriptsize{}} }
 ;
\node [  .l,  label={[label distance=-1mm]0:{\scriptsize{}}}] at (1.0, 0.0) {$\otimes$}
   ;
\node [  b,  label={[label distance=-1mm]0:{\scriptsize{}}}] at (2.0, 0.0) {}
  child {   node [    b,    label={[label distance=-1mm]0:{\scriptsize{}}}  ] at (0.0, 1.0) {}
     edge from parent[    -,    solid, solid,    draw=black  ]   node [!l,right] {\scriptsize{}} }
 ;
 }}

\newcommand{\forestKE}{
\tikz[planar forest default, planar forest, ] {
  \node [  b,  label={[label distance=-1mm]0:{\scriptsize{}}}] at (0.0, 0.0) {}
  child {   node [    b,    label={[label distance=-1mm]0:{\scriptsize{}}}  ] at (0.0, 1.0) {}
     edge from parent[    -,    solid, solid,    draw=black  ]   node [!l,right] {\scriptsize{}} }
 ;
\node [  .l,  label={[label distance=-1mm]0:{\scriptsize{}}}] at (1.0, 0.0) {$\otimes$}
   ;
\node [  b,  label={[label distance=-1mm]0:{\scriptsize{}}}] at (2.0, 0.0) {}
   ;
\node [  b,  label={[label distance=-1mm]0:{\scriptsize{}}}] at (3.0, 0.0) {}
   ;
 }}

\newcommand{\forestLE}{
\tikz[planar forest default, planar forest, ] {
  \node [  b,  label={[label distance=-1mm]0:{\scriptsize{}}}] at (0.0, 0.0) {}
   ;
\node [  b,  label={[label distance=-1mm]0:{\scriptsize{}}}] at (1.0, 0.0) {}
   ;
\node [  .l,  label={[label distance=-1mm]0:{\scriptsize{}}}] at (2.0, 0.0) {$\otimes$}
   ;
\node [  b,  label={[label distance=-1mm]0:{\scriptsize{}}}] at (3.0, 0.0) {}
   ;
\node [  b,  label={[label distance=-1mm]0:{\scriptsize{}}}] at (4.0, 0.0) {}
   ;
 }}

\newcommand{\forestME}{
\tikz[planar forest default, planar forest, ] {
  \node [  b,  label={[label distance=-1mm]0:{\scriptsize{}}}] at (0.0, 0.0) {}
  child {   node [    b,    label={[label distance=-1mm]0:{\scriptsize{}}}  ] at (-1.0, 1.0) {}
     edge from parent[    -,    solid, solid,    draw=black  ]   node [!l,right] {\scriptsize{}} }
child {   node [    b,    label={[label distance=-1mm]0:{\scriptsize{}}}  ] at (0.0, 1.0) {}
  child {   node [    b,    label={[label distance=-1mm]0:{\scriptsize{}}}  ] at (0.0, 1.0) {}
     edge from parent[    -,    solid, solid,    draw=black  ]   node [!l,right] {\scriptsize{}} }
   edge from parent[    -,    solid, solid,    draw=black  ]   node [!l,right] {\scriptsize{}} }
child {   node [    b,    label={[label distance=-1mm]0:{\scriptsize{}}}  ] at (1.0, 1.0) {}
     edge from parent[    -,    solid, solid,    draw=black  ]   node [!l,right] {\scriptsize{}} }
 ;
 }}

\newcommand{\forestNE}{
\tikz[planar forest default, planar forest, ] {
  \node [  b,  label={[label distance=-1mm]0:{\scriptsize{}}}] at (0.0, 0.0) {}
  child {   node [    b,    label={[label distance=-1mm]0:{\scriptsize{}}}  ] at (0.0, 1.0) {}
     edge from parent[    -,    solid, solid,    draw=black  ]   node [!l,right] {\scriptsize{}} }
 ;
\node [  .l,  label={[label distance=-1mm]0:{\scriptsize{}}}] at (1.0, 0.0) {$\times$}
   ;
\node [  b,  label={[label distance=-1mm]0:{\scriptsize{}}}] at (2.0, 0.0) {}
   ;
\node [  .l,  label={[label distance=-1mm]0:{\scriptsize{}}}] at (3.0, 0.0) {$\otimes$}
   ;
\node [  b,  label={[label distance=-1mm]0:{\scriptsize{}}}] at (4.0, 0.0) {}
   ;
 }}

\newcommand{\forestOE}{
\tikz[planar forest default, planar forest, ] {
  \node [  b,  label={[label distance=-1mm]0:{\scriptsize{}}}] at (0.0, 0.0) {}
   ;
\node [  b,  label={[label distance=-1mm]0:{\scriptsize{}}}] at (1.0, 0.0) {}
   ;
\node [  .l,  label={[label distance=-1mm]0:{\scriptsize{}}}] at (2.0, 0.0) {$\times$}
   ;
\node [  b,  label={[label distance=-1mm]0:{\scriptsize{}}}] at (3.0, 0.0) {}
   ;
\node [  .l,  label={[label distance=-1mm]0:{\scriptsize{}}}] at (4.0, 0.0) {$\otimes$}
   ;
\node [  b,  label={[label distance=-1mm]0:{\scriptsize{}}}] at (5.0, 0.0) {}
   ;
 }}

\newcommand{\forestPE}{
\tikz[planar forest default, planar forest, ] {
  \node [  b,  label={[label distance=-1mm]0:{\scriptsize{}}}] at (0.0, 0.0) {}
  child {   node [    b,    label={[label distance=-1mm]0:{\scriptsize{}}}  ] at (0.0, 1.0) {}
     edge from parent[    -,    solid, solid,    draw=black  ]   node [!l,right] {\scriptsize{}} }
 ;
\node [  .l,  label={[label distance=-1mm]0:{\scriptsize{}}}] at (1.0, 0.0) {$\otimes$}
   ;
\node [  b,  label={[label distance=-1mm]0:{\scriptsize{}}}] at (2.0, 0.0) {}
   ;
\node [  .l,  label={[label distance=-1mm]0:{\scriptsize{}}}] at (3.0, 0.0) {$\times$}
   ;
\node [  b,  label={[label distance=-1mm]0:{\scriptsize{}}}] at (4.0, 0.0) {}
   ;
 }}

\newcommand{\forestQE}{
\tikz[planar forest default, planar forest, ] {
  \node [  b,  label={[label distance=-1mm]0:{\scriptsize{}}}] at (0.0, 0.0) {}
   ;
\node [  b,  label={[label distance=-1mm]0:{\scriptsize{}}}] at (1.0, 0.0) {}
   ;
\node [  .l,  label={[label distance=-1mm]0:{\scriptsize{}}}] at (2.0, 0.0) {$\otimes$}
   ;
\node [  b,  label={[label distance=-1mm]0:{\scriptsize{}}}] at (3.0, 0.0) {}
   ;
\node [  .l,  label={[label distance=-1mm]0:{\scriptsize{}}}] at (4.0, 0.0) {$\times$}
   ;
\node [  b,  label={[label distance=-1mm]0:{\scriptsize{}}}] at (5.0, 0.0) {}
   ;
 }}

\newcommand{\forestRE}{
\tikz[planar forest default, planar forest, ] {
  \node [  b,  label={[label distance=-1mm]0:{\scriptsize{}}}] at (0.0, 0.0) {}
  child {   node [    b,    label={[label distance=-1mm]0:{\scriptsize{}}}  ] at (-1.5, 1.0) {}
     edge from parent[    -,    solid, solid,    draw=black  ]   node [!l,right] {\scriptsize{}} }
child {   node [    b,    label={[label distance=-1mm]0:{\scriptsize{}}}  ] at (-0.5, 1.0) {}
     edge from parent[    -,    solid, solid,    draw=black  ]   node [!l,right] {\scriptsize{}} }
child {   node [    b,    label={[label distance=-1mm]0:{\scriptsize{}}}  ] at (0.5, 1.0) {}
     edge from parent[    -,    solid, solid,    draw=black  ]   node [!l,right] {\scriptsize{}} }
child {   node [    b,    label={[label distance=-1mm]0:{\scriptsize{}}}  ] at (1.5, 1.0) {}
     edge from parent[    -,    solid, solid,    draw=black  ]   node [!l,right] {\scriptsize{}} }
 ;
 }}

\newcommand{\forestSE}{
\tikz[planar forest default, planar forest, ] {
  \node [  b,  label={[label distance=-1mm]0:{\scriptsize{}}}] at (0.0, 0.0) {}
   ;
\node [  .l,  label={[label distance=-1mm]0:{\scriptsize{}}}] at (1.0, 0.0) {$\times$}
   ;
\node [  b,  label={[label distance=-1mm]0:{\scriptsize{}}}] at (2.0, 0.0) {}
   ;
\node [  .l,  label={[label distance=-1mm]0:{\scriptsize{}}}] at (3.0, 0.0) {$\otimes$}
   ;
\node [  b,  label={[label distance=-1mm]0:{\scriptsize{}}}] at (4.0, 0.0) {}
   ;
\node [  .l,  label={[label distance=-1mm]0:{\scriptsize{}}}] at (5.0, 0.0) {$\times$}
   ;
\node [  b,  label={[label distance=-1mm]0:{\scriptsize{}}}] at (6.0, 0.0) {}
   ;
 }}

\tikzstyle planar forest=[scale=2]
\begin{pycode}
pf.add_macro('x', '.$\\times$-.')
pf.add_macro('o', '.$\\otimes$-.')
pf.add_macro('v', '.$\\vdots$')
\end{pycode}

\numberwithin{equation}{section}

\newcommand{\N}{\ensuremath{\mathbb{N}}\xspace}
\newcommand{\R}{\ensuremath{\mathbb{R}}\xspace}

\newcommand{\FF}{\ensuremath{\mathcal{F}}\xspace}
\newcommand{\HH}{\ensuremath{\mathcal{H}}\xspace}
\newcommand{\OO}{\ensuremath{\mathcal{O}}\xspace}
\newcommand{\TT}{\ensuremath{\mathcal{T}}\xspace}
\newcommand{\UU}{\ensuremath{\mathcal{U}}\xspace}

\DeclareMathOperator{\Trace}{Tr}

\DeclareMathOperator{\id}{id}
\DeclareMathOperator{\dexp}{dexp}
\DeclareMathOperator{\Prim}{Prim}

\newcommand\restrict[1]{\raisebox{-.5ex}{$|$}_{#1}}
\newcommand{\butcher}{\to}

\newcommand{\frakgl}{\mathfrak{gl}}
\newcommand{\fraku}{\mathfrak{u}}

\newcommand{\one}{\mathbf{1}}
\newcommand{\dF}{\mathbb{F}}

\newcommand{\BT}{\mathcal{BT}}
\newcommand{\BF}{\mathcal{BF}}
\newcommand{\blank}{{-}}

\newcommand{\graft}{\curvearrowright}

\DeclareMathOperator{\Ad}{Ad}
\DeclareMathOperator{\ad}{ad}

\allowdisplaybreaks

\newtheorem{remark}[theorem]{Remark}

\newtheorem{example}[theorem]{Example}
\newtheorem*{maintheorem}{Main Theorem}

\newcommand{\note}[1]{\textcolor{blue}{#1}\index {#1}}

\title{Backward error analysis for matrix discretizations \\ of 2-D~Euler equations
\thanks{Submitted to the editors \today
\funding{This work was supported by the Swedish Research Council (grant number 2022-03453), the Knut and Alice Wallenberg Foundation (grant numbers KAW~2024.0229 and KAW~2023.0433), and the Göran Gustafsson Foundation for Research in Natural Sciences and Medicine.
The computations were enabled by resources provided by Chalmers e-Commons at Chalmers.}
}}

\author{
    Eugen Bronasco\thanks{Department of Mathematical Sciences, Chalmers University of Technology and University of Gothenburg, Sweden (\email{bronasco@chalmers.se}, \email{klas.modin@chalmers.se})}
    \and Klas Modin\footnotemark[2]
}

\headers{Backward error analysis for matrix discretizations}{E. Bronasco and K. Modin}

\begin{document}

\maketitle

\begin{abstract}
    We introduce a formalism of Lie--Poisson reduction of Butcher series.
    The corresponding forest momentum map allows
    for describing backward error analysis of isospectral symplectic Runge--Kutta methods applied to Zeitlin's matrix discretization of the 2-D Euler equations on the sphere.
    Based thereon, we obtain exponentially small error bounds for the conservation of modified Hamiltonians, valid for exponentially long time intervals.
    Crucially, the error bounds and the length of the time intervals are independent of the spatial discretization parameter $n$ (the matrix size) when the time step for different $n$ is scaled as $h = \mathcal{O}(n^{-1})$.
    Our results thus extend the classical backward error analysis result for finite-dimensional Hamiltonian systems to the infinite-dimensional case of the 2-D~Euler equations discretized via matrix hydrodynamics.
\end{abstract}

\begin{keywords}
    matrix hydrodynamics, backward error analysis, Butcher series, biplanar forests, Zeitlin's model, 2-D Euler,  Hamiltonian PDEs, Lie--Poisson reduction, isospectral flows, symplectic Runge--Kutta methods
\end{keywords}

\begin{MSCcodes}
    65P10, 35Q31, 37M15, 53D50, 65M99
\end{MSCcodes}

\section{Introduction}

The success of symplectic numerical integration schemes for long-term simulations of Hamiltonian systems can conceptually be explained as follows.
Given a symplectic phase space $M$ and a Hamiltonian system on it
\begin{equation}\label{eq:ham_sys}
    \dot y = X_H(y) \qquad \big(\text{if $M=\R^{2d}$ then $X_H =J^{-1}\nabla H$ for $J = \begin{bmatrix}0 & -I \\ I & 0 \end{bmatrix}$}\big),
\end{equation}
let $\Phi_h\colon M\to M$ denote a numerical integration method for the system~\cref{eq:ham_sys}.
If we try to reinterpret the discrete integrator map $\Phi_h$ as the exact flow of a \emph{modified vector field} $\tilde X_h$ (this is the basic idea of backward error analysis), then, if $\Phi_h$ is a symplectic map, the vector field $\tilde X_h$ must be symplectic.
In particular, it (locally) corresponds to a \emph{modified Hamiltonian function} $\tilde H_h$ on $M$, namely $\tilde X = X_{\tilde H_h}$.
This means that the discrete trajectory traced out by $y_{k+1}= \Phi_h(y_k)$ corresponds to the exact flow of a true Hamiltonian system.
All properties shared by solutions to Hamiltonian systems are thereby shared by the numerical trajectory.
More specifically, if the method is consistent and the step size is small enough, all properties shared by Hamiltonian systems \emph{nearby} the system~\cref{eq:ham_sys} are shared by the numerical trajectory.
For example, if the system~\cref{eq:ham_sys} is integrable in the Arnold--Liouville sense,  KAM-theory tells us that nearby Hamiltonian system are ``almost'' integrable (cf.~\cite{Ar1989}), leading to near conservation of the first integrals.
As another and more generic example, as long as the numerical trajectory remains on a compact subset of $M$, we have that $H(y_k) - \tilde H_h(y_k) = \mathcal{O}(h^p)$ where $p$ is the order of the method.
Consequently, since $\tilde H_h$ is exactly preserved by the modified vector field, the original Hamiltonian $H$ is nearly conserved.

This conceptual story of backward error analysis for symplectic integrators is pleasing and the consequences drawn from it are essentially correct for a finite-dimensional phase space $M$.
But the story is not quite right; the devil is in the details.
The core issue is a deep result in the analysis of infinite-dimensional Lie groups.
The set of diffeomorphisms on $M$ form a Lie group (for example in the Fréchet topology of smooth functions, cf.~\cite{Ha1982}).
Its Lie algebra consists of vector fields on $M$.
The corresponding exponential map is obtained by integrating the vector field to a diffeomorphism at time one. 
And herein lies the core issue: this exponential map is never locally surjective (contrary to finite-dimensional Lie groups where it is always locally surjective).
Thus, there are diffeomorphisms arbitrarily close to the identity map that cannot be generated by time one integration of a vector field.
The same is true for the subgroup of symplectic diffeomorphisms and its Lie algebra of symplectic vector fields.
Thus, the very outset of backward error analysis, that there exists a modified vector field, is invalid in general.
As a result, modified Hamiltonians do not exist in general.
Instead, the rigorous theory of backward error analysis looks for an ``almost'' modified Hamiltonian, obtained as an optimally truncated Taylor expansion $$\tilde H_h = H + h H_1 + \cdots + h^{N-1}H_N.$$
This theory, given by Benettin and Giorgilli~\cite{BeGi1994}, is technical.
In particular, it requires a real analytic setting and, because of the truncation, the results are valid only for exponentially long time-intervals $t \leq e^{c/h}$, where the constant $c > 0$ depends on the vector field $X_H$, the numerical method, and the compact domain of the numerical trajectory, but it is independent of the step size $h$.
In practise, the exponentially long time-interval means that other numerical artifacts, such as drift due to round-off errors, will dominate over the truncation error in the modified Hamiltonian. 
The conceptual backward error analysis story, as told in the first paragraph, is thereby valid, in essence, for finite-dimensional Hamiltonian systems.

At first sight, the situation for infinite-dimensional Hamiltonian PDEs seems similar: if the spatial discretization gives a finite-dimensional Hamiltonian system we can use a symplectic integration scheme to obtain near conservation of energy.
But this approach fails, because all error constants in the estimates depend on the spatial discretization parameter $n$ in an uncontrolled way as $n\to\infty$.
Furthermore, the framework depends on qualitative arguments that automatically fail as $n\to\infty$, for example that ``the numerical trajectory remains on a compact subset of~$M$''.
These hurdles for rigorous backward error analysis of Hamiltonian PDEs have persisted since the birth of symplectic numerical integration.
So far, the only rigorous results are for semi-linear PDEs, such as the nonlinear Schrödinger equation, and are based on splitting~\cite{FaGr2011,BaFaGr2013,FaMaSc2025arxiv}.

In this paper, we give rigorous results on backward error analysis and near conservation of energy for full discretizations of the two-dimensional (2-D) Euler equations---a fully nonlinear Hamiltonian PDE.
The spatial discretization is based on \emph{matrix hydrodynamics} (cf.~\cite{ModinTFM24}), leading to an isospectral flow of matrices which we discretize in time using \emph{isospectral symplectic Runge--Kutta (ISOSYRK) methods}~\cite{ModinLPM20}.
The core theoretical tool behind our results is a new theory of Butcher series over biplanar forests, applicable to isospectral flows.
The theory is obtained by lifting discrete Lie--Poisson reduction theory (cf.~\cite{ModinLPM20}) from individual vector fields and Hamiltonian functions to the level of Butcher series.
Interestingly, the resulting series for ISOSYRK methods, via biplanar forests, are distinct from Butcher series for Runge--Kutta--Munthe-Kaas (RKMK) methods.
Although the focus in this paper is on the 2-D Euler equations, the approach works also for other Hamiltonian PDEs admitting matrix discretizations as isospectral flows.

We continue the introduction with a brief recollection of the Hamiltonian (Lie--Poisson) structure of the 2-D Euler equations, spatial discretizations via matrix hydrodynamics, the ISOSYRK methods, and a statement of the main result.
Thereafter, we continue in \cref{sec:butcher_bea} with a review of Butcher series and modified vector fields for symplectic Runge--Kutta methods for canonical Hamiltonian systems.
This theory is the basis for Butcher series and backward error analysis over biplanar forests developed in \cref{sec:butcher_biplanar} and \cref{sec:bea_biplanar_forests}, providing a proof of the main theorem.
Finally, in \cref{sec:numerics}, we give numerical verifications of the theoretical results.

\subsection{The 2-D Euler equations and matrix hydrodynamics}

For the incompressible Euler equations on the sphere $\mathbb{S}^2$,
if $v = v(x,t)$ is the velocity field for the fluid and $\omega = \operatorname{curl}v$ is the corresponding vorticity function, the equations are
\begin{equation}\label{eq:euler2D}
    \dot\omega + \{\omega, \psi \} = 0, \quad -\Delta\psi = \omega,\quad \omega(\cdot,0) = \omega_0\in C^\infty(\mathbb{S}^2),
\end{equation}
where $\dot\omega = \frac{\partial\omega}{\partial t}$ is the time derivative, $\{\cdot,\cdot\}$ is the Poisson bracket on $\mathbb{S}^2$, $\Delta$ is the Laplace-Beltrami operator, and $\psi$ is the stream function.
Euler's equations~\eqref{eq:euler2D} constitute a \emph{Lie--Poisson system} for the infinite-dimensional Lie algebra of divergence free vector fields on $\mathbb{S}^2$ (see e.g.~\cite[sect.~2]{ModinTFM24}).
The Hamiltonian for the system is
\begin{equation*}
    H(\omega) = \frac{1}{2}\int_{S^2} \omega (-\Delta)^{-1}\omega \, \mathrm{d}x .
\end{equation*}

The core idea of \emph{matrix hydrodynamics}, initiated by Zeitlin~\cite{Ze1991, Ze2004}, is to obtain spatial discretizations for Lie--Poisson PDEs via quantization theory: replace the infinite-dimensional Poisson algebra of smooth functions $(C^\infty(\mathbb{S}^2), \{ \cdot,\cdot\})$ by the finite-dimensional Lie algebra of skew-Hermitian matrices $\mathfrak{u}(n)$.
This can be achieved via the \emph{Berezin--Toeplitz quantization operator} $\mathcal T_n\colon C^\infty(\mathbb{S}^2) \to \fraku(n)$ (cf.~\cite{Le2018}).
For the Euler equations~\eqref{eq:euler2D}, it leads to the \emph{Euler--Zeitlin equations}, namely the isospectral matrix flow
\begin{equation}\label{eq:euler-zeitlin}
    \dot W + \frac{1}{\hbar_n}[W, S] = 0, \quad -\Delta_n S = W, \quad W(0) = W_0\in \fraku(n),
\end{equation}
where $W=W(t)$ and $S = S(t)$ are skew-Hermitian matrices, $\hbar_n = 2/\sqrt{n^2-1}$ is a scaling constant, $[\cdot,\cdot]$ is the commutator, and $\Delta_n$ is the Hoppe--Yau operator (see \cref{sec:hoppe_yau_estimates}).
The system~\eqref{eq:euler-zeitlin} constitute itself a Lie--Poisson system for the Lie algebra $\mathfrak{u}(n)$ with the scaled bracket $\frac{1}{\hbar_n}[\cdot,\cdot]$ and the Hamiltonian
\begin{equation}\label{eq:zeitlin_hamiltonian}
    H_n(W) = \frac{2\pi}{n}\Trace(W^\dagger (-\Delta_n)^{-1} W).
\end{equation}
For $\omega\in C^\infty(S^2)$ we have $H_n\big(\mathcal T_n(\omega)\big) \to  H(\omega)$ as $n\to \infty$.
That the flow \eqref{eq:euler-zeitlin} is isospectral means geometrically that it preserves the \emph{coadjoint orbits} (or symplectic leaves) of the Lie--Poisson structure. 
For a full account of matrix hydrodynamics, including convergence results as $n\to\infty$, we refer to the work of Modin and Viviani~\cite{ModinTFM24} and references therein.

\subsection{Isospectral symplectic Runge--Kutta (ISOSYRK) methods}

Consider iso\-spec\-tral flows of the form
\begin{equation}
    \label{eq:isospectral}
    \dot{W} = [f(W), W] \,, \quad W(0) = W_0\in \fraku(n),
\end{equation}
for a matrix $W = W(t)$ and a \emph{linear} function $f\colon \fraku(n)\to\fraku(n)$.\footnote{For simplicity, we work here with the Lie algebra $\fraku(n)$, but the ISOSYRK methods work more generally for a subset $V\subset \frakgl(n,\mathbb{C})$ and $f\colon V \to \{ \xi \in V \, | \, [\xi, V] \subset V \}$ (cf.~\cite{ModinLPM20}).}
When $f(W) = \frac{\delta H}{\delta W}$ for some function~$H$, the flow \eqref{eq:isospectral} constitutes a Hamiltonian system on~$\fraku(n)^*\simeq \fraku(n)$ for the Lie--Poisson structure
\begin{equation*}
    \prec F,G\succ(W) = \Big\langle W, \Big[\frac{\delta F}{\delta W}, \frac{\delta G}{\delta W}\Big]\Big\rangle, \quad F,G\in C^\infty(\fraku(n)^*).
\end{equation*}
The Euler--Zeitlin equations~\eqref{eq:euler-zeitlin} corresponds to the case $f(W) = \frac{\delta H}{\delta W} = -\frac{1}{\hbar}\Delta_n^{-1}W$ for the quadratic Hamiltonian~\eqref{eq:zeitlin_hamiltonian} scaled by $n/(2\pi \hbar)$. 


Given a Butcher tableau $A = (a_{ij})_{i,j=1}^s$ and $b = (b_i)_{i=1}^s$ of a symplectic Runge--Kutta method, 
the corresponding ISOSYRK method $W_k \mapsto W_{k+1}$ for the flow~\eqref{eq:isospectral} is defined by
\begin{equation}\label{eq:isosyrk}
\begin{aligned}
    W_{k+1} &= W_k + h \sum_{i=1}^s b_i [f(\tilde{W}_i), \tilde{W}_i] \,, 
    &\tilde{W}_i &= W_k + h \sum_{j=1}^s a_{ij} (X_j + Y_j + K_{ij}) \,, \\
    X_i &= - (W_n + h \sum_{j=1}^s a_{ij} X_j) f(\tilde{W}_i) \,, 
    &Y_i &= f(\tilde{W}_i) (W_k + h \sum_{j=1}^s a_{ij} Y_j) \,, \\
    K_{ij} &= h \sum_{k=1}^s f(\tilde{W}_i) (a_{ik} X_k + a_{jk} K_{ik}) \,, & &
\end{aligned}    
\end{equation}
where $h>0$ is the timestep and $\tilde W_{i}, X_i, Y_i, K_{ij}$ are intermediate variables.

\begin{theorem}[\cite{ModinLPM20}]
    The ISOSYRK method~\eqref{eq:isosyrk} is isospectral.
    Furthermore, if $f= -\frac{\delta H}{\delta W}$ it preserves the Lie--Poisson structure: the map $W_k\to W_{k+1}$ restricted to a coadjoint orbit (i.e., an isospectral surface) is symplectic.
\end{theorem}

The proof of this results uses Lie--Poisson reduction theory.
A brief summary goes as follows. Via the momentum map 
\begin{equation*}
    \mu\colon T^*\mathrm{U}(n) \to \fraku(n)^*, \quad \mu(Q,P) = Q^\dagger P ,  
\end{equation*}
the Hamiltonian $H$ on $\fraku(n)^*$ is lifted to a Hamiltonian $\mathcal H = H\circ\mu$ in the variables $(Q,P)$.
In these variables, the Lie--Poisson system \eqref{eq:isospectral} becomes a canonical Hamiltonian system on $T^*\mathrm{U}(n)$
\begin{equation}\label{eq:canonical_equations}
    \dot Q = \frac{\delta \mathcal H}{\delta P} = Q f(Q^\dagger P)^\dagger ,\quad \dot P = -\frac{\delta \mathcal H}{\delta Q} = -P f(Q^\dagger P), \quad Q(0) = Q_0, \quad P(0) = P_0.
\end{equation}
Its solutions $(Q(t),P(t))$ thus map to solutions $W(t) = \mu(Q(t),P(t))$ of equation~\eqref{eq:isospectral}.
Similarly, via the momentum map the ISOSYRK method~\eqref{eq:isosyrk} corresponds to the underlying symplectic Runge--Kutta method applied to the canonical Hamiltonian system~\eqref{eq:canonical_equations}.

\begin{example}[ISOMP method]\label{example:isomp}
    The simplest ISOSYRK method is the isospectral midpoint method (ISOMP), defined by
    \begin{equation}\label{eq:isomp}
    \begin{aligned}
        W_{k} &= \left(I - \frac{h}{2}f(\tilde W) \right)\tilde W \left(I +  \frac{h}{2}f(\tilde W)\right) \\
        W_{k+1} &= W_k + h[f(\tilde W), \tilde W]
    \end{aligned}        
    \end{equation}
    Here, the underlying Runge--Kutta scheme is the implicit midpoint method.
\end{example}

\subsection{Main result}
Given that ISOSYRK methods are symplectic, it is natural to ask if the energy is nearly conserved for long times $t_k\gg 1$ while at the same time $n\to\infty$.
As explained in the introduction, this is a question about the modified Hamiltonian and backward error analysis.
Our main result is that arbitrary truncations of the modified Hamiltonian for the Euler--Zeitlin equations are well-behaved as $n\to\infty$ as long as the time step is scaled as $h = \mathcal{O}(\epsilon/n)$.

\begin{maintheorem}
    For $\omega_0\in C^\infty(S^2)$, consider an ISOSYRK method of order $p$ applied to the Euler--Zeitlin equations~\eqref{eq:euler-zeitlin} with initial data $W_0 = \mathcal T_n(\omega_0)$ and time step $h= \epsilon \hbar_n = 2 \epsilon/\sqrt{n^2-1}$.
    Then there exists a truncated modified Hamiltonian $\tilde H_{n,h}$ such that $\tilde H_{n,h}(W) = H_n(W) + \mathcal O(h^p)$ and 
    \begin{equation*}
        \left| \tilde H_{n,h}(W_k) - \tilde H_{n,h}(W_0) \right| \leq 2\pi A(\lVert \omega_0 \rVert_{L^\infty})\exp(-\epsilon_0/(2\epsilon \lVert \omega_0 \rVert_{L^\infty}))
    \end{equation*}
    whenever $t_k = hk \leq \exp(\epsilon_0/(2\epsilon \lVert \omega_0 \rVert_\infty))$.
    Here, $A(r) = 2e (1 + 6eC_a) (1+3e C_a + 18 C_a r)r^2$ and $\epsilon_0 = 1/(18e C_a)$, where $C_a>0$ only depends on the Butcher tableau of the ISOSYRK method.\footnote{$C_a = 1$ for the ISOMP method~\eqref{eq:isomp}.}
    In particular, the error estimate and the length of the time interval are independent of $n$.
\end{maintheorem}
\begin{proof}
    Zeitlin's model is of the form of equation~\eqref{eq:isospectral} for $f(W) = \frac{1}{\hbar_n}(-\Delta_n)^{-1}W$.
    Let $C_f$ be the operator norm of $f$ relative to the spectral norm $\lVert\cdot\rVert_\infty$ (cf.~\cref{sec:long_time_behavior}).
    From Theorem~\ref{thm:hoppe_yau_estimates} in the appendix it follows that $C_f = 1/\hbar_n$.
    The Hamiltonian corresponding to $f$ is $\mathcal H_n(W) = \frac{1}{2\hbar_n}\operatorname{tr}(W(-\Delta_n)^{-1}W)$.
    Let $\tilde{\mathcal H}_{n,h}$ denote the corresponding optimally truncated modified Hamiltonian, in accordance with Theorem~\ref{thm:bound_modflow}.
    From Theorem~\ref{thm:long_time_Hamiltonian_conservation} it follows that
    \begin{equation*}
        \lvert \tilde{\mathcal H}_{n,h}(W_k) - \tilde{\mathcal H}_{n,h}(W_0) \rvert \leq \frac{1}{\hbar_n^2}A(\lVert W_0\rVert_\infty) \exp(-h_0/(2h\lVert W_0\rVert_\infty)) 
    \end{equation*}
    for $kh \leq \exp(h_0/(2h\lVert W_0\rVert_\infty))$ and $A(r) = L r^2 C(r)$.
    We notice that $h_0/h = 1/(18e C_a C_f \epsilon \hbar_n) = 1/(18e C_a \epsilon) = \epsilon_0/\epsilon$.
    
    By comparing with \eqref{eq:zeitlin_hamiltonian}, we see that the correctly scaled modified Hamiltonian (which is consistent with $H(\omega)$ as $n\to\infty$) is $\tilde H_{n,h} := \frac{4\pi \hbar_n}{n}\tilde{\mathcal H}_{n,h}$.
    Thus, 
    \begin{equation*}
        \lvert \tilde{H}_{n,h}(W_k) - \tilde{H}_{n,h}(W_0) \rvert \leq \frac{4\pi \hbar_n}{n \hbar_n^2}A(\lVert W_0\rVert_\infty) \exp(-\epsilon_0/(2\epsilon\lVert W_0\rVert_\infty)) .
    \end{equation*}
    The result now follows since $\frac{1}{n\hbar} = \frac{\sqrt{n^2-1}}{2n} \leq 1/2$ and $\lVert W_0 \rVert_\infty \leq \lVert \omega_0 \rVert_{L^\infty} $, where the last inequality is a property of the Berezin--Toeplitz quantization operator $\mathcal T_n$ (cf.~\cite{Le2018}).
\end{proof}

\section{Butcher series and backward error analysis}\label{sec:butcher_bea}

The theory of Butcher series was introduced in the 1960's by Butcher \cite{butcherCoefficientsStudyRungeKutta1963} and later developed by Hairer and  Wanner~\cite{hairerButcherGroupGeneral1974} in the 1970's. Consider an ordinary differential equation in $\R^d$ of the following form
\[ y^\prime = f(y), \quad f : \R^d \to \R^d, \quad y(0) = y_0. \]
The solution $y(h)$ as well as the approximation $y_1$ computed using a Runge--Kutta method can be Taylor expanded around $h=0$. It was noted by Cayley~\cite{Cayley1857} in 1857 that the terms (also called \emph{elementary differntials}) appearing in such Taylor expansions can be represented by non-planar trees. 
Let $T$ be the set of non-planar rooted trees as defined in \cref{def:np_tree}.

\begin{definition}
    \label{def:np_tree}
    A \emph{non-planar tree} $\tau \in T$ is a connected directed acyclic graph with vertices $V(\tau)$ and edges $E(\tau)$ such that all vertices $v \in V(\tau)$ have at most one outgoing edge. Connectedness implies that there is exactly one vertex without an outgoing edge and this vertex is called the \emph{root}.
\end{definition}

All non-planar trees up to size $4$ can be found below with all edges being directed downwards,
\[ \forestA \,, \quad \forestB\,, \quad \forestC\,, \quad \forestD\,, \quad \forestE\,, \quad \forestF\,, \quad \forestG\,, \quad \forestH\,. \]
We note that since the trees are non-planar, the order of branches does not matter and we have
\[ \forestI = \forestJ. \]

Let $F$ be the set of non-planar forests defined in \cref{def:np_forest}.

\begin{definition}
    \label{def:np_forest}
    A \emph{non-planar forest} $\pi \in F$ is a directed acyclic graph such that each connected component is a non-planar tree.
\end{definition}

Note that the empty graph $\one$ is a forest. Let $B^+ : F \to T$ be a surjective map that adds a new vertex to a forest and connects all roots of the forest to the new vertex thus creating a tree whose root is the new vertex. For example, 
\[ B^+(\one) = \forestK \,, \quad B^+ (\forestL) = \forestM \,, \quad B^+(\forestN) = \forestO \,. \]

The correspondence between non-planar trees and elementary differentials is denoted by a map $\dF_f$ defined in \cref{def:F_np}

\begin{definition}
    \label{def:F_np}
    Let $\tau \in T$ be a non-planar tree with $\tau = B^+(\tau_1 \cdots \tau_n)$ for some $\tau_i \in T$. Then, the map $\dF_f$ is defined as
    \begin{align*}
        \dF_f (\tau) &:= \sum_{i_1, \dots, t_n =1}^d \dF_f (\tau_1)^{i_1} \cdots \dF_f (\tau_n)^{i_n} \frac{\partial^n}{\partial x_{i_1}\cdots \partial x_{i_n}} f \\
                    &= f^{(n)} \big(\dF_f (\tau_1), \dots, \dF_f (\tau_n)\big).
    \end{align*}
\end{definition}

Some examples of elementary differentials are,
\[ \dF_f(\forestP) = f \,, \quad \dF_f(\forestQ) = \sum_{i,j,k=1}^d f^i f^k \frac{\partial f^j}{\partial x_i} \frac{\partial^2 f}{\partial x_j \partial x_k} = f^{(2)} (f, f^\prime f).  \]

Let $\sigma : T \to \R$ denote the symmetry of a tree defined as the number of automorphisms of the tree, for example,
\[ \sigma(\forestR) = 1 \,, \quad \sigma(\forestS) = 2 \,, \quad \sigma(\forestT) = 4 \,, \quad \sigma (\forestU) = 3! \,. \]
Given a tree $\tau = B^+(\tau_1^{r_1} \cdots \tau_n^{r_n})$ where all $\tau_i$ are distinct with multiplicities $r_i$ for $i=1,\dots,n$, we have the following recursive formula,
\[ \sigma(\tau) = r_1! \cdots r_n! \sigma(\tau_1)^{r_1} \cdots \sigma(\tau_n)^{r_n} \,. \]
Let $|\tau|$ denote the number of vertices of a tree $\tau$, for example $|\forestV| = 5$.

\begin{definition}
    \label{def:Bseries_np}
    Let $a : T \to \R$ be a \emph{coefficient map}, then, a \emph{Butcher series} is defined as
    \[ B_f (a)(y_0) := \sum_{\tau \in T} h^{|\tau|} \frac{a(\tau)}{\sigma(\tau)} \dF_f(\tau)(y_0) \,, \]
    with the coefficient map $a$ characterizing the Butcher series.
\end{definition}

The Taylor expansions of the exact solution and of Runge--Kutta methods $\Phi_h$ are written using Butcher series as
\[ y(h) = y_0 + B_f(1/\gamma)(y_0), \quad \Phi_h (y_0) = y_0 + B_f(a) (y_0), \]
with the coefficient map $a : T \to \R$ being defined in \cref{prop:coeff_a_np}
and the map $\gamma$ being the factorial of a tree $\tau = B^+(\tau_1 \cdots \tau_n)$ defined as
\[ \gamma(\tau) = |\tau| \cdot \prod_{i=1}^n \gamma(\tau_i) \,, \quad \gamma(\forestW) = 1 \,. \]

\begin{proposition}
    \cite{butcherCoefficientsStudyRungeKutta1963}
    \label{prop:coeff_a_np}
    Consider a Runge--Kutta method $\Phi_h$ with coefficients $b_i, a_{ij}$ for $i, j = 1, \dots, s$. Then, the coefficient map $a : T \to \R$ of the Butcher series $B_f(a)$ such that $\Phi_h(y_0) = y_0 + B_f (a) (y_0)$ is given by
    \begin{align*}
        a\big(B^+(\tau_1 \cdots \tau_n)\big) &= \sum_{i=1}^s b_i a^i (\tau_1) \cdots a^i (\tau_n) \,, \\ 
        a^i\big(B^+(\tau_1 \cdots \tau_n)\big) &= \sum_{j=1}^s a_{ij} a^j (\tau_1) \cdots a^j (\tau_n) \,.
    \end{align*}
\end{proposition}

\subsection{Backward error analysis of a symplectic integrator}
\label{sec:bea_symplectic}

We present backward error analysis by following \cite[Ch. IX.9, IX.10]{HairerWannerGNI} and use the formalism of rooted trees to compute the modified equation and the modified Hamiltonian. Assume an integrator $\Phi_h$ can be written using Butcher series as
\[ \Phi_h (y) = y + B_f (a)(y) \,. \]
Let $\tau \butcher \gamma$ denote the Butcher product of trees $\tau, \gamma \in T$ which attaches the root of $\tau$ to the root of $\gamma$, for example,
\[ \forestX \butcher \forestY = \forestAB \,. \]
Integrator $\Phi_h$ is \emph{symplectic} if the coefficient map $a$ satisfies the following property
\[ a(\tau \butcher \gamma) + a(\gamma \butcher \tau) = a(\tau) a(\gamma) \quad \text{for } \tau, \gamma \in T \,. \]
Integrator $\Phi_h$ can be viewed as an exact flow of a \emph{modified vector field} $\tilde{f}_h$, that is,
\[ \Phi_h(y) = y + B_f(a)(y) = y + B_{\tilde{f}_h} (1/\gamma)(y) \,. \]

\begin{theorem}
    \label{thm:mod_eq_Bseries}
    \cite[Section IX.9]{HairerWannerGNI}
    Modified vector field $\tilde{f}_h$ of $\Phi_h$ is written as
    \[ \tilde{f}_h (y) = \frac{1}{h} B_f (b) (y) \,, \]
    where $b : T \to \R$ can be computed using formulas derived in \cite[Section IX.9]{HairerWannerGNI}, see also \cite{ChartierNIB07,ChartierASB10a}.
    Moreover, if the integrator $\Phi_h$ is symplectic and $f = J^{-1}\nabla H$ is a Hamiltonian system, then the coefficient map $b$ satisfies,
    \begin{equation}
        \label{eq:b_Butcher_product}
        b(\tau \butcher \gamma) + b(\gamma \butcher \tau) = 0 \,, \quad \text{for all } \tau, \gamma \in T \,.
    \end{equation}
    In particular, $b(\tau \butcher \tau) = 0$ for all $\tau \in T$.
\end{theorem}

Property \cref{eq:b_Butcher_product} is a consequence of the symplecticity of the integrator and is equivalent to the fact that the modified vector field $\tilde{f}_h$ is Hamiltonian \cite[Theorem IX.9.3]{HairerWannerGNI}.
It is used to introduce an equivalence relation on $T$,
\[ \tau \butcher \gamma \  \sim \  \gamma \butcher \tau \,, \quad \text{for all } \tau, \gamma \in T \,, \]
which, together with the property $b(\tau \butcher \tau) = 0$, leads to the introduction of \emph{non-superfluous free trees} which are equivalence classes of rooted trees with respect to the equivalence relation $\sim$ excluding equivalence classes which contain $\gamma = \tau \butcher \tau$ for some $\tau \in T$. The set of non-superfluous free trees is denoted by $FT^\prime$.

We set the \emph{canonical representative} of every equivalence class to be the maximal element of the class with respect to a total order of rooted trees defined in \cite{MuruaHAR06,bogfjellmoHamiltonianBseriesLie2017}. Let the canonical representative of $\hat{\tau} \in FT^\prime$ be denoted by $\hat{\tau}_* \in T$. 

The values of the coefficient map $b$ are completely determined by the values of $b$ on the canonical representatives of non-superfluous free trees which greatly reduces the number of computations that need to be performed to compute the values of $b$.

Canonical representatives in $FT^\prime$ up to size $5$ are
\[ \forestBB\,, \quad \forestCB \,, \quad \forestDB \,, \quad \forestEB \,, \quad \forestFB \,, \quad \forestGB \,, \dots. \]

The values of the coefficient map $b$ given in \cref{thm:mod_eq_Bseries} for canonical representatives of order up to $5$ are given in \cref{figure:values_b}.
\begin{table} 
    \centering
    \renewcommand{\arraystretch}{1.5}
    \begin{tabular}{ccc}
        \toprule
        $|\hat{\tau}_*|$ & $\hat{\tau}_*$ & $b(\hat{\tau}_*)$ \\
        \midrule
        $1$ & $\forestHB$            & $1$ \\
        $3$ & $\forestIB$       & $a(\forestJB) - \frac{1}{3}$ \\
        $4$ & $\forestKB$     & $a(\forestLB) - \frac{3}{2} a(\forestMB) + \frac{1}{4}$ \\
        $5$ & $\forestNB$   & $a(\forestOB) - 2 a(\forestPB) + a(\forestQB) - \frac{1}{30}$ \\
        $5$ & $\forestRB$  & $a(\forestSB) - \frac{1}{2} a(\forestTB) + \frac{1}{4} a(\forestUB) - \frac{7}{120}$ \\
        $5$ & $\forestVB$ & $a(\forestWB) - \frac{1}{20}$ \\
        \bottomrule
    \end{tabular}
    \caption{Values of $b(\hat{\tau}_*)$ for all canonical representatives of $\hat\tau \in FT^\prime$ of size up to $5$, $|\hat{\tau}_*| \leq 5$.}
    \label{figure:values_b}
\end{table}

\begin{example}
    For the implicit midpoint method, the coefficient map is given by $a(\tau) = 1 / 2^{|\tau|-1}$. Therefore,
    \[ b(\forestXB) = - \frac{1}{12} \,, \quad b(\forestYB) = 0 \,, \quad b(\forestAC) = \frac{7}{240} \,, \quad b(\forestBC) = \frac{1}{240} \,, \quad b(\forestCC) = \frac{1}{80} \,. \]
\end{example}

We recall that the modified equation of a symplectic integrator applied to a Hamiltonian system is a modified Hamiltonian system. That is, there exists a modified Hamiltonian $\tilde{H}_h$ such that $\tilde{f}_h = J^{-1} \nabla \tilde{H}_h$. Let us now give an explicit formula for $\tilde{H}_h$ as it is derived in \cite[Section IX.9]{HairerWannerGNI}. 

\begin{definition}
    An \emph{elementary Hamiltonian} $H(\tau)$ for $\tau \in T$ is defined as
    \[ H(\bullet) = H\,, \quad H(\tau) = H^{(n)} \big( \dF_f(\tau_1), \dots, \dF_f(\tau_n) \big) \,, \]
    for $\tau = B^+(\tau_1 \cdots \tau_n)$.
\end{definition}

Elementary Hamiltonians are used to represent the modified Hamiltonian as a formal sum over non-superfluous non-rooted trees as follows.

\begin{theorem}
    \label{thm:mod_Hamiltonian}
    Consider a symplectic integrator applied to a Hamiltonian system, then, its modified equation is Hamiltonian with the modified Hamiltonian given by
    \[ \tilde{H}_h (y) = \sum_{\hat\tau \in FT^\prime} h^{|\hat\tau_*|-1} \frac{b(\hat\tau_*)}{\sigma(\hat\tau_*)} H(\hat\tau_*)(y) \,. \]
\end{theorem}

\begin{example}
    The modified Hamiltonian of the implicit midpoint method is
    \begin{align*}
        \tilde{H}_h (y) = H(y) &- \frac{1}{24} h^2 H^{(2)}(y) \big[ f(y), f(y) \big] \\
                               &+ \frac{7}{5760} h^4 H^{(4)} (y) \big[ f(y), f(y), f(y), f(y) \big] \\
                               &+ \frac{1}{480} h^4 H^{(3)} (y) \big[ f(y), f^\prime(y) f(y), f(y) \big] \\
                               &+ \frac{1}{160} h^4 H^{(2)} (y) \big[ f^\prime(y) f(y), f^\prime (y) f(y) \big] + \OO(h^6) \,.
    \end{align*}
\end{example}

\subsubsection{Analytic bounds}

We note that the series defining the modified Hamiltonian $\tilde{H}_h$ is a formal power series in $h$ and its convergence is not guaranteed. As shown in \cite[Ch. IX.8]{HairerWannerGNI} (see also \cite{BeGi1994}), there exists an optimally truncated modified Hamiltonian $\tilde{H}_h^N$ such that the numerical trajectory generated by the integrator $\Phi_h$ preserves $\tilde{H}^N_h$ over exponentially long times. More precisely, there exists a constant $c > 0$ such that
\[ \tilde{H}_h^N (y_n) = \tilde{H}_h^N (y_0) + \OO(e^{-c/h}) \,, \quad \text{for } nh \leq e^{c/h} \,. \]
This result is obtained by assuming that the trajectory stays on a compact domain. This assumption does not hold in our setting as is discussed in the introduction and, therefore, we obtain this result with a different set of assumptions in \cref{sec:long_time_behavior}. The key ingredients for our derivation are the bounds of the number of rooted trees of size $j \in \N$ as well as the bounds of $|b(\tau)|$ for $\tau \in T$. We derive these bounds below.

\begin{lemma}
    \label{lem:combinatorial_bounds_classic}
    The size of the set $T_j$ of classical trees of size $j$ is bounded as,
    \[ |T_j| \leq 3^{j-1} \,. \]
\end{lemma}
\begin{proof}
    Following \cite{OtterNT48}, the asymptotic estimate for the number $|T_j|$ of non-labeled rooted trees of size $j \in \N$ is given by,
    \[ |T_j| \sim D \alpha^j j^{-3/2} \,, \quad \text{as } j \to \infty \,, \]
    with $D \approx 0.44$ and $\alpha \approx 2.955$. Since $|T_1| = 1$ and the growth of the number of trees accelerates while never growing faster than $3^j$, we have, $|T_j| \leq 3^{j-1}$ for all $j \in \N$.
\end{proof}

Let us consider the product $a \graft b$ of coefficient maps $a, b : T \to \R$ defined as,
\[ (a \graft b) (\tau) = \sum_{e \in E(\tau)} a(\tau \setminus_1 e) b(\tau \setminus_2 e) \,, \]
where $E(\tau)$ is the set of edges of $\tau$, $\tau \setminus e$ is the pair of trees obtained by cutting the edge $e$ of $\tau$, $\tau \setminus_2 e$ is the tree which contains the root of $\tau$, and $\tau \setminus_1 e$ is the tree which does not contain the root of $\tau$. Coefficient map $a^{\graft m}$ is defined as
\begin{align*}
    a^{\graft m} (\tau) &= \big( a \graft (a \graft \cdots (a \graft a) \cdots ) \big) (\tau) \\
                        &= \sum_{\substack{e_i \in E(\tau)\\ i = 1, \dots, m}} a(\tau \setminus_1 e_1) \cdots a(\tau \setminus_1 e_m) a\big(\tau \setminus_2 (e_1, \dots, e_m)\big) \,,
\end{align*}
where $\tau \setminus_2 (e_1, \dots, e_m) = ( \cdots (\tau \setminus_2 e_1) \cdots ) \setminus_2 e_m$ is the tree obtained by cutting the edges $e_1, \dots, e_m$ of $\tau$ and keeping only the part which contains the root of $\tau$. We note that $a^{\graft m} (\tau) = 0$ if $m \geq |\tau|$ since there are only $|\tau| - 1$ edges in $\tau$. We denote $a^{\graft 0} = 0$.

Let $b : T \to \R$ denote the coefficient map for the modified equation of the symplectic Runge--Kutta method with coefficient map $a : T \to \R$.
Following the discussion in \cite[Chapter IX.9.1]{HairerWannerGNI} (see also \cite{ChartierASB10a}), we have,
\[ a = \exp^\graft(b) = \sum_{m=1}^\infty \frac1{m!} b^{\graft m} \,, \]
analogously to $1/\gamma = \exp^\graft(\delta_\bullet)$ for the exact solution where $\delta_\bullet (\forestDC) = 1$ and $\delta_\bullet (\tau) = 0$ for $\tau \neq \bullet$.

\begin{lemma}
    \label{lem:bound_b}
    Let an integrator $\Phi_h$ have a Taylor expansion written as a Butcher series with coefficient map $a$ satisfying $|a(\tau)| \leq C^{|\tau|}_a$ for some constant $C_a \in \R$ and let the modified vector field $\tilde{f}_h$ be expressed using a Butcher series with coefficient map $b$, then,
    \[ |b(\tau)| \leq e (|\tau| - 1)! C_a^{|\tau|} \,. \]
\end{lemma}
\begin{proof}
    The coefficient map $b : T \to \R$ can be computed using the logarithm with respect to the product $\graft$ as,
    \[ b = \log^\graft(a) = \sum_{m=1}^\infty \frac{(-1)^m}{m} a^{\graft m} \,. \]
    Therefore, we obtain the following bound for $b(\tau)$,
    \begin{align*}
        |b(\tau)| &\leq \sum_{m=1}^{|\tau|-1} \frac1{m} |a^{\graft m}(\tau)| \\
                  &\leq \sum_{m=1}^{|\tau|-1} \frac1{m} \sum_{\substack{e_i \in E(\tau)\\ i = 1, \dots, m}} |a(\tau \setminus_1 e_1)| \cdots |a(\tau \setminus_1 e_m)| |a\big(\tau \setminus_2 (e_1, \dots, e_m)\big)| \\
                  &\leq \sum_{m=1}^{|\tau|-1} \frac1{m} C_a^{|\tau|} \sum_{\substack{e_i \in E(\tau)\\ i = 1, \dots, m}} 1 \,,
    \end{align*}
    where we use the bound of $|a(\tau)|$ for $\tau \in T$. It remains to bound the number of terms in the sum indexed by $e_i \in E(\tau)$ for $i = 1, \dots, m$. We note that depending on the structure of the tree $\tau$, the order in which the edges are chosen may or may not matter. In the worst case, the order does matter and we have $|E(\tau)| \cdot (|E(\tau)| - 1) \cdots (|E(\tau)| - m + 1)$ terms in the sum, which is bounded by $(|\tau| - 1)! / (|\tau| - m - 1)!$. Therefore, we have,
    \begin{align*}
        |b(\tau)| &\leq \sum_{m=1}^{|\tau|-1} \frac1{m} C_a^{|\tau|} \frac{(|\tau| - 1)!}{(|\tau| - m - 1)!} \\
                  &\leq C_a^{|\tau|} (|\tau| - 1)! \sum_{m=1}^{|\tau|-1} \frac1{m (|\tau| - m - 1)!}
                  \intertext{we remove $1/m$ term and reorder the sum,}
                  &\leq C_a^{|\tau|} (|\tau| - 1)! \sum_{m=1}^{|\tau|-1} \frac1{(m-1)!} \leq e (|\tau| - 1)! C_a^{|\tau|} \,,
    \end{align*}
    where we use $\sum_{m=1}^{|\tau|-1} \frac1{(m-1)!} \leq \sum_{m=1}^{\infty} \frac1{ (m-1)!} = e$.
\end{proof}

\section{Butcher series over biplanar forests}
\label{sec:butcher_biplanar}

We introduce the algebraic tools necessary for the description of Lie--Poisson reduction of Butcher series which leads to the introduction of biplanar forests and forest momentum map. Lie--Poisson reduction of Butcher series is a key ingredient in the backward error analysis of ISOSRYK methods.

Let $\UU$ be the universal enveloping algebra of the Lie algebra $\fraku(n)$. Let us consider the commutative tensor product $\UU^2 = \UU \cdot \UU$ where $(A,B) = (B,A)$ for $(A,B) \in \UU^2$. We endow $\UU^2$ with the product,
\[ (A, B) \cdot (C, D) = (AC, BD) + (AD, BC) \,, \quad \text{for } A, B, C, D \in \UU \,. \]
Let us introduce the following action $\odot : \UU^2 \otimes \fraku(n) \to \fraku(n)$. Let $A, B \in \UU$ and $W \in \fraku(n)$, then,
\[ (A, B) \odot W := AWB^\dagger + BWA^\dagger \,. \]
We define $A \odot W$ as $(A,I) \odot W$. This action arises naturally when considering directional derivatives of the momentum map $\mu : T^*\text{U}(n) \to \fraku(n)^* \cong \fraku(n)$, $\mu(Q,P) = Q^\dagger P$ as is done in \cref{sec:IsoSyRK_expansion}.

In the rest of this section, we introduce Butcher series over biplanar forests which are constructed using the action of $\UU^2$ on $\fraku(n)$. We introduce the forest momentum map and use it to obtain the Butcher series of isospectral symplectic Runge--Kutta (ISOSYRK) methods from the corresponding Butcher series of the symplectic Runge--Kutta methods through the \emph{Lie--Poisson reduction of Butcher series}. We use the obtained biplanar Butcher series to compare ISOSYRK methods to the well-known Runge--Kutta--Munthe-Kaas methods for Lie--Poisson systems.

\subsection{Biplanar forests}

Let us introduce biplanar trees and forests which are used to denote the terms appearing in the Taylor expansion of ISOSYRK methods as well as the exact solution of an isospectral differential equation \cref{eq:isospectral}. In particular, biplanar forests represent elements of $\UU^2$ through the map $\dF_f$ defined in \cref{def:F_f_bp}.
\begin{definition}
    \label{def:bp_tree_forest}
    A \emph{biplanar tree} $\tau = B^+(\pi,\eta)$ is a tree with two sets of branches $\pi$ and $\eta$ where $\pi$ and $\eta$ are ordered monomials of biplanar trees. A \emph{biplanar forest} is a pair $(\pi, \eta)$ of ordered monomials of biplanar trees $\pi$ and $\eta$. We assume that the pairing is commutative, that is, $(\pi, \eta) = (\eta, \pi)$.
\end{definition}

Let $BT$ and $BF$ denote the sets of biplanar trees and forests defined in \cref{def:bp_tree_forest}. 
We draw biplanar forest $(\pi, \eta)$ by separating $\pi$ and $\eta$ by $\times$. If $\pi = \one$ or $\eta = \one$, then $\times$ is omitted.
Some examples of biplanar forests can be found below,
\[ \one, \quad \forestEC, \quad \forestFC\,, \quad \text{where } \forestGC \neq \forestHC\,, \quad \text{but } \forestIC = \forestJC \,. \]
The vector spaces spanned by biplanar forests and trees are denoted by $\BF$ and $\BT$, respectively. Let $T(\BT)$ denote the tensor algebra of biplanar trees, then, $\BF$ is the commutative tensor product of $T(\BT)$ with itself, that is, $\BF = T(\BT)^2 = T(\BT) \cdot T(\BT)$ and $\BT = B^+ (\BF)$ as vector spaces. Let us abuse the notation and denote the set of ordered monomials of biplanar trees by $T(BT)$.

\begin{definition}
    \label{def:F_f_bp}
    Let $\pi, \eta \in T(BT)$ and $(\pi, \eta) \in BF$, then, we define,
    \begin{align*}
        \dF_f (\pi, \eta) (W) &= \big(\dF_f(\pi)(W), \dF_f(\eta)(W)\big) &&\in \UU^2 \,, \\
        \dF_f (\pi \cdot \eta) (W) &= \dF_f(\eta)(W) \cdot \dF_f(\pi)(W) &&\in \UU \,, \\
        \dF_f \big(B^+\big(\pi, \eta)\big) (W) &= f\big(\dF_f (\pi, \eta) \odot W \big) &&\in \fraku(n) \,,
    \end{align*}
    and $\dF_f (\bullet) (W) = f(W)$. Note that the order of $\pi$ and $\eta$ is reversed in the second expression. We often omit writting $W$ for conciseness.
\end{definition}
For example,
\begin{align*}
    \dF_f(\forestKC) \odot W = &\dF_f(\forestLC) (W) \cdot W \cdot \dF_f(\forestMC)^\dagger + \dF_f(\forestNC) (W) \cdot W \cdot \dF_f(\forestOC)^\dagger \\ 
                                    = &f(f(W) \cdot W) \cdot W \cdot f(W)^\dagger + f(W \cdot f(W)^\dagger) \cdot W \cdot f(W)^\dagger  \\
                                      &+ f(W) \cdot W \cdot f(f(W) \cdot W)^\dagger + f(W) \cdot W \cdot f(W \cdot f(W)^\dagger)^\dagger \,.
\end{align*}

Let the \emph{symmetry} $\sigma(\pi, \eta)$ of a biplanar forest $(\pi, \eta) \in BF$ be defined as,
    \[ \sigma(\pi, \eta) = \begin{cases} 2 \sigma(\pi)^2 \,, \quad \text{if } \pi = \eta \,, \\ \sigma(\pi) \sigma(\eta) \,, \quad \text{otherwise,} \end{cases}\]
    and $\sigma(\pi \cdot \eta) = \sigma(\pi) \sigma(\eta) \,, \sigma \big(B^+(\pi,\eta)\big) = \sigma(\pi, \eta) \,, \sigma(\one) = 1$.
Some examples of the values of the symmetry $\sigma(\pi, \eta)$ for $(\pi, \eta) \in BF$ can be found below,
\[ \sigma(\forestPC) = 2 \,, \quad \sigma(\forestQC) = 8 \,, \quad \sigma(\forestRC) = 2 \,, \quad \sigma(\forestSC) = 1 \,. \]

\begin{definition}
    A \emph{biplanar Butcher series} is the following formal sum, 
    \[ B_f (\alpha) := \sum_{(\pi, \eta) \in BF} h^{|\pi,\eta|} \frac{\alpha(\pi, \eta)}{\sigma(\pi, \eta)} \dF_f (\pi, \eta) \,, \]
    where $\alpha : BF \to \R$ is a coefficient map over biplanar forests and $|\pi,\eta|$ denotes the number of vertices in $(\pi, \eta)$.
\end{definition}

We note that we denote the coefficient maps over classical trees and forests by latin letters, for example, $a : F \to \R \,, \  b : T \to \R$, and the coefficient maps over biplanar trees and forests by greek letters, for example, $\alpha : BF \to \R \,, \  \beta : T(BT) \to \R$. This also allows us to distinguish between classical Butcher series and biplanar Butcher series.

\subsubsection{Coadjoint coefficient maps}

In this section, we introduce \emph{coadjoint coeficient maps} and \emph{infinitisimal coadjoint coefficient maps} which correspond to biplanar Butcher series with nice geometric properties.

Recall that the tensor algebra $T(\BT)$ is a Hopf algebra with the deconcatenation coproduct $\Delta_\cdot : T(\BT) \to T(\BT) \otimes T(\BT)$ defined, for $\tau_i \in BT$, as,
\[ \Delta_\cdot (\tau_1 \cdots \tau_n) = \one \otimes \tau_1 \cdots \tau_n + \sum_{i=1}^{n-1} (\tau_1 \cdots \tau_i) \otimes (\tau_{i+1} \cdots \tau_n) + \tau_1 \cdots \tau_n \otimes \one \,, \]
and antipode $S : T(\BT) \to T(\BT)$ defined as,
\[ S(\tau_1 \cdots \tau_n) = (-1)^{n} \tau_n \cdots \tau_1 \,.  \]
For example, $\Delta_\cdot (\tau_1 \tau_2 \tau_3) = \one \otimes \tau_1 \tau_2 \tau_3 + \tau_1 \otimes \tau_2 \tau_3 + \tau_1 \tau_2 \otimes \tau_3 + \tau_1 \tau_2 \tau_3 \otimes \one$ and $S(\tau_1 \tau_2 \tau_3) = - \tau_3 \tau_2 \tau_1$.

Let $\pi \in T(\BT)$, we note the following property of the elementary differential $\dF_f(\pi)$,
\begin{equation}
    \label{eq:F_antipode}
    \dF_f (\pi)^\dagger = \dF_f \big(S(\pi)\big) \,.
\end{equation}

Let $\alpha : BF \to \R$ be a coefficient map over biplanar forests, then, we denote by $\alpha_T$ the restriction of $\alpha $ to $T(BT) \cong T(BT) \cdot \one \subset BF$, that is, $\alpha_T := \alpha \restrict{T(BT)}$. Given two coefficient maps $\alpha, \tilde\alpha : BF \to \R$, we define $\alpha_T \cdot \tilde{\alpha}_T : T(BT) \to \R$ as,
\[ (\alpha_T \cdot \tilde{\alpha}_T)(\pi) = \sum_{(\pi)} \alpha_T (\pi_{(1)}) \tilde{\alpha}_T(\pi_{(2)}) \,, \]
where $\Delta_\cdot (\pi) = \sum_{(\pi)} \pi_{(1)} \otimes \pi_{(2)}$ is the deconcatenation coproduct of $\pi$ written using Sweedler's notation. Denote by $\delta_\one : T(BT) \to \R$ the coefficient map $\delta_\one (\one) = 1$ and $\delta_\one (\pi) = 0$ for $\pi \neq \one$.

\begin{proposition}
    \label{prop:coadjoint_coeff}
    Let $\alpha : BF \to \R$ satisfy the following properties for all $\pi, \eta \in T(BT)$,
    \begin{equation}
        \label{eq:coadjoint_coeff_prop_1}
        \alpha_T(\one) = 1 \,, \quad \alpha(\pi, \eta) = \alpha_T(\pi) \alpha_T(\eta) \,, \quad \alpha_T \cdot (\alpha_T \circ S) = (\alpha_T \circ S) \cdot \alpha_T = \delta_\one \,,
    \end{equation}
    then, $B_f (\alpha_T) \in \text{U}(n)$ is an element of the unitary group and we have,
    \[ B_f (\alpha) \odot W = \Ad^*_{B_f (\alpha_T)} W \,. \]
    Such coefficient maps $\alpha : BF \to \R$ are called \emph{coadjoint coefficient maps}.
\end{proposition}
\begin{proof}
    We use the properties \cref{eq:coadjoint_coeff_prop_1} to perform the following computation,
    \begin{align*}
        B_f (\alpha) \odot W &= \sum_{(\pi, \eta) \in BF} \frac{\alpha(\pi, \eta)}{\sigma(\pi, \eta)} \dF_f (\pi, \eta) \odot W 
                           = \sum_{\substack{\pi, \eta \in T(\BT)}} \frac{\alpha(\pi)}{\sigma(\pi)} \frac{\alpha(\eta)}{\sigma(\eta)} \dF_f (\pi) \cdot W \cdot \dF_f (\eta)^\dagger \\
                           &= B_f (a_T) \cdot W \cdot B_f (a_T)^\dagger \,,
    \end{align*}
    We need to prove that $B_f (a_T) \in U(n)$.
    We use \cref{eq:F_antipode} and the fact that the antipode $S$ is self-adjoint to obtain,
    \[ B_f (\alpha_T)^\dagger = B_f (\alpha_T \circ S) \,. \]
    Therefore, we have,
    \begin{align*}
        B_f (\alpha_T)^{\dagger} B_f (\alpha_T) &= B_f (\alpha_T \cdot (\alpha_T \circ S)) = I \,, \\
        B_f (\alpha_T) B_f (\alpha_T)^{\dagger} &= B_f ((\alpha_T \circ S) \cdot \alpha_T) = I \,.
    \end{align*}
    This finishes the proof.
\end{proof}

\begin{proposition}
    \label{prop:infinitesimal_coadjoint_coeff}
    Let the coefficient map $\beta : BF \to \R$ satisfy the following properties for all $\pi, \eta \in T(BT)$,
    \begin{equation}
        \label{eq:inf_coadjoint_coeff_prop_1}
        \beta_T(\one) = 0 \,, \quad \beta(\pi, \eta) = \delta_\one (\pi) \beta_T(\eta) + \beta_T(\pi) \delta_\one (\eta) \,, \quad (\beta_T \circ S)(\pi) = -\beta_T(\pi) \,,
    \end{equation}
    then, $B_f (\beta_T) \in \fraku(n)$ is an element of the Lie algebra $\fraku(n)$ and we have,
    \[ B_f (\beta) \odot W = \ad^*_{B_f(\beta_T)(W)} W \,. \]
    Such coefficient maps $\beta : BF \to \R$ are called \emph{infinitesimal coadjoint coefficient maps}.
\end{proposition}
\begin{proof}
    We use assumption \cref{eq:inf_coadjoint_coeff_prop_1} to obtain the following computation,
    \begin{align*}
        B_f (\beta) \odot W &= \sum_{\pi \in T(BT)_*} \frac{\beta_T(\pi)}{\sigma(\pi)} \dF_f (\pi) \odot W \\
        \intertext{where the sum $\pi \in T(BT)_*$ is over non-empty monomials $\tau_1 \cdots \tau_n$ with $\tau_i \in BT$,}
                              &= \sum_{\pi \in T(BT)_*} \frac{\beta_T(\pi)}{\sigma(\pi)} \big( \dF_f (\pi)(W) \cdot W + W \cdot \dF_f (\pi)(W)^\dagger \big) \\
                              &= B_f (\beta_T) \cdot W + W \cdot B_f (\beta_T)^\dagger \,.
    \end{align*}
    We need to prove that $B_f (\beta_T) \in \fraku(n)$.
    Since $\dF_f (\pi)^\dagger = \dF_f (S(\pi))$ by \cref{eq:F_antipode} and the antipode $S$ is self-adjoint, we have,
    \[ B_f (\beta_T)^\dagger = B_f (\beta_T \circ S) = - B_f (\beta_T) \,. \]
    The statement then follows.
\end{proof}

\subsection{Forest momentum map}

Let us introduce the \emph{forest momentum map} which plays a central role in the Lie--Poisson reduction of Butcher series.
Let $T^2$ denote the set of non-planar forests with at most two trees.

\begin{definition}
    The \emph{forest momentum map} $\psi : BF \to T^2$ is defined as,
    \begin{align}
        \psi (\pi, \eta) &= \psi(\pi) \cdot \psi(\eta) \,, \label{eq:psi_1} \\
        \psi \big(\pi \cdot B^+ (\eta, \nu) \big) &= B^+ \big( \psi(\pi) \cdot \psi(\eta) \cdot \psi(\nu) \big) \,, \label{eq:psi_2}
    \end{align}
    where $\pi, \eta, \nu \in T(BT)$ and $\psi (\one) = \one$.
\end{definition}

For example,
\[ \psi(\forestTC) = \forestUC \,, \quad \psi(\forestVC) = \psi(\forestWC) = \forestXC \,, \quad \psi(\forestYC) = \psi(\forestAD) = \forestBD \,, \]
\[ \psi(\forestCD) = \psi(\forestDD) = \psi(\forestED) = \psi(\forestFD) = \forestGD \,. \]

The map $\psi$ is related to the Butcher product $\butcher : \TT \otimes \TT \to \TT$ in the following way,
\[ \psi \big(\pi \cdot \tau \big) = \psi(\pi) \butcher \psi(\tau) \,, \]
where $\pi \in T(BT), \tau \in BT$. Let the Butcher product be extended to $\butcher : \TT \cup \{\one\} \otimes \TT \to \TT$ as,
\[ \one \butcher \tau = \tau \,, \quad \text{for } \tau \in \TT \,. \]

Let us abuse the notation and denote by $\langle \blank, \blank \rangle_\sigma$ inner products over non-planar and biplanar forests $\FF$ and $\BF$ defined as
\[ \langle \pi, \eta \rangle_\sigma = \begin{cases} \sigma(\pi) \,, \quad \text{if } \pi = \eta \,, \\ 0 \,, \quad \text{otherwise.} \end{cases}\]

Let $\psi^* : \TT^2 \to \BF$ denote the adjoint of the momentum forest map $\psi$ with respect to the inner products $\langle \blank, \blank \rangle_\sigma$, that is,
\[ \langle \psi(\pi, \eta) , \nu \rangle_\sigma = \langle (\pi, \eta), \psi^* (\nu) \rangle_\sigma \,, \quad \text{for } \pi, \eta \in T(BT) \,, \nu \in \TT^2 \,. \]

The rest of this section is devoted to the derivation of an explicit formula for $\psi^*$. Let us start by introducing the necessary algebraic ingredients.
Let $\Delta_\shuffle : \FF \to \FF \otimes \FF$ denote the deshuffle coproduct defined as
\[ \Delta_\shuffle (\tau) = \one \otimes \tau + \tau \otimes \one \,, \quad \Delta_\shuffle (\pi \cdot \eta) = \Delta_\shuffle (\pi) \cdot \Delta_\shuffle (\eta) \,, \]
where $(a \otimes b) \cdot (c \otimes d) = (a \cdot c) \otimes (b \cdot d)$. For example,
\begin{align*}
    \Delta_\shuffle (\tau_1 \tau_2 \tau_3) &= \tau_1 \tau_2 \tau_3 \otimes \one + \tau_1 \tau_2 \otimes \tau_3 + \tau_1 \tau_3 \otimes \tau_2 + \tau_2 \tau_3 \otimes \tau_1 \\
                                           &\quad + \tau_1 \otimes \tau_2 \tau_3 + \tau_2 \otimes \tau_1 \tau_3 + \tau_3 \otimes \tau_1 \tau_2 + \one \otimes \tau_1 \tau_2 \tau_3 \,.
\end{align*}
Let $\Delta_\butcher : \TT \to \TT \cup \{\one\} \otimes \TT$ denote the map which, for $\tau = B^+(\tau_1 \cdots \tau_n)$, is defined as
\[ \Delta_\butcher (\tau) = \one \otimes \tau + \sum_{i=1}^n \tau_i \otimes B^+(\tau_1 \cdots \tau_{i-1} \tau_{i+1} \cdots \tau_n) \,. \]

\begin{proposition}
    \label{prop:adjoint_coproducts}
    We have the following relations,
    \begin{enumerate}
        \item $\langle \pi \cdot \eta \,, \nu \rangle_\sigma = \langle \pi \otimes \eta \,, \Delta_\cdot (\nu) \rangle_\sigma$  for $\pi, \eta, \nu \in T(BT)$,
        \item $\langle \pi \cdot \eta \,, \nu \rangle_\sigma = \langle \pi \otimes \eta \,, \Delta_\shuffle (\nu) \rangle_\sigma$ for $\pi, \eta, \nu \in \FF$,
        \item $\langle \tau \butcher \gamma \,, \rho \rangle_\sigma = \langle \tau \otimes \gamma \,, \Delta_\butcher (\rho) \rangle_\sigma$ for $\tau \in \TT \cup \{\one\}, \gamma, \rho \in \TT$.
    \end{enumerate}
    That is, $\Delta_\cdot$ is the adjoint of the concatenation $\cdot$ on $T(\BT)$, $\Delta_\shuffle$ is the adjoint of the commutative concatenation $\cdot$ on $\FF$, and $\Delta_\butcher$ is the adjoint of the Butcher product $\butcher$ on $\TT$ with respect to the inner product $\langle \blank, \blank \rangle_\sigma$.
\end{proposition}

We prove some properties of $\psi^*$ in \cref{lemma:psi*_prop} that are used in \cref{prop:psi*_formula} to derive a formula for $\psi^*$.

\begin{lemma}
    \label{lemma:psi*_prop}
    The adjoint $\psi^* : \TT^2 \to \BF$ of the forest momentum map $\psi$ satisfies the following properties,
    \begin{enumerate}
        \item given $\tau \neq \gamma \in T$ and $\pi \in T(BT)$, if $\langle \psi^*(\tau), \pi \rangle_\sigma \neq 0$, then $\langle \psi^*(\gamma), \pi \rangle_\sigma = 0$,
        \item given $\tau \in T$ and $\pi \in F$, we have
            \[ \langle \psi^*\big( B^+(\pi) \big) , \tau \rangle_\sigma = \langle B^+\big( \psi^*(\pi) \big) , \tau \rangle_\sigma \,. \]
    \end{enumerate}
\end{lemma}
\begin{proof}
    The first property follows from the fact that $\psi$ restricts to a map between the basis $T(BT)$ of $T(\BT)$ and the basis $T$ of $\TT$. That is, for $\pi \in T(BT)$ we have $\psi(\pi) = \tau$ for some $\tau \in T$.
The second property follows from the following computation,
    \begin{align*}
        \langle \psi^*\big( B^+(\pi) \big) , \tau \rangle_\sigma &= \langle B^+(\pi) , \psi(\tau) \rangle_\sigma
        \intertext{let $\tau = B^+(\pi_\tau)$, then we use the definition of $\psi$,}
                                                                 &= \langle B^+(\pi) , B^+\big( \psi(\pi_\tau) \big) \rangle_\sigma
        \intertext{use the property $\sigma \circ B^+ = \sigma$ and the definition of the inner product,}
                                                                 &= \langle \pi , \psi(\pi_\tau) \rangle_\sigma = \langle \psi^*(\pi) , \pi_\tau \rangle_\sigma = \langle B^+ \big( \psi^*(\pi) \big) , \tau \rangle_\sigma \,.
    \end{align*}
    This finishes the proof.
\end{proof}

\begin{proposition}
    \label{prop:psi*_formula}
    Let us denote $\Delta_\butcher (\tau) = \sum_{(\tau)} \tau_{(1)} \otimes \tau_{(2)}$ and $\pi_{(2)}$ be such that $\tau_{(2)} = B^+(\pi_{(2)})$. The adjoint $\psi^* : \TT^2 \to \BF$ of the forest momentum map $\psi$ has the following form,
    \begin{align}
        \psi^* (\tau) &= \sum_{(\tau)} \psi^* (\tau_{(1)}) B^+ \big( \psi^*(\pi_{(2)}) \big) \,, \label{eq:psi*_1} \\
        \psi^* (\tau_1 \tau_2) &= \big( \psi^* (\tau_1), \psi^*(\tau_2) \big) \,, \label{eq:psi*_2}
    \end{align}
    with $\psi^*(\one) = \one$ and $\psi^*(\forestHD) = \forestID$. 
\end{proposition}
\begin{proof}
    Let us prove formula \cref{eq:psi*_2} by following the computation with $(\pi_1, \pi_2) \in BF$ and $\tau_1 \tau_2 \in T$,
    \begin{align*}
        \langle \psi^* (\tau_1 \tau_2) , (\pi_1, \pi_2) \rangle_\sigma &= \langle \tau_1 \tau_2 , \psi(\pi_1, \pi_2) \rangle_\sigma \\
                                                                       &= \langle \tau_1 \tau_2 , \psi(\pi_1) \cdot \psi(\pi_2) \rangle_\sigma \\
                                                                       &= \langle \Delta_\shuffle (\tau_1 \tau_2) , \psi(\pi_1) \otimes \psi(\pi_2) \rangle_\sigma
                                                                       \intertext{we use the fact $\langle \psi(\pi_i), \tau_1 \tau_2 \rangle_\sigma = 0$ to ignore the corresponding terms in $\Delta_\shuffle (\tau_1 \tau_2)$,}
                                                                       &= \langle \tau_1 \otimes \tau_2 + \tau_2 \otimes \tau_1 , \psi(\pi_1) \otimes \psi(\pi_2) \rangle_\sigma \\
                                                                       &= \langle \psi^*(\tau_1), \pi_1 \rangle_\sigma \langle \psi^*(\tau_2), \pi_2 \rangle_\sigma + \langle \psi^*(\tau_1), \pi_2 \rangle_\sigma \langle \psi^*(\tau_2), \pi_1 \rangle_\sigma \,.
    \end{align*}
    We have two cases for which $\langle \psi^* (\tau_1 \tau_2) , (\pi_1, \pi_2) \rangle_\sigma \neq 0$. In the first case $\tau_1 = \tau_2$ and $\pi_1 = \pi_2$, then,
    \[ \langle \psi^* (\tau_1 \tau_1) , (\pi_1, \pi_1) \rangle_\sigma = 2 \langle \psi^*(\tau_1), \pi_1 \rangle_\sigma^2 = \langle \big(\psi^*(\tau_1), \psi^*(\tau_1)\big), (\pi_1, \pi_1) \rangle_\sigma \,, \]
    following the definition of the symmetry coefficient $\sigma$ and the inner product on biplanar forests. In the second case $\tau_1 \neq \tau_2$ and $\pi_1 \neq \pi_2$, then, following property (1) in \cref{lemma:psi*_prop} and assuming $\langle \psi^*(\tau_1), \pi_1 \rangle_\sigma \neq 0$ and $\langle \psi^*(\tau_2), \pi_2 \rangle_\sigma \neq 0$, we have,
    \[ \langle \psi^* (\tau_1 \tau_2) , (\pi_1, \pi_2) \rangle_\sigma = \langle \psi^*(\tau_1), \pi_1 \rangle_\sigma \langle \psi^*(\tau_2), \pi_2 \rangle_\sigma = \langle \big(\psi^*(\tau_1), \psi^*(\tau_2)\big) , (\pi_1, \pi_2) \rangle_\sigma \,. \]
    This proves the formula \cref{eq:psi*_2}

    The formula \cref{eq:psi*_1} is follows from the computation with $\pi \in BF, \gamma \in BT$, and $\tau \in T$,
    \begin{align*}
        \langle \psi^*(\tau), \pi \gamma \rangle_\sigma &= \langle \tau, \psi(\pi \gamma) \rangle_\sigma = \langle \tau, \psi(\pi) \butcher \psi(\gamma) \rangle_\sigma
        \intertext{apply property (2) from \cref{prop:adjoint_coproducts},}
                                                        &= \langle \Delta_\butcher (\tau), \psi(\pi) \otimes \psi(\gamma) \rangle_\sigma \\
                                                        &= \sum_{(\tau)} \langle \psi^*(\tau_{(1)}) \otimes \psi^*(\tau_{(2)}), \pi \otimes \gamma \rangle_\sigma \\
                                                        &= \sum_{(\tau)} \langle \psi^*(\tau_{(1)}) , \pi \rangle_\sigma \langle \psi^*(\tau_{(2)}), \gamma \rangle_\sigma
                                                        \intertext{let $\tau_{(2)} = B^+(\pi_{(2)})$, then apply property (2) from \cref{lemma:psi*_prop},}
                                                        &= \sum_{(\tau)} \langle \psi^*(\tau_{(1)}) , \pi \rangle_\sigma \langle B^+ \big( \psi^*(\pi_{(2)}) \big), \gamma \rangle_\sigma \\
                                                        &= \sum_{(\tau)} \langle \psi^*(\tau_{(1)}) B^+ \big( \psi^*(\pi_{(2)}) \big), \pi \gamma \rangle_\sigma \,.
    \end{align*}
    This finishes the proof.
\end{proof}

\subsection{Lie--Poisson reduction of Butcher series}
\label{sec:IsoSyRK_expansion}

In this section, we obtain the biplanar Butcher series of ISOSYRK methods from the Butcher series of the corresponding symplectic Runge--Kutta methods by performing Lie--Poisson reduction of Butcher series.

Let $V \in \fraku(n)$ be an element of the Lie algebra, then $V$ acts on $(Q,P) \in T^* \text{U}(n)$ as
\[ V \cdot (Q, P) = (Q V^\dagger, -P V) \,. \]
The lift of a vector field $f : \fraku(n) \to \fraku(n)$ to $T^* \text{U}(n)$ is then given by the vector field,
\[ (\mu_\sharp f) (Q, P) := f(\mu(Q,P)) \cdot (Q,P) = \Big( Q f(Q^\dagger P)^\dagger, -P f(Q^\dagger P) \Big) \,. \]
We extend the action of $\fraku(n)$ on $T^* \text{U}(n)$ to an action of the universal enveloping algebra $\UU$ on $T^* \text{U}(n)$ in the following way,
\begin{align*}
    \one \cdot (Q, P) &= (Q, P) \,, \\
    (V A) \cdot (Q, P) &= V \cdot \big(A \cdot (Q,P) \big) \,.
\end{align*}
where $V \in \fraku(n), A \in \UU, (Q,P) \in T^* \text{U}(n)$ which can be written explicitly as
\[ A \cdot (Q, P) = \big(Q A^\dagger, P S(A)\big) \,, \]
where $S$ is the antipode of the universal enveloping algebra $\UU$, that is,
\[ S(V_1 \cdots V_n) = (-1)^n V_n \cdots V_1 \,, \]
for $V_i \in \fraku(n)$.
We use this extension to define, for $\pi \in T(BT)$,
\[ \mu_\sharp\dF_f (\pi) (Q, P) = \Big( Q \cdot \dF_f(\pi)^\dagger (W), P \cdot \dF_f\big(S(\pi)\big) (W) \Big) \,, \]
where $W = Q^\dagger P$. We note that $\dF_{\mu_\sharp f} (\tau) : T^*\text{U}(n) \to TT^*\text{U}(n)$ for $\tau \in T$ is an elementary differential defined on $T^* \text{U}(n)$, while $\dF_f (\pi) : \fraku(n) \to \UU$ for $\pi \in T(BT)$ is an elementary differential defined on $\fraku(n)$.

Let $k \in \{0, 1, 2\}$, then the directional derivative of the momentum map $\mu : T^* \text{U}(n) \to \fraku(n)$ is denoted by
\[ \dF_{\mu_\sharp f} (\pi) [\mu] := \mu^{(k)} \Big( \dF_{\mu_\sharp f} (\tau_1), \dots, \dF_{\mu_\sharp f} (\tau_k) \Big) \,, \quad \text{for } \pi = \tau_1 \cdots \tau_k \,, \tau_i \in T \,. \]
Let $T(BT)_* := T(BT) \setminus \{\one\}$ be the set of non-empty monomials $\tau_1 \dots \tau_n$ with $\tau_i \in BT$.

\begin{proposition}
    \label{prop:dF_F_relation}
    The elementary differentials $\dF_{\mu_\sharp f}(\tau)$ for $\tau \in T$ and $\dF_f (\pi)$ for $\pi \in T(BT)_*$ are related in the following way,
    \begin{equation}
        \label{eq:dF_F_relation}
        \dF_{\mu_\sharp f} (\tau) = \mu_\sharp\dF_f \big( \psi^* (\tau) \big) \,.
    \end{equation}
    Moreover, for $\pi \in T^2$ and $W = \mu(Q, P) = Q^\dagger P$, we have,
    \begin{equation}
        \label{eq:dF_F_relation_2}
        \dF_{\mu_\sharp f} (\pi) [\mu] (Q,P) = \dF_f \big( \psi^* (\pi) \big) \odot W \,.
    \end{equation}
\end{proposition}
\begin{proof}
    We note that \cref{eq:dF_F_relation_2} follows from \cref{eq:dF_F_relation} and the definition of the action $\odot$ as well as the definition of $\mu_\sharp\dF_f (\pi)$. Therefore, we only need to prove \cref{eq:dF_F_relation}.
    Let $\dF_{\mu_\sharp f}(\tau)_Q$ denote the $Q$-component of $\dF_{\mu_\sharp f}(\tau)$ and $\dF_{\mu_\sharp f}(\tau)_P$ denote the $P$-component of $\dF_{\mu_\sharp f}(\tau)$, then we need to prove the following formulas for $\tau \in T$,
    \[ \dF_{\mu_\sharp f} (\tau)_Q (Q, P) = Q \cdot \dF_f \big(\psi^* (\tau) \big)^\dagger (Q^\dagger P) \,, \quad \text{and} \quad \dF_{\mu_\sharp f} (\tau)_P (Q, P) = P \cdot \dF_f \big(\psi^* (\tau) \big)^\dagger (Q^\dagger P) \,, \]
    where we use the fact that $\dF_f(S(\pi)) = \dF_f(\pi)^\dagger$ for $\pi \in T(BT)$ shown in \cref{eq:F_antipode}.

    We prove $\dF (\tau)_Q = Q \cdot \dF_f \big(\psi^* (\tau) \big)^\dagger$ and the statement for $\dF (\tau)_P$ is proved analogously. We use induction on the size of the $\tau \in T$ with the initial step given by
    \[ \dF_{\mu_\sharp f}(\bullet)_Q (Q,P) = Q \cdot f(Q^\dagger P)^\dagger = Q \cdot \dF_f (\bullet)^\dagger(Q^\dagger P) \,. \]
    Let us assume that the statement holds for all trees with less than $n$ vertices and let $\tau$ be a tree with $n$ vertices. Let $\pi \in F$ be such that $\tau = B^+(\pi)$, then,
    \begin{align*}
        \dF_{\mu_\sharp f} \big(B^+(\pi)\big)_Q &= \dF_{\mu_\sharp f}(\pi) [\dF_{\mu_\sharp f}(\bullet)_Q] = \dF_{\mu_\sharp f}(\pi) \big[ Q \cdot f(Q^\dagger P)^\dagger \big]
    \intertext{we apply the Leibniz rule and use the deshuffle coproduct $\Delta_\shuffle (\pi) = \sum_{(\pi)} \pi_{(1)} \otimes \pi_{(2)}$,}
                                                &= \sum_{(\pi)} \dF_{\mu_\sharp f}(\pi_{(1)})[Q] \cdot \dF_{\mu_\sharp f}(\pi_{(2)}) \big[ f(Q^\dagger P)^\dagger \big]
                             \intertext{use the linearity of $f$, the fact that $Q^\dagger P = \mu(Q, P) =: W$, and \cref{eq:dF_F_relation_2},}
                                                &= \sum_{(\pi)} \dF_{\mu_\sharp f}(\pi_{(1)})[Q] f \Big( \dF_f \big(\psi^* (\pi_{(2)}) \big) \odot W \Big)^\dagger
                                                \intertext{note that the only terms which are non-zero are the ones where $\pi_{(1)} \in \{\one\} \cup T$, for which $\dF_{\mu_\sharp f}(\pi_{(1)})[Q] = \dF_{\mu_\sharp f}(\pi_{(1)})_Q$, use the definition of $\dF_f$,}
                                                &= \sum_{(\pi)} \dF_{\mu_\sharp f}(\pi_{(1)})_Q \dF_f \Big(B^+ \big( \psi^* (\pi_{(2)}) \big) \Big)^\dagger
                                                \intertext{use \cref{eq:dF_F_relation} to obtain,}
                             &= \sum_{(\pi)} Q \cdot \dF_f\big(\psi^*(\pi_{(1)})\big)^\dagger \dF_f \Big(B^+ \big( \psi^* (\pi_{(2)}) \big) \Big)^\dagger \\
                             &= \sum_{(\pi)} Q \cdot \dF_f\Big( \psi^*(\pi_{(1)}) B^+ \big( \psi^* (\pi_{(2)}) \big) \Big)^\dagger
                             \intertext{and we finish by applying the formula for $\psi^*$ from \cref{eq:psi*_1},}
                             &= Q \cdot \dF_f\Big( \psi^* \big( B^+ (\pi) \big) \Big)^\dagger \,.
    \end{align*}
    This proves the statement for $\tau = B^+(\pi)$.
\end{proof}

\begin{figure}[ht]
    \centering

    \begin{tikzpicture}[
        >=Latex,
        node distance=2cm and 3cm,
        every node/.style={font=\normalsize}
    ]

    \node[style={font=\normalsize}] (coeff) at (-3,3.5) {Coefficients};
    \node[style={font=\normalsize}] (integrators) at (0,3.5) {Integrators};
    \node[style={font=\normalsize}] (vecfields) at (3,3.5) {Vector Fields};
    \node[style={font=\normalsize}] (hamiltonians) at (6,3.5) {Hamiltonians};

    \draw[->] (coeff) -- (integrators);
    \draw[<-] (integrators) -- (vecfields);
    \draw[<-] (vecfields) -- (hamiltonians);

    \node (Tstar) at (-3,1.5) {$T^\ast$};

    \node (BHa) at (0,2.8) {$\id + B_{\mu_\sharp f}(a)$};

    \node (TG) at (0,0.8) {$T^\ast G$};

    \node (XTG) at (3,1.5) {$\mathfrak{X}(T^*G)$};

    \node (CinftyTG) at (6,1.5) {$C^\infty(T^\ast G)$};

    \node (BFstar) at (-3,-1.5) {$BF^\ast$};

    \node (BHaphi) at (0,-2.8) {$B_f(a \circ \psi) \odot \blank$};

    \node (gstar) at (0,-0.8) {$\mathfrak{g}^\ast$};

    \node (Xgstar) at (3,-1.5) {$\mathfrak{X}(\mathfrak{g}^\ast)$};

    \node (Cinftygstar) at (6,-1.5) {$C^\infty(\mathfrak{g}^\ast)$};

    \draw[->] (Tstar) -- node[left] {$\psi^\ast$} (BFstar);

    \draw[->] (TG) -- node[right] {$\mu$} (gstar);

    \draw[->] (Cinftygstar) -- node[right] {$\mu^\ast$} (CinftyTG);


    \draw[->] (Cinftygstar) -- (Xgstar);

    \draw[->] (CinftyTG) -- (XTG);

    \draw[->] (Xgstar) -- node[right] {$\mu_\sharp$} (XTG);

    \draw[->] (Tstar) to[out=20,in=180] (BHa);

    \draw[->] (XTG) to[out=160,in=0] (BHa);

    \draw[->] (BFstar) to[out=340,in=180] (BHaphi);

    \draw[->] (Xgstar) to[out=200,in=0] (BHaphi);

    \draw[->] (TG) .. controls +(-3,2) and +(3,2) .. (TG);

    \draw[->] (gstar) .. controls +(-3,-2) and +(3,-2) .. (gstar);

    \end{tikzpicture}

    \caption{Commutative diagram of the momentum maps $\psi$ and $\mu$. Here, $\mu^*H = H \circ \mu$ for a Hamiltonian $H$, and $\psi^*a = a \circ \psi$, where the coefficient map $a$ is extended multiplicatively to $T^2$ by $a(\tau\gamma)=a(\tau)a(\gamma)$. A Lie--Poisson integrator can be obtained in two ways: either by pulling back the Hamiltonian $H:\fraku(n)\to\R$ via $\mu$, applying a symplectic integrator with coefficient map $a\in T^*$, integrating on the symplectic manifold ("upstairs"), and projecting the solution via $\mu$; or by pulling back the coefficient map via $\psi$, thereby constructing the ISOSYRK method from the underlying symplectic Runge--Kutta method, and integrating directly on the Lie--Poisson manifold ("downstairs").}
    \label{fig:commutative_diagram}
\end{figure}
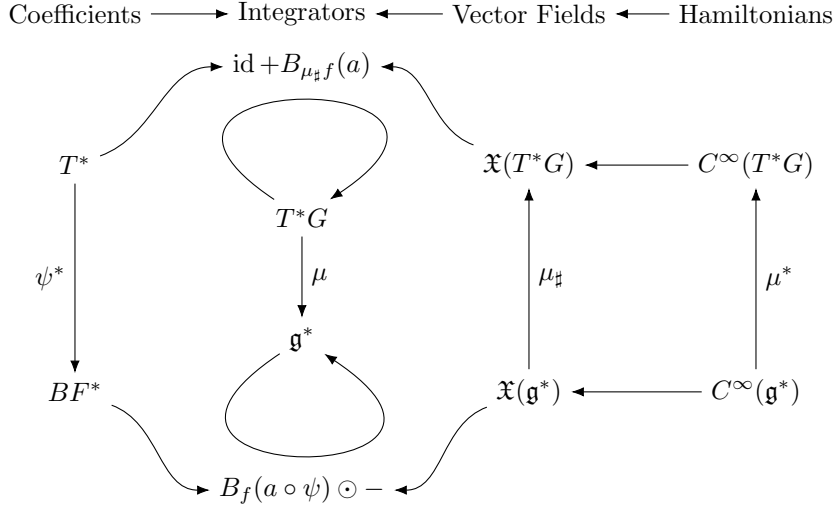

Given a coefficient map $a : T \to \R$, the integrator with the Taylor expansion given by $y_0 + B_{\mu_\sharp f} (a)(y_0)$ is denoted by $\Phi_{\mu_\sharp f}^{(a)}$.
Similarly, given a coefficient map $\alpha : BF \to \R$, the integrator with the Taylor expansion given by $B_f (\alpha) \odot W_0$ is denoted by $\Phi_f^{(\alpha)}$.

\Cref{thm:LP_reduction} shows that the biplanar Butcher series with coefficient map $a \circ \psi$ is related to the Butcher series with coefficient map $a$ via the momentum map $\mu : T^* \text{U}(n) \to \fraku(n)$, and the composition and substitution laws of the biplanar Butcher series follow from the composition and substitution laws of the classical Butcher series.

\begin{theorem}
    \label{thm:LP_reduction}
    Let $a, b : T \to \R$ be coefficient maps extended to $T^2$ as
    \[ a(\one) = 1 \,, a(\tau \gamma) = a(\tau) a(\gamma) \quad \text{and} \quad b(\one) = b(\tau \gamma) = 0 \,. \]
    Then,
    \[ \mu \circ \Phi_{\mu_\sharp f}^{(a)} = \Phi_f^{(a\circ\psi)} \circ \mu \,, \quad B_{\mu_\sharp f} (b) = \mu_\sharp B_f (b \circ \psi) \,. \]
    \emph{Composition} and \emph{substitution} laws follow,
    \begin{align*}
        \Phi_f^{(a\circ\psi)} \circ \Phi_f^{(b \circ \psi)} &= \Phi_f^{\big((b * a) \circ \psi\big)} \,, \\
        \Phi_{\frac1h B_f (b \circ \psi)}^{(a \circ \psi)} &= \Phi_f^{\big((b \star a) \circ \psi\big)} \,.
    \end{align*}
\end{theorem}
\begin{proof}
    Let $\dF_{\mu_\sharp f}(\tau)_Q$ and $\dF_{\mu_\sharp f}(\tau)_P$ denote the $Q$-component and $P$-component of $\dF_{\mu_\sharp f}(\tau)$, respectively, then we need to show that for any $(Q, P) \in T^*U(n)$,
    \[ \mu \big((Q, P) + B_{\mu_\sharp f} (a) (Q,P) \big) = B_f (a \circ \psi) \odot W \,. \]
    We use the definitions of the momentum map $\mu(Q, P) = Q^\dagger P$ and perform the following computation with $W = \mu(Q, P) = Q^\dagger P$,
    \begin{align*}
        \mu \big((Q, P) + B_{\mu_\sharp f} (a) (Q,P) \big) &= \big( Q + \sum_{\tau \in T} \frac{a(\tau)}{\sigma(\tau)} \dF_{\mu_\sharp f}(\tau)_Q \big)^\dagger \big( P + \sum_{\tau \in T} \frac{a(\tau)}{\sigma(\tau)} \dF_{\mu_\sharp f}(\tau)_P \big) \\
                                                   &= W + \sum_{\tau \in T} \frac{a(\tau)}{\sigma(\tau)} \big( Q^\dagger \dF_{\mu_\sharp f}(\tau)_P + \dF_{\mu_\sharp f}(\tau)_Q^\dagger P \big) \\
                                                   &\quad\quad\quad + \sum_{\tau, \gamma \in T} \frac{a(\tau) a(\gamma)}{\sigma(\tau) \sigma(\gamma)} \dF_{\mu_\sharp f}(\tau)_Q^\dagger \dF_{\mu_\sharp f}(\gamma)_P
        \intertext{apply \cref{prop:dF_F_relation},}
                                                &= W + \sum_{\tau \in T} \frac{a(\tau)}{\sigma(\tau)} \big( W \cdot \dF_f \big(\psi^*(\tau))^\dagger + \dF_f \big(\psi^*(\tau)) \cdot W \big) \\
                                                &\quad\quad\quad + \sum_{\tau, \gamma \in T} \frac{a(\tau) a(\gamma)}{\sigma(\tau) \sigma(\gamma)} \dF_f \big( \psi^*(\tau) \big) \cdot W \cdot \dF_f \big( \psi^*(\gamma) \big)
        \intertext{use the definition of the action $\odot$ and biplanar forests,}
                                                &= W + \sum_{\tau \in T} \frac{a(\tau)}{\sigma(\tau)} \dF_f \big(\psi^*(\tau)) \odot W \\
                                                &\quad\quad\quad + \sum_{\tau\gamma \in T^2} \frac{a(\tau \gamma)}{\sigma(\tau \gamma)} \dF_f \big( \psi^*(\tau \gamma) \big) \odot W
        \intertext{use the definition of $\psi^*$ as the adjoint of $\psi$, $\langle \psi^*(\tau \gamma), (\pi, \eta) \rangle_\sigma = \langle \tau \gamma, \psi(\pi, \eta) \rangle_\sigma$, and the fact that the set of biplanar forests $BF$ includes the cases where $\pi = \one$ and $\eta = \one$,}
                                                &= \sum_{(\pi, \eta) \in BF} \frac{a\big(\psi(\pi, \eta)\big)}{\sigma(\pi, \eta)} \dF_f (\pi, \eta) \odot W = B_f (a \circ \psi) \odot W \,.
    \end{align*}
    The identity $B_{\mu_\sharp f} (b) = \mu_\sharp B_f(b \circ \psi)$ follows from the linearity of $\mu_\sharp$ and \cref{prop:dF_F_relation}.
    The composition and substitution laws follow from 
    \[ \Phi_f^{(a \circ \psi)} \circ \Phi_f^{(b \circ \psi)} \circ \mu = \mu \circ \Phi_{\mu_\sharp f}^{(a)} \circ \Phi_{\mu_\sharp f}^{(b)} = \mu \circ \Phi_{\mu_\sharp f}^{(b * a)} = \Phi_f^{\big((b * a) \circ \psi \big)} \circ \mu \,, \]
    and
    \[ \Phi_{\frac1h B_f (b\circ\psi)}^{(a \circ \psi)} \circ \mu = \mu \circ \Phi_{\frac1h \mu_\sharp B_f(b \circ \psi)}^{(a)} = \mu \circ \Phi_{\frac1h B_{\mu_\sharp f} (b)}^{(a)} = \mu \circ \Phi_{\mu_\sharp f}^{(b \star a)} = \Phi_f^{\big((b \star a) \circ \psi\big)} \circ \mu \,. \]
    Since $\mu$ is surjective, the composition and substitution laws follow.
\end{proof}

\begin{corollary}
    \label{cor:order_IsoSyRK}
    Consider a symplectic Runge--Kutta method $\Phi_{\mu_\sharp f}^{(a)}$, then the corresponding ISOSYRK method is $\Phi_f^{(a \circ \psi)}$.
\end{corollary}

We note that the Butcher product on $\TT$ satisfies the relation,
\begin{equation}
    \label{eq:butcher_product_relation}
    \tau \butcher (\gamma \butcher \eta) = \gamma \butcher (\tau \butcher \eta) \,.
\end{equation}
 
We prove that the coefficient map $a \circ \psi$ of biplanar Butcher series is a coadjoint coefficient map as is defined in \cref{prop:coadjoint_coeff}.

\begin{proposition}
    \label{prop:symplecticity}
    The coefficient map $a \circ \psi : BF \to \R$ of an ISOSYRK method $\Phi^{(a \circ \psi)}_f$ is a coadjoint coefficient map.
\end{proposition}
\begin{proof}
    We prove that the coefficient map $\alpha = a \circ \psi$ satisfies the following properties,
    \[ \alpha_T(\one) = 1 \,, \quad \alpha(\pi, \eta) = \alpha_T (\pi) \alpha_T (\eta) \,, \quad \alpha_T \cdot (\alpha_T \circ S) = \delta_\one \,, \]
    with the remaining property $(\alpha_T \circ S) \cdot \alpha_T = \delta_\one$ proven in an analogous way as the third property.
    The first two properties follow from the definition of $\psi$ and the fact that $a$ is extended multiplicatively to $T^2$. To prove the last property, we recall that a coefficient map $a : T \to \R$ corresponding to a symplectic Runge--Kutta method satisfies the following relation,
    \begin{equation}
        \label{eq:symplectic_coeff_condition}
        a (\tau \butcher \gamma) + a(\gamma \butcher \tau) = a(\tau) a(\gamma) \,.
    \end{equation}
    Let us denote $\tilde{\tau}_i = \psi(\tau_i)$ for conciseness and let us apply \cref{eq:symplectic_coeff_condition} repeatedly to obtain the following relation with $n \geq 3$,
    \begin{align*}
        a \big( \psi ( \tau_1 \cdots \tau_n ) \big) &= a\bigg( \psi(\tau_1 \cdots \tau_{n-1}) \butcher \tilde{\tau}_n \bigg)
        \intertext{apply \cref{eq:symplectic_coeff_condition},}
                                                    &= a\big( \psi(\tau_1 \cdots \tau_{n-1}) \big) a \big(\psi(\tau_n)\big) - a\Big( \tilde{\tau}_n \butcher \big( \psi(\tau_1 \cdots \tau_{n-2}) \butcher \tilde{\tau}_{n-1} \big) \Big)
        \intertext{apply \cref{eq:butcher_product_relation} to the second term,}
                                                    &= a\big( \psi(\tau_1 \cdots \tau_{n-1}) \big) a \big(\psi(\tau_n)\big) - a \big( \psi(\tau_1 \cdots \tau_{n-2}) \butcher \psi(\tau_n \tau_{n-1}) \big)
        \intertext{apply \cref{eq:symplectic_coeff_condition} and \cref{eq:butcher_product_relation} again to the second term,}
                                                    &= a\big( \psi(\tau_1 \cdots \tau_{n-1}) \big) a \big(\psi(\tau_n)\big) - a \big( \psi(\tau_1 \cdots \tau_{n-2}) \big) a \big( \psi(\tau_n \tau_{n-1}) \big) \\
                                                    &\quad\quad + a\Big( \psi(\tau_n \tau_{n-1}) \butcher \big( \psi(\tau_1 \cdots \tau_{n-3}) \butcher \tilde{\tau}_{n-2} \big) \Big) \\
                                                    &= a\big( \psi(\tau_1 \cdots \tau_{n-1}) \big) a \big(\psi(\tau_n)\big) - a \big( \psi(\tau_1 \cdots \tau_{n-2}) \big) a \big( \psi(\tau_n \tau_{n-1}) \big) \\
                                                    &\quad\quad + a\big( \psi(\tau_1 \cdots \tau_{n-3}) \butcher \psi(\tau_n \tau_{n-1} \tau_{n-2})\big)
        \intertext{and so on until we obtain the term $a(\psi(\tau_n \cdots \tau_1))$, therefore, we have,}
        a \big( \psi ( \tau_1 \cdots \tau_n ) \big) &= \sum_{k=0}^{n-1} (-1)^{n-k-1} a\big(\psi(\tau_1 \cdots \tau_k)\big) a\big(\psi(\tau_n \cdots \tau_{k+1})\big) \,.
    \end{align*}
    Moving all terms to the left hand side, we obtain the following relation for any $\pi = \tau_1 \cdots \tau_n$ with $n \geq 3$,
    \begin{equation}
        \label{eq:proof_rel_1}
        \sum_{k=0}^{n} (-1)^{n-k} a\big(\psi(\tau_1 \cdots \tau_k)\big) a\big(\psi(\tau_n \cdots \tau_{k+1})\big) = 0 \,.
    \end{equation}
    For $n = 2$, the relation \cref{eq:proof_rel_1} also holds by the symplectic condition \cref{eq:symplectic_coeff_condition}. For $n = 1$, the relation holds trivially. Therefore, the relation \cref{eq:proof_rel_1} holds for any $\pi \in T, \pi \neq \one$ and the statement follows.
\end{proof}

\subsection{Runge--Kutta--Munthe-Kaas methods for Lie--Poisson systems}
\label{sec:RKMK}

In \cite{EngoNIL01}, the authors proposed a class of numerical integrators for Lie--Poisson systems which are obtained by using the Runge--Kutta--Munthe-Kaas (RKMK) framework \cite{Munthe-KaasHOR99,Munthe-KaasRML98}. In this section, we show that RKMK methods applied to the isospectral equation \cref{eq:isospectral} can be expanded using biplanar Butcher series by introducing a subclass of coadjoint coefficient maps called \emph{character coefficient maps}.

The solution of the isospectral equation \cref{eq:isospectral} evolves on the coadjoint orbit $\OO_{W_0}$ defined as
\[ \OO_{W_0} = \{ \Ad^*_Q W_0 \; : \; Q \in \text{U}(n) \} \,, \]
in particular, the solution can be written as $W(t) = \Ad^*_{\exp(\omega(t))} W_0$ for some curve $\omega : \R \to \fraku(n)$ with $\omega(0) = 0$ which solves,
\[ \dot{\omega}(t) = \dexp^{-1}_{\omega(t)} \Big( f\big(W(t)\big) \Big) \,, \quad \text{where } \dexp^{-1}_u (v) = \sum_{k=0}^\infty \frac{B_k}{k!} \ad^k_u (v) \,, \]
with $\ad^k_u (v) = [u, \ad^{k-1}_u (v)]$, $\ad_u^0 (v) = v$, and $B_k$ are the Bernoulli numbers. Therefore, the solution of the isospectral equation can be obtained by solving a differential equation on $\fraku(n)$ and applying the coadjoint action of $\exp\big(\omega(t)\big)$ to $W_0$.

This perspective is used in \cite{EngoNIL01} to propose numerical integrators of the form,
\begin{align*}
    W_{n+1} &= \Ad^*_{\exp(\omega_{n+1})} W_n \,, \\
    \omega_{n+1} &= h \sum_{i=1}^s b_i \dexp^{-1}_{\omega_n^i} \big( f(W_n^i) \big) \,, \\
    W^i_n &= \Ad^*_{\exp(\omega_n^i)} W_n \,, \\
    \omega_n^i &= h \sum_{j=1}^s a_{ij} \dexp^{-1}_{\omega_n^j} \big( f(W_n^j) \big) \,,
\end{align*}
where $b = (b_i)_{i=1}^s$ and $A = (a_{ij})_{i,j=1}^s$ are coefficients of a Runge--Kutta method of order $p$. The resulting integrator $W_n \mapsto W_{n+1}$ is of order $p$.

Let us consider the shuffle product $\shuffle : T(\BT) \otimes T(\BT) \to T(\BT)$ defined as,
\[ \tau\pi \shuffle \gamma\eta := \tau (\pi \shuffle \gamma\eta) + \gamma (\tau\pi \shuffle \eta) \,, \quad \text{for } \tau, \gamma \in BT, \ \  \pi, \eta \in T(BT) \,, \]
for example, $\tau_1 \tau_2 \shuffle \tau_3 \tau_4 = \tau_1 \tau_2 \tau_3 \tau_4 + \tau_1 \tau_3 \tau_2 \tau_4 + \tau_1 \tau_3 \tau_4 \tau_2 + \tau_3 \tau_1 \tau_2 \tau_4 + \tau_3 \tau_1 \tau_4 \tau_2 + \tau_3 \tau_4 \tau_1 \tau_2$.

We introduce a subclass of coadjoint coefficient maps, see \cref{prop:coadjoint_coeff}, that correspond to RKMK methods.

\begin{proposition}
    \label{prop:character_coeff}
    Consider a RKMK method $W_n \mapsto W_{n+1}$ for the isospectral equation \cref{eq:isospectral}. It can be expanded as a Butcher series $B_f (\alpha)$ over biplanar forests with coadjoint coefficient map $\alpha : BF \to \R$ satisfying the following property,
    \[ \alpha_T (\pi \shuffle \eta) = \alpha_T (\pi) \alpha_T (\eta) \,, \quad \text{for } \pi, \eta \in T(BT) \,. \]
    Such coefficient maps are called \emph{character coefficient maps}.
\end{proposition}
\begin{proof}
    We note that $\omega_{n+1}$ and $\omega_n^i$ are linear combinations of terms of the form $\ad^k_{\omega_n^j} \big( f(W_n^j) \big)$ for $k \in \N$. This implies that $\omega_{n+1}$ and $\omega_n^i$ are linear combinations of terms of the form $\dF_f(\tau)$ for $\tau \in \Prim\big(T(\BT)\big)$, where $\Prim\big(T(\BT)\big)$ is the Lie algebra generated by biplanar trees.
    We recall that $W_{n+1} = \Ad^*_{\exp(\omega_{n+1})} W_n$ which implies that $B_f (\alpha_T) = \exp(\omega_{n+1})$. Since $\omega_{n+1} \in \dF_f \Big(\Prim\big(T(\BT)\big)\Big)$, $\alpha_T$ is a character with respect to the shuffle product $\shuffle$ on $T(\BT)$.
    We check that the character property of $\alpha_T$ implies $\alpha_T \cdot (\alpha_T \circ S) = \delta_\one$. For any $\one \neq \pi \in T(\BT)$,
    \[ \alpha_T \cdot (\alpha_T \circ S) (\pi) = \sum_{(\pi)} \alpha_T(\pi_{(1)}) \alpha_T\big(S(\pi_{(2)})\big) = \sum_{(\pi)} \alpha_T\big(\pi_{(1)} \shuffle S(\pi_{(2)})\big) = 0 \,, \]
    following the property of the antipode $S$ where $\Delta_\cdot (\pi) = \sum_{(\pi)} \pi_{(1)} \otimes \pi_{(2)}$. The property $\alpha_T \cdot (\alpha_T \circ S) = \delta_\one$ follows analagously.
\end{proof}

Note that the converse of \cref{prop:character_coeff} does not hold. Indeed, given a character coefficient map $\alpha : BF \to \R$, the corresponding method is not necessarily an RKMK method. The reason is that the space $\dF_f \Big(\Prim\big(T(\BT)\big)\Big)$ is much larger than the subspace accessible by $\omega_{n+1}$ and $\omega^i_n$. In other words, the stages of an RKMK method probe only a restricted family of primitive elements.

\begin{corollary}
    \label{cor:exact_coeff}
    Let $\hat{\gamma}$ denote the ordered tree factorial defined in \cite{munthe-kaasLieButcherSeries2018}, then $\hat{\gamma} = \gamma \circ \psi$ where $\gamma$ is the tree factorial over classical trees $T$.
\end{corollary}

\begin{corollary}
    ISOSYRK methods are not part of the RKMK class of methods for Lie--Poisson systems as considered in \cite{EngoNIL01}.
\end{corollary}
\begin{proof}
    The coefficient maps of ISOSYRK methods are given by $a \circ \psi$ where $a : T \to \R$ is a coefficient map of a symplectic Runge--Kutta method. By \cref{prop:character_coeff}, for the method to be an RKMK method, the coefficient map $a \circ \psi$ must be a character with respect to the shuffle product on $T(\BT)$. In particular, $a(\psi(\bullet^n)) = 1/n!$ for any $n \in \N$ assuming $a(\psi(\bullet)) = 1$. Now, $\psi(\bullet^n)$ is the rooted tree consisting of a single branch of height $n$, that is,
    \[ \psi(\bullet^n) = \forestJD \,, \quad \text{hence,} \quad a(\forestKD) = 1/n! \,. \]
    Therefore, the corresponding symplectic Runge--Kutta method must be exact for linear problems. However, no finite-stage Runge--Kutta method can be exact for general linear problems, see \cite[Chapter IV.3]{HairerSOD96}. This yields a contradiction.
\end{proof}

\section{Backward error analysis over biplanar forests}\label{sec:bea_biplanar_forests}

In this section, we study the geometric properties of ISOSYRK methods, in particular, we perform backward error analysis and obtain the modified equation as a Butcher series over biplanar forests. We also give an explicit formula for the modified Hamiltonian by introducing \emph{fair biplanar trees}.

Let us use the substitution law described in \cref{thm:LP_reduction} to perform the backward error analysis. We recall that the backward error analysis of a Runge--Kutta method consists in finding a modified vector field $\tilde{f}_h$ such that the exact flow of $\tilde{f}_h$ coincides with the numerical flow of the Runge--Kutta method.
The modified vector field $\tilde{f}_h$ is expressed as a Butcher series with coefficients given by the substitution law. In particular, we find the coefficient map $b: T \to \R$ such that,
\[ B_f (a) = B_{\frac1h B_f (b)} (1/\gamma) = B_f \big(b \star (1/\gamma)\big) \,. \]
Following \cite{calaqueTwoInteractingHopf2011}, we can endow the set of coefficient maps together with the substitution law as a product with a group structure, therefore, $b = a \star (1/\gamma)^{\star-1}$. Using \cref{thm:LP_reduction} and given a coefficient map $b$ computed for a symplectic Runge--Kutta method, the coefficient map of the modified vector field of the corresponding ISOSYRK method is given by $b \circ \psi$, that is,
\[ B_f (a \circ \psi) = B_{\frac1h B_f (b \circ \psi)} \big( (1/\gamma) \circ \psi \big) \,. \]

We note that the coefficient map $b$ satisfies the identity,
\begin{equation}
    \label{eq:Hamiltonian_vf}
     b(\tau \butcher \gamma) + b(\gamma \butcher \tau) = 0 \,,
 \end{equation}
 \Cref{prop:modified_isospectrality} uses this identity to show that the modified vector field of an ISOSYRK method is isospectral.

\begin{proposition}
    \label{prop:modified_isospectrality}
    Let $b : T \to \R$ be a coefficient map which is extended to $T^2$ as $b(\one) = b(\tau \gamma ) = 0$ for any $\tau, \gamma \in T$, and, additionally, satisfies the identity,
    \[ b(\tau \butcher \gamma) + b(\gamma \butcher \tau) = 0 \,, \quad \text{for any } \tau, \gamma \in T \,. \]
    Then, $b \circ \psi$ is an infinitesimal coadjoint coefficient map, see \cref{prop:infinitesimal_coadjoint_coeff}.
\end{proposition}
\begin{proof}
    We reuse the argument of the proof of \cref{prop:symplecticity} to show that $b \circ \psi \circ S = - b \circ \psi$ where we replace the deconcatenation coproduct $\Delta_\cdot (\pi)$ by $\pi \otimes \one + \one \otimes \pi$ due to the difference between the symplecticity condition on a coefficient map $a : T \to \R$ and the identity satisfied by $b: T \to \R$. This proves $b \circ \psi + b \circ \psi \circ S = 0$.
\end{proof}

We recall \cref{sec:bea_symplectic} in which the identity \cref{eq:Hamiltonian_vf} is used to introduce non-rooted trees and the set of non-superfluous non-rooted trees $FT^\prime$. Non-superfluous non-rooted trees are then used to represent the modified vector field $\tilde{f}_h$. This approach is used to dramatically decrease the amount of computation necessary to compute the values of the coefficient map $b : T \to \R$ of the modified vector field as the set of non-superfluous non-rooted trees is much smaller than the set of rooted trees.

We do not represent the biplanar Butcher series of the modified equation using a subset of biplanar forests corresponding to the non-superfluous non-rooted trees since the values of the coefficient map $b \circ \psi$ are already given by the values of the coefficient map $b: T \to \R$. However, we introduce such a set in the next section to obtain an explicit biplanar Butcher series representation of the modified Hamiltonian of an ISOSYRK method.

\subsection{Modified Hamiltonian}
\label{sec:biplanar_elementary_Hamiltonians}

Recall \cref{thm:mod_Hamiltonian} in which the modified Hamiltonian of a symplectic Runge--Kutta method is expressed as
\begin{equation}
    \label{eq:modified_Hamiltonian_series}
    \tilde{\HH}_h (Q, P) = \sum_{\hat{\tau} \in FT^\prime} h^{|\hat\tau|-1} \frac{b(\hat{\tau}_*)}{\sigma(\hat{\tau}_*)} \HH(\hat{\tau}_*)(Q, P) \,,
\end{equation}
with appropriately defined coefficient map $b: T \to \R$ where elementary Hamiltonian $\HH(\hat{\tau}_*)$ is given by
\[ \HH(\hat{\tau}_*) (Q, P) = \dF_{\mu_\sharp f} (\pi) [\HH] (Q, P) \,, \quad \text{for } \hat{\tau}_* = B^+(\pi) \,. \]

Similarly to how the directional derivatives of the momentum map $\mu$ give rise to biplanar forests, the directional derivatives of the Hamiltonian $\HH$ give rise to the space $\BF^{\leq 2} := \BF \oplus \BF^2$ which is a subspace of the symmetric algebra $S(\BF)$. Let the basis of $\BF^2$ be denoted by $(\pi_1, \eta_1) \cdot (\pi_2, \eta_2)$ with $(\pi_i, \eta_i) \in BF$. We have,
\[ (\pi_1, \eta_1) \cdot (\pi_2, \eta_2) = (\pi_2, \eta_2) \cdot (\pi_1, \eta_1) \,, \quad \text{for } (\pi_i, \eta_i) \in BF \,. \]
We extend the forest momentum map $\psi : BF \to T^2$ to $\psi : BF^2 \to T^4$ by defining for $(\pi_i, \eta_i) \in BF$,
\[ \psi\big( (\pi_1, \eta_1) \cdot (\pi_2, \eta_2) \big) = \psi(\pi_1, \eta_1) \psi(\pi_2, \eta_2) \,,  \]
and $\sigma : BF \to \N$ to $\sigma : BF^2 \to \N$ by,
\[ \sigma \big( (\pi_1, \eta_1) \cdot (\pi_2, \eta_2) \big) = \begin{cases} 2 \sigma(\pi_1, \eta_1)^2 \,, \quad \text{if } (\pi_1, \eta_1) = (\pi_2, \eta_2) \,, \\ \sigma(\pi_1, \eta_1) \sigma(\pi_2, \eta_2) \,, \quad \text{otherwise} \,. \end{cases} \]
Let us also define $B^+$ on $BF^2$ with $\psi$ commuting with $B^+$, that is,
\[ \psi \circ B^+ = B^+ \circ \psi \,, \quad \text{and, moreover,} \quad \sigma \circ B^+ = \sigma \,.  \]

\begin{definition}
    \label{def:fair_biplanar_trees}
    A tree $B^+(\pi)$ with $\pi \in BF^{\leq 2}$ is called a \emph{fair biplanar tree} if $\psi \big( B^+(\pi) \big) = \hat{\tau}_*$ for some $\hat{\tau} \in FT^\prime$.
\end{definition}

The set of fair biplanar trees is denoted by $FBT$.
The name \emph{fair} is motivated by the fact that the canonical representatives $\hat{\tau}_*$ of non-rooted trees in $FT^\prime$ are maximal with respect to a total order \cite{MuruaHAR06,bogfjellmoHamiltonianBseriesLie2017} according to which a tree is bigger if it has more branches with each branch having a similar number of vertices. This corresponds to the forest $(\pi_1, \eta_1) \cdot (\pi_2, \eta_2) \in BF^{\leq2}$ having $\pi_1, \pi_2, \eta_1, \eta_2$ with the same number of vertices if possible. A systematic study of fair biplanar trees is left for future work. All non-empty fair biplanar trees up to order $4$ are listed below, where we separate $(\pi_1, \eta_1)$ and $(\pi_2, \eta_2)$ in $(\pi_1, \eta_1) \cdot (\pi_2, \eta_2)$ by $\otimes$,
\[ \forestLD \,, \quad \forestMD \,, \quad\forestND \,, \quad \forestOD \,,  \]
\[ \forestPD \,, \quad \forestQD \,, \quad\forestRD \,, \quad \forestSD \,,  \]
\[ \forestTD \,, \quad \forestUD \,, \quad \forestVD \,, \quad \forestWD \,. \]
\[ \forestXD \,, \quad \forestYD \,. \]

Let $\pi \in F$ be a forest and $P(\pi)$ denote the set of partitions of trees in $\pi$. Let $P_{\leq2}(\pi)$ be the subset of $P(\pi)$ consisting of partitions with at most two parts.
The adjoint of $\psi : BF^{\leq 2} \to T^4$ is denoted by $\psi^* : T^4 \to BF^{\leq 2}$ and is given by,
\[ \psi^* (\pi) = \sum_{\substack{p \in P_{\leq2}(\pi) \\ p_i \in T^2}} \psi^*(p_1) \psi^*(p_2) \,, \]
where $\pi \in T^4$ and $p_2 = \one$ if $p$ has only one part. We also have $B^+ \circ \psi^* = \psi^* \circ B^+$.
We note that the set $\{ p \in P_{\leq2} (\pi) \; | \; p_i \in T^2 \}$ has at most $3$ elements, for example, given $\pi = \tau_1 \cdots \tau_4$,
\[ \{ p \in P_{\leq2} (\pi) \; | \; p_i \in T^2 \} = \big\{ \{ \tau_1 \tau_2, \tau_3 \tau_4 \} \,, \{ \tau_1 \tau_3, \tau_2 \tau_4 \} \,, \{ \tau_1 \tau_4, \tau_2 \tau_3 \} \big\} \,. \]

The elementary Hamiltonian $H\big( B^+(\pi) \big)$ for $\pi \in BF^{\leq 2}$ is defined as follows,
\begin{align*}
    H \big(B^+ ( \pi, \eta ) \big) (W) &:= \dF_f (\pi, \eta) [H] (W) = H^\prime (W) \big(\dF_f(\pi, \eta) \odot W \big) \,, \\
    H \Big(B^+ \big( (\pi_1, \eta_1) \cdot (\pi_2, \eta_2) \big) \Big) (W) &:= \dF_f \big((\pi_1, \eta_1) \cdot (\pi_2, \eta_2)\big) [H] (W) \\
                                                                       &= H^{(2)} (W) \big(\dF_f(\pi_1, \eta_1) \odot W \,, \ \dF_f (\pi_2, \eta_2) \odot W\big) \,.
\end{align*}

\begin{proposition}
    \label{prop:isospectral_modified_Hamiltonian}
    Let $\tilde{\HH}_h : T^*\text{U}(n) \to \R$ be the modified Hamiltonian of a symplectic Runge--Kutta method, then,
    \[ \tilde{\HH}_h = \tilde{H}_h \circ \mu \,, \]
    where $\tilde{H}_h : \fraku(n)^* \to \R$ is the modified Hamiltonian of the corresponding ISOSYRK method. Moreover, the modified Hamiltonian $\tilde{H}_h$ of the ISOSYRK method is given by the series expansion,
    \begin{equation}
        \label{eq:biplanar_modified_Hamiltonian_series}
        \tilde{H}_h (W) = \sum_{\hat\tau \in FBT} h^{|\hat\tau|-1} \frac{b\big(\psi(\hat\tau)\big)}{\sigma(\hat{\tau})} H(\hat\tau)(W) \,.
    \end{equation}
    where $b : T \to \R$ is the coefficient map of the modified vector field of the corresponding symplectic Runge--Kutta method.
\end{proposition}
\begin{proof}
    Note that $\HH = H \circ \mu$ and, therefore, the elementary Hamiltonian $\HH(\tau)$ in \cref{eq:modified_Hamiltonian_series} has the form, for $\HH(\hat{\tau}_*) = \HH\big(B^+(\pi)\big)$,
    \[ \HH\big( B^+(\pi) \big) (Q, P) = \dF_{\mu_\sharp f} (\pi) \big[ H \circ \mu \big] (Q, P) = \sum_{p \in P(\pi)} \Big( \prod_{p_i} \dF_{\mu_\sharp f} (p_i) [\mu] \Big) [H] (Q, P) \,, \]
    using the Fa\`a di Bruno's formula where $P(\pi)$ is the set of partitions of the set of trees in $\pi \in F$. We recall that the vector field $f$ is linear, therefore, the Hamiltonian $H$ is quadratic. This implies that the only non-zero terms remaining are those that correspond to partitions $p \in P(\pi)$ with at most two parts. Moreover, by using \cref{prop:dF_F_relation}, we obtain, 
    \[ \HH\big( B^+(\pi) \big) (Q, P) = \sum_{p \in P_{\leq2}(\pi)} \Big( \prod_{p_i} \dF_f \big(\psi^*(p_i)\big) \odot W \Big) [H] (W) \,. \]
    where $W = \mu(Q, P) = Q^\dagger P$. Since $\psi^* (p_i) = 0$ if $p_i \notin T^2$ we get,
    \[ \HH\big( B^+(\pi) \big) (Q, P) = \sum_{\substack{p \in P_{\leq2}(\pi) \\ p_i \in T^2}} \Big( \prod_{p_i} \dF_f \big(\psi^*(p_i)\big) \odot W \Big) [H] (W) = H \Big( \psi^* \big(B^+(\pi)\big) \Big) (W) \,. \]
    This implies that the modified Hamiltonian $\tilde{\HH}_h$ over $T^*\text{U}(n)$ of a symplectic Runge--Kutta method can be written as a modified Hamiltonian $\tilde{H}_h$ on $\fraku(n)^*$ of the corresponding ISOSYRK method precomposed with the momentum map $\mu(Q, P) = W$, that is,
    \[ \tilde{\HH}_h = \tilde{H}_h \circ \mu \,. \]
    We verify that $\tilde{H}_h : \fraku(n)^* \to \R$ is the Hamiltonian of the ISOSYRK method by following the computation,
    \begin{align*}
        \tilde{H}_h (W) &= (\tilde{H}_h \circ \mu) (Q, P) = \big( \tilde{H}_h \circ \mu \circ \Phi_{\mu_\sharp f}^{(a)} \big) (Q, P) \\
                        &= \big( \tilde{H}_h \circ \Phi_f^{(a \circ \psi)} \circ \mu \big) (Q, P) = \big( \tilde{H}_h \circ \Phi_f^{(a \circ \psi)} \big) (W) \,.
    \end{align*}
    This finishes the proof.
\end{proof}

This gives us an efficient way to compute the modified Hamiltonian of an ISOSYRK method. Using \cref{prop:isospectral_modified_Hamiltonian}, the modified Hamiltonian $\tilde{H}_h$ can be written as,
\begin{align*}
    \tilde{H}_h = H &+ h^2 b(\forestAE) \dF_f \Big( \frac1{2}  \forestBE + \frac1{2} \forestCE \Big)[H] + h^3 b(\forestDE) \dF_f \Big( \frac1{2} \forestEE\Big) [H] \\
                         &+ h^4 b(\forestFE) \dF_f \Big( \frac1{2} \forestGE + \forestHE + \frac1{2} \forestIE \\
                         &\quad \quad \quad \quad \quad \ \  + \frac1{2} \forestJE + \forestKE + \frac1{2} \forestLE \Big) [H] \\
                         &+ h^4 b(\forestME) \dF_f \Big( \forestNE + \forestOE + \frac1{2} \forestPE + \frac1{2} \forestQE \Big) [H] \\
                         &+ h^4 b(\forestRE) \dF_f \Big(\frac1{8}\forestSE\Big) [H] + \OO(h^5) \,.
\end{align*}

See \cref{sec:numerics} for an explicit computation of the modified Hamiltonian of the ISOMP method applied to the Euler-Zeitlin equation.

\subsection{Analytic bounds}
\label{sec:long_time_behavior}

Here we derive analytic bounds concerning the long-time behavior of the numerical flow of ISOSYRK methods.
The procedure from \cite[Ch. IX.7]{HairerWannerGNI} assumes that the trajectory stays on a compact domain and uses Cauchy estimates to bound the derivatives of the vector flow of the problem. 
In our setting, we drop the compactness assumption and replace it by the assumptions
\begin{equation}
    \label{eq:bound_f}
    \lVert f(W) \rVert_\infty \leq C_f \rVert W \rVert_\infty \,, \quad \lVert f^\prime(W) \rVert_\infty \leq C_f \,,
\end{equation}
for some constant $C_f$ independent of $W$. 
Additionally, we assume for $a: T \to \R$ that
\begin{equation}
    \label{eq:bound_coeff}
    | a(\tau) | \leq C_a^{|\tau|} \,,
\end{equation}
for some constant $C_a$ independent of $\tau$.

First, we need a bound on the number of ordered monomials of biplanar trees $T(BT)_j$ and biplanar forests $BF_j$ of size $j \in \N$.

\begin{lemma}
    \label{lem:combinatorial_bounds}
    The sizes of the sets $T(BT)_j$ and $BF_j$ of ordered monomials of biplanar trees and biplanar forests of size $j$ are bounded as,
    \[ |T(BT)_j| \leq \frac13 9^j \quad \text{and} \quad |BF_j| \leq 9^j \,. \]
\end{lemma}
\begin{proof}
    The bound for $|T(BT)_j|$ follows from \cref{lem:combinatorial_bounds_classic} and the fact that every element of $T(BT)_j$ appears as a term in $\psi^*(\tau)$ for some $\tau \in T_j$ with $\psi^*(\tau)$ having at most $3^{|\tau|}$ terms, therefore, $|T(BT)_j| \leq \frac13 3^j \cdot 3^j = \frac13 9^j$. The bound for $|BF_j|$ follows from a similar argument where we consider the set of forests $F_j$ of size $j$ and $|F_j| = |T_{j+1}| \leq 3^j$.
\end{proof}

Next, we expand the ISOSYRK methods $\Phi_h$ as
\[ \Phi_h (W) = W + h f(W) + h^2 d_2 (W) + h^3 d_3 (W) + \cdots \,. \]

\begin{proposition}
    \label{prop:bound_dj}
    Let $\pi \in T(BT)$ and $(\pi, \eta) \in BF$, then, the elementary differentials $\dF_f(\pi)$ and $\dF_f(\pi, \eta)$ are bounded as,
    \begin{align*}
        \| \dF_f (\pi)(W) \|_\infty &\leq \big( 2 C_f \|W\|_\infty \big)^{|\pi|} \,,\\
        \| \dF_f (\pi, \eta) \odot W \|_\infty &\leq 2 \|W\|_\infty \big( 2 C_f \|W\|_\infty \big)^{|\pi| + |\eta|} \,.
    \end{align*}
    Moreover, for a given $j \in \N$, we have
    \[ \| d_j(W) \|_\infty \leq 2 \|W\|_\infty \big( 18 C_a C_f \|W\|_\infty \big)^j. \]
\end{proposition}
\begin{proof}
The bounds of $\dF_f(\pi)$ and $\dF_f(\pi,\eta)$ follow from the definition of $\dF_f$ and properties \cref{eq:bound_f} and the multiplicativity of the infinite norm, in paraticular, given $\pi, \eta \in T(BT)$, we have,
    \begin{align*}
        \lVert \dF_f (\pi, \eta) \odot W \rVert_\infty &\leq \lVert \dF_f (\pi) \cdot W \cdot \dF_f (\eta)^\dagger + \dF_f (\eta) \cdot W \cdot \dF_f (\pi)^\dagger \rVert_\infty \\
                                                            &\leq 2 \lVert \dF_f (\pi) \rVert_\infty \lVert \dF_f (\eta)  \rVert_\infty \lVert W \rVert_\infty \,, \\
                \lVert \dF_f (\pi \cdot \eta) \rVert_\infty &\leq \lVert \dF_f (\pi) \rVert_\infty \lVert \dF_f (\eta) \rVert_\infty \,, \\
                \lVert \dF_f \big( B^+ (\pi, \eta) \big) \rVert_\infty &\leq C_f \lVert \dF_f (\pi, \eta) \odot W \rVert_\infty \,.
    \end{align*}

    Let $\Phi_h$ denote the ISOSYRK method corresponding to a symplectic Runge--Kutta method with coefficient map $a : T \to \R$. We use the fact that $\Phi_h(W) = B_f(a \circ \psi) \odot W$ to write
    \begin{align*}
        \|d_j (W)\|_\infty &= \Big\| \sum_{(\pi, \eta) \in BF_j} \frac{(a \circ \psi)(\pi, \eta)}{\sigma(\pi, \eta)} \dF_f (\pi, \eta) \odot W \Big\|_\infty \\
                           &\leq |BF_j| \cdot C_a^j \cdot 2 \|W\|_\infty \big( 2 C_f \|W\|_\infty \big)^j \leq 2 \|W\|_\infty \big( 18 C_a C_f \|W\|_\infty \big)^j \,,
    \end{align*}
    where we use \cref{lem:combinatorial_bounds}.
\end{proof}

Let the modified vector field for an ISOSYRK method be written as
\[ \tilde{f}_h (W) = f(W) + h f_2 (W) + h^2 f_3(W) + \cdots \,. \]

\begin{proposition}
    The term $f_j(W)$ of the modified vector field $\tilde{f}_h$ of the ISOSYRK method corresponding to a symplectic Runge--Kutta method with coefficient map $a$ is bounded as,
    \[ \| f_j (W) \|_\infty \leq \frac{e}3 (j-1)! \big( 18 C_a C_f \|W\|_\infty \big)^j \,. \]
\end{proposition}
\begin{proof}
    We note that $\tilde{f}_h (W) = B_f (b \circ \psi) \odot W$, therefore, we have,
    \[ f_j (W) = \sum_{\pi \in T(BT)_j} \frac{(b\circ\psi) (\pi)}{\sigma(\pi)} \dF_f(\pi) \,. \]
    This implies the following bound for $f_j(W)$,
    \[ \| f_j (W) \|_\infty \leq |T(BT)_j| \cdot e (j-1)! \cdot \big( 2 C_a C_f \|W\|_\infty \big)^j \leq \frac{e}3 (j-1)! \big(18 C_a C_f \|W\|_\infty\big)^j \,, \]
    where we use \cref{lem:combinatorial_bounds} to bound $|T(BT)_j|$ and \cref{lem:bound_b} to bound $|b\big(\psi(\pi)\big)|$ for $\pi \in T(BT)_j$.
\end{proof}

Let us now consider the truncated modified vector field $\tilde{f}_h^N$,
\[ \tilde{f}_h^N (W) = f(W) + h f_2 (W) + \cdots + h^{N-1} f_N(W) \,, \]
and its flow $\tilde{\varphi}_h^N$ for $N \in \N$. We bound $\|\tilde{f}^N_h(W)\|_\infty$ as follows,
\begin{align*}
    \|\tilde{f}^N_h(W)\|_\infty &\leq \|f(W)\|_\infty + \sum_{j=2}^N h^{j-1} \|f_j(W)\|_\infty \\
                             &\leq C_f \|W\|_\infty + \sum_{j=2}^N h^{j-1} \frac{e}3 (j-1)! \big( 18 C_a C_f \|W\|_\infty \big)^j \,.
\end{align*}
Similarly to the discussion in \cite[Ch. IX.7.3]{HairerWannerGNI}, we introduce
\[ h_0 := \frac1{18 e C_a C_f} \,, \]
and assume $N \leq \frac{e h_0}{h\|W\|_\infty}$. We obtain,
\begin{equation}
    \label{eq:bound_truncated_modeq}
    \|\tilde{f}^N_h(W)\|_\infty \leq C_f \|W\|_\infty + 6 e C_a C_f \|W\|_\infty \Big( \sum_{j=2}^N \frac{(j-1)!}{N^{j-1}} \Big) \leq C_f \|W\|_\infty \big( 1 +  6 e C_a \big) \,.
\end{equation}

\begin{theorem}
    \label{thm:bound_modflow}
    Assume $h$ and $N$ are chosen such that $2 \leq N \leq h_0 / (h \|W\|_\infty)< N+1$, then, the numerical flow $\Phi_h$ of an ISOSYRK method and the flow $\tilde{\varphi}^N_h$ of the truncated modified vector field $\tilde{f}^N_h$ satisfy the following bound,
    \[ \| \Phi_h (W) - \tilde{\varphi}^N_h (W) \|_\infty < h C(\|W\|_\infty) C_f \|W\|_\infty e^{-h_0/(h\|W \|_\infty)} \,, \]
    where $C(r) = 2 e (1 + 3e C_a + 18 C_a r)$.
\end{theorem}
\begin{proof}
    We adapt the proof of Theorem IX.7.6 from \cite{HairerWannerGNI} to our setting. We define,
    \[ g(h) := \Phi_h (W) - \tilde{\varphi}^N_h (W) \,, \]
    which is an analytic function in $h$ in the neighborhood of $h=0$ due to the analyticity of $\Phi_h$ and $\tilde{\varphi}^N_h$. Due to the definition of $\tilde{\varphi}^N_h$, $\Phi_h$ and $\tilde{\varphi}^N_h$ agree up to term $h^N$, therefore, $g(h)$ has a factor $h^{N+1}$ and maximum modulus principle implies that,
    \[ \| g(h) \|_\infty \leq \Big(\frac{h}{\epsilon}\Big)^{N+1} \max_{|z| \leq \epsilon} \| g(z) \|_\infty \,, \quad \text{for } 0 \leq h \leq \epsilon \,, \]
    if $g(z)$ is analytic in the disk of radius $\epsilon := e h_0 / (N\|W\|_\infty)$ around $0$. 
    We estimate $\|g(z)\|_\infty$ for $|z| \leq \epsilon$ by estimating separately $\|\Phi_z (W) - W\|_\infty$ and $\|\tilde{\varphi}^N_z (W) - W\|_\infty$.
    
    The term $\Phi_z (W)$ converges if $|z|^j \|d_j(W)\|_\infty \leq c (1/2)^j$ for some constant $c$. 
    Therefore, using \cref{prop:bound_dj}, $\Phi_z(W)$ converges for,
    \[ |z| \leq \epsilon \leq \frac{e h_0}{2 \|W\|_\infty} = \frac1{36 C_a C_f \|W\|_\infty} \,, \]
    where we use the assumption $N \geq 2$. Moreover, we obtain,
    \begin{align*}
        \| \Phi_z (W) - W \|_\infty &\leq |z| \Big( \|f(W)\|_\infty + \sum_{j=2}^\infty |z|^{j-1} \|d_j(W)\|_\infty \Big) \\
        \intertext{use the bounds on $|z|$, $\|f(W)\|_\infty$, and $\|d_j(W)\|_\infty$ from \cref{prop:bound_dj},}
                                    &\leq |z| \Big( C_f \|W\|_\infty + 36 C_a C_f \|W\|^2_\infty \sum_{j=2}^\infty \frac1{2^{j-1}} \Big) \\
                                    &\leq |z| C_f \|W\|_\infty (1 + 36 C_a \|W\|_\infty) \,.
    \end{align*}
    Analogously, we bound $\|\tilde{\varphi}^N_z (W) - W \|_\infty$ using,
    \[ \|\tilde{\varphi}^N_z (W) - W \|_\infty \leq \int_0^z \|\tilde{f}^N_s (W) \|_\infty ds \leq |z| C_f \|W\|_\infty (1 + 6eC_a) \,. \]
    Let $\tilde{C} := 2 + 6e C_a + 36 C_a \|W\|_\infty$, we bound $\| g(h) \|_\infty$ as,
    \begin{align*}
        \| g(h) \|_\infty &\leq \epsilon \tilde{C} C_f \|W\|_\infty \Big(\frac{h}{\epsilon}\Big)^{N+1} \leq h \tilde{C} C_f \|W\|_\infty \Big(\frac{h}{\epsilon}\Big)^N \\
                          &= h \tilde{C} C_f \|W\|_\infty \Big(\frac{hN \|W\|_\infty}{e h_0}\Big)^N \leq h \tilde{C} C_f \|W\|_\infty e^{-N} \leq h e \tilde{C} C_f \|W\|_\infty e^{-h_0/(h\|W\|_\infty)} \,,
    \end{align*}
    using $hN \leq h_0/\|W\|_\infty$. 
    Since $N \leq h_0 / (h\|W\|_\infty) < N+1$, we have $e^{-N} < e \cdot e^{-h_0 / (h\|W\|_\infty)}$ and the statement follows.
\end{proof}

We now turn to the truncated modified Hamiltonian 
\[ \tilde{H}^N_h (W) := H(W) + h^p H_{p+1} (W) + \cdots + h^{N-1} H_N (W) \,, \]
where $f_j(W) = \nabla H_j(W)^\dagger$. 
We use the bound \cref{eq:bound_truncated_modeq} of $\|f_h^N(W)\|_\infty$ as the Lipschitz constant of the truncated Hamiltonian $\tilde{H}_h^N$ in the following proposition, which is used to prove the conservation of the modified Hamiltonian for exponential time in \cref{thm:long_time_Hamiltonian_conservation}.

\begin{proposition}
    \label{prop:bounds2}
    Assume $\|W\|_\infty = \|\tilde{W}\|_\infty$. We have the following bound,
    \[ | \tilde{H}^N_h (W) - \tilde{H}^N_h (\tilde{W}) | \leq L C_f \|W\|_\infty \| W - \tilde{W} \|_\infty \,, \]
    where 
    $L :=  1 +  6 e C_a$.
\end{proposition}
\begin{proof}
    We recall that $\nabla \tilde{H}^N_h (W) = \tilde{f}^N_h (W)^\dagger$ which implies $\| \nabla \tilde{H}^N_h (W) \|_\infty = \| \tilde{f}^N_h (W) \|_\infty$ , therefore, by the Intermediate Value Theorem, there exists $Z$ on the line segment between $W$ and $\tilde{W}$ such that,
    \begin{align*}
        | \tilde{H}^N_h (W) - \tilde{H}^N_h (\tilde{W}) | &\leq \| \nabla \tilde{H}^N_h (Z) \|_\infty \| W - \tilde{W} \|_\infty = \| \tilde{f}^N_h (Z) \|_\infty \| W - \tilde{W} \|_\infty \\
                                                          &= L C_f \|Z\|_\infty \| W - \tilde{W} \|_\infty \,,
    \end{align*}
    where we used the bound \cref{eq:bound_truncated_modeq} of $\|\tilde{f}^N_h(W)\|_\infty$. Since $Z = t W + (1 - t) \tilde{W}$ for some $t \in [0,1]$ and $\|W\|_\infty = \|\tilde{W}\|_\infty$, we have $\|Z\|_\infty \leq \|W\|_\infty$
    and the bound follows.
\end{proof}

\Cref{thm:long_time_Hamiltonian_conservation} is an adaptation to our setting of the well-known result that symplectic integrators for finite-dimensional systems conserve the modified Hamiltonian over long times, see, e.g., \cite[Thm. IX.8.1]{HairerWannerGNI}.

\begin{theorem}
    \label{thm:long_time_Hamiltonian_conservation}
    Under the assumptions in \cref{thm:bound_modflow}, let $\tilde{H}^N_h$ be the modified Hamiltonian of the flow $\varphi_h^N$ and let $W_{k+1} = \Phi_h(W_k)$. Then
    \[ | \tilde{H}^N_h (W_k) - \tilde{H}^N_h (W_0) | \leq L C_f^2   \|W_0\|_\infty^2 C(\|W_0\|_\infty) e^{-h_0/(2h\|W_0\|_\infty)} \,, \]
    for all $k$ such that $kh \leq e^{h_0/(2h\|W_0\|_\infty)}$.
\end{theorem}
\begin{proof}
    We use the following telescopic sum to bound $| \tilde{H}^N_h (W_k) - H(W_0) |$,
    \begin{align*}
        \tilde{H}^N_h (W_k) - \tilde{H}^N_h(W_0) | &\leq \sum_{i=0}^{n-1} | \tilde{H}^N_h (W_{i+1}) - \tilde{H}^N_h (W_i) | \\
        \intertext{use $\tilde{H}^N_h (W_i) = \tilde{H}^N_h \big( \tilde{\varphi}^N_h(W_i) \big)$,}
                                     &\leq \sum_{i=0}^{k-1} | \tilde{H}^N_h (W_{i+1}) - \tilde{H}^N_h \big(\tilde{\varphi}^N_h (W_i)\big) | \\
                                     \intertext{apply \cref{prop:bounds2} and note that $\|W_0\| _\infty = \|W_{i+1} \|_\infty = \| \tilde{\varphi}^N_h (W_i) \|_\infty$,}
                                     &\leq L C_f^2 \|W_0\|_\infty^2 C(\|W_0\|_\infty) kh e^{-h_0/(h\|W_0\|_\infty)} \,.
    \end{align*}
    The statement follows from the assumption $kh \leq e^{h_0/(2h\|W_0\|_\infty)}$.
\end{proof}

\section{Numerical experiments}\label{sec:numerics}

In this section we carry out numerical simulations with the ISOMP method~\eqref{eq:isomp} applied to the Euler--Zeitlin equations~\eqref{eq:euler-zeitlin}.
We use the Python package QUFLOW.\footnote{Available at \href{https://github.com/klasmodin/quflow}{github.com/klasmodin/quflow}.}
For initial data, we select random normally distributed spherical harmonic coefficients $\omega_{\ell,m}$ for $\ell \leq 16$.
For each matrix size $n$, this gives rise to an initial data matrix $W_0 \in \mathfrak{su}(n)$.
For a detailed example of how to solve the Euler--Zeitlin equation with QUFLOW, see Modin and Viviani~\cite{MoVi2026b}.
Our simulation here follows the same setup.

As a first step, let us compute a few terms of the modified Hamiltonian. 
Let $H_n(W)$ denote the Hamiltonian~\eqref{eq:zeitlin_hamiltonian} and let $\langle A,B\rangle_2 := \frac{4\pi}{n}\Trace(A^\dagger B)$ so that $H_n(W) = \frac12\langle W,(-\Delta)^{-1}W\rangle_2$.
Following the discussion in \cref{sec:biplanar_elementary_Hamiltonians} and using the notation,
\[ S = \mathcal L^{-1}_n W := (-\Delta_n)^{-1}W  \,, \quad U = S^2 + \mathcal L^{-1}_n [S, W] \,, \quad \epsilon = h/\hbar_n \]
the modified Hamiltonian $\tilde{H}_{n,h}(W)$ corresponding to $H_n(W)$ is given by,
\begin{align*}
    \tilde{H}_{n,h}(W) = H_n(W) &- \frac1{12} \epsilon^2 \langle S W S^\dagger, S \rangle_2 - \frac1{24} \epsilon^2 \langle [S, W], \mathcal L^{-1}_n [S, W] \rangle_2 \\
                              &+ \frac1{80} \epsilon^4 \langle U W U^\dagger, S \rangle_2 + \frac1{160} \epsilon^4 \langle U W + W U^\dagger, \mathcal L^{-1}_n (U W) + W U^\dagger \rangle_2 \\
                              &+ \frac1{240} \epsilon^4 \langle U W S^\dagger + S W U^\dagger, \mathcal L^{-1}_n[S, W] \rangle_2 \\
                              &+ \frac1{240} \epsilon^4 \langle U W + W U^\dagger, \mathcal L^{-1}_n (S W S^\dagger) \rangle_2 \\
                              &+ \frac7{480} \epsilon^4 \langle S W S^\dagger, \mathcal L^{-1}_n (S W S^\dagger) \rangle_2 + \OO(\epsilon^6) \,.
\end{align*}
Notice, as expected, that any truncation $\tilde{H}_{n,h}^N(W)$ only depends on $h$ via $\epsilon$, and on $n$ via $\langle\cdot,\cdot\rangle_2$ and $\mathcal L^{-1}_n$. 
Furthermore, under the quantization $\mathcal T_n$, these operations converge to the corresponding infinite-dimensional operations as $n\to \infty$ (see~\cite{ModinTFM24}).

Carrying out short simulations, with $n=128$ and $0 \leq t\leq 2$, for 5 different values of $\epsilon$, we show in \cref{fig:mod_ham} the maximum error in the Hamiltonian $\lvert H_{n}(W_k)-H_{n}(W_0)\rvert$ an the modified Hamiltonian $\lvert \tilde H^N_{n}(W_k)-\tilde H^N_{n}(W_0)\rvert$ truncated to $N=4$.
As expected, the Hamiltonian $H_n$ is preserved up to the order $\mathcal{O}(\epsilon^2)$ of the ISOMP method, whereas the modified Hamiltonian $\tilde H_{n,h}$ is preserved up to order $\mathcal{O}(\epsilon^6)$ (there are no odd $\epsilon^p$-terms since the method is symmetric).

In much longer simulations, $0\leq t\leq 2500$ and fixed $\epsilon = 0.1$, we show in \cref{fig:ham_time_plot} how the relative error of the Hamiltonian $H_n$ varies with time for 4 different choices of $n$.
Notice that it fluctuates in all simulations, but the magnitudes of the fluctuations do not grow nor decrease for growing $n$.

\begin{figure}
    \centering
    \includegraphics{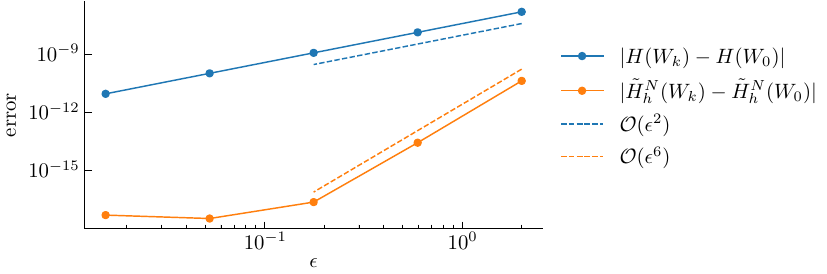}
    \caption{Error in the Hamiltonian and modified truncated Hamiltonian for the ISOMP method~\eqref{eq:isomp} applied to the Euler--Zeitlin equations~\eqref{eq:euler-zeitlin} with different step size parameters $\epsilon = h/\hbar_n$. 
    The matrix size is $n=128$, the truncation for the modified Hamiltonian is $N=4$, and the simulation time interval is $[0,2]$. As expected, the Hamiltonian error decreases as $\mathcal{O}(\epsilon^2)$, whereas the modified Hamiltonian error decreases as $\mathcal{O}(\epsilon^{N+2}) = \mathcal{O}(\epsilon^{6})$ until machine precision is reached (since the method is symmetric there are only even terms in the expansion).}
    \label{fig:mod_ham}
\end{figure}

\begin{figure}
    \centering
    \includegraphics{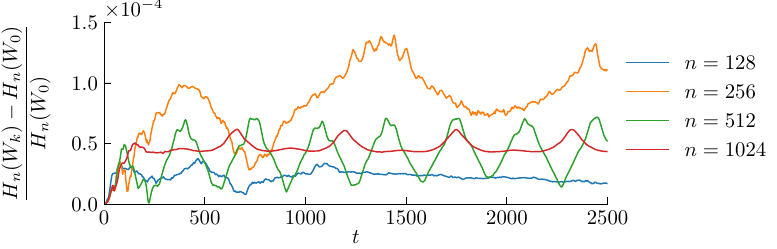}
    \caption{Relative error in the Hamiltonian over time for the ISOMP method~\eqref{eq:isomp} applied to the Euler--Zeitlin equations~\eqref{eq:euler-zeitlin} for different matrix size $n$ but fixed $\epsilon = h/\hbar_n$. 
    Notice that the energy fluctuates but remains bounded as $t$ and $n$ grows, in alignment with the results in Theorem~\ref{thm:long_time_Hamiltonian_conservation}.}
    \label{fig:ham_time_plot}
\end{figure}

\appendix

\section{Estimates for the Hoppe--Yau operator}\label{sec:hoppe_yau_estimates}

The Hoppe--Yau operator~\cite{HoYa1998} (or Hoppe--Yau Laplacian) is the matrix mapping $$\Delta_n\colon \mathfrak{u}(n) \to \mathfrak{u}(n)$$ 
defined by
\begin{equation}\label{eq:hoppe_yau}
    \Delta_n P = \frac{1}{\hbar^2}\sum_{\alpha=1}^3 [[P, X_\alpha], X_\alpha],
\end{equation}
where the constant $\hbar = 2/\sqrt{n^2-1}$ and the matrices $X_1,X_2,X_3 \in \mathfrak{su}(n)$ are generators for an irreducible, unitary representation of $\mathfrak{so}(3)$ on $\mathbb{C}^n$ scaled so that $\sum_{\alpha=1}^3 X_\alpha^\dagger X_\alpha  = I$. 
In physics, these matrices give a spin-$s$ representation for $s = (n-1)/2$.

The Hoppe--Yau operator is an approximation of the Laplacian on the sphere: via Berezin--Toeplitz quantization, the Poisson algebra of smooth functions  $$\big(C^\infty(S^2,\mathbb{R}), \{ \cdot,\cdot\}\big)$$ is approximated by the finite-dimensional Lie algebra $$\big(\mathfrak{u}(n), \frac{1}{\hbar}[\cdot,\cdot]\big)$$ and in this approximation the Hoppe--Yau operator $\Delta_n$ corresponds to the Laplace--Beltrami operator $\Delta$ (cf.~\cite{HoYa1998}).
From the point-of-view of representation theory, the Hoppe--Yau operator~\cref{eq:hoppe_yau} is the Casimir element for the induced representation on $\mathfrak{u}(n)$ generated by the operators $\mathrm{ad}_{X_1},\mathrm{ad}_{X_2},\mathrm{ad}_{X_2}$.
Thus, the spectrum of $\Delta_n$ is $\{-\ell(\ell+1) \}$ for $\ell = 0,\ldots,n-1$ with eigenspaces $V_\ell$ of dimension $2\ell + 1$ (cf.~\cite{ModinTFM24}). 
It is precisely the truncation of the spectrum of $\Delta$.

Just as the Laplace--Beltrami operator has a 1-dimensional kernel given by constant functions, the Hoppe--Yau operator has a 1-dimensional kernel spanned $\mathrm{i}\mathbb{R}I$.
To make it invertible on $\mathfrak{u}(n)$ we make the following extension,  where we also change the sign
\begin{equation}\label{eq:extededHoppeYau}
    \mathcal{L}_n\colon \mathfrak{u}(n) \to \mathfrak{u}(n), \qquad \mathcal{L}_nP = \frac{\operatorname{tr}(P)I}{n} - \Delta_n.
\end{equation}
This operator has positive spectrum $1,2,6,\ldots,n(n-1)$.
It is invertible and fulfills $\operatorname{tr}(\mathcal{L}_n P) = \operatorname{tr}(P)$.

For $p\in [1,2,\ldots,\infty]$, consider the scaled Schatten $p$-norms on $\mathfrak{u}(n)$ 
\begin{equation}
    \lVert P \rVert_p = \left\{ \begin{aligned} \left(\frac{1}{n}\sum_{i=1}^n \lvert \lambda_i \rvert^p \right)^{1/p} & \quad\text{if $p< \infty$} \\ \max_i(\lvert \lambda_1\rvert,\ldots,\lvert \lambda_n\rvert) & \quad\text{if $p=\infty$}\end{aligned} \right.
\end{equation}
where $\lambda_1,\ldots,\lambda_n$ are the eigenvalues of $P$.
These norms correspond, in the limit $n\to\infty$, to the $L^p$-norms on $C^\infty(S^2,\mathbb{R})$.
Since the smallest eigenvalue of $\mathcal{L}_n$ is~$1$, it follows that
\begin{equation}
    \lVert \mathcal{L}_n^{-1}W \rVert_2 \leq \lVert W \rVert_2 .
\end{equation}
In other words, the operator norm of $\mathcal{L}_n^{-1}$ is unitary relative to the $2$-norm on $\mathfrak{u}(n)$.
Our main result in this appendix is the analogous result for all $p$-norms.

\begin{theorem}\label{thm:hoppe_yau_estimates}
    For $W\in \mathfrak{u}(n)$ and $p\in [1,2,\ldots,\infty]$, the extended Hoppe--Yau operator~\cref{eq:extededHoppeYau} fulfills
    \[
        \lVert \mathcal{L}_n^{-1}W \rVert_p \leq \lVert W \rVert_p .
    \]
    Equality is attained if and only if $W \in \mathrm{i}\mathbb{R}I$.
\end{theorem}

\subsection{Proof via positivity}

Recall that a linear operator $\Phi\colon \mathfrak{gl}(n,\mathbb{C}) \to \mathfrak{gl}(n,\mathbb{C})$ between complex matrices is positive if it maps positive Hermitian matrices to positive Hermitian matrices.
For details on positive operators, we refer to the monograph by Bhatia~\cite{Bh2015}.

The operator \cref{eq:extededHoppeYau} naturally extends from $\mathfrak{u}(n)$ to an operator on all complex matrices $\mathfrak{gl}(n,\mathbb{C})$.
Indeed, this extension corresponds to the complexification $\mathfrak{u}(n)\otimes\mathbb{C} \simeq \mathfrak{gl}(n,\mathbb{C})$.

\begin{theorem}\label{thm:positivity}
    The inverse operator $\mathcal{L}_n^{-1}\colon \mathfrak{gl}(n,\mathbb{C})\to \mathfrak{gl}(n,\mathbb{C})$ is positive.
    Thus, if $W \in \mathrm{i}\mathfrak{u}(n)$ is positive (i.e., all its eigenvalues are non-negative), then so is $\mathcal{L}_n^{-1}W$.
    Furthermore, the inverse operator is unital, $\mathcal{L}_n^{-1}I = I$.
\end{theorem}

For the proof, we use the following result.

%
%
\begin{lemma}\label{lem:conjugate_form}
    The operator $\mathcal{L}_n$ applied to $P\in \mathfrak{gl}(n,\mathbb{C})$ can be written
    \begin{equation*}
        \mathcal{L}_nP = \frac{\operatorname{tr}(P)I}{n} + \frac{2}{\hbar^2}\left( P - \sum_{\alpha=1}^3 X_\alpha^\dagger P X_\alpha \right)
    \end{equation*}
\end{lemma}

\begin{proof}
    Direct calculations yield
    \begin{align*}
        \mathcal{L}_nP &= \frac{\operatorname{tr}(P)I}{n} - \frac{1}{\hbar^2}\sum_{\alpha=1}^3 [[P,X_\alpha],X_\alpha]  = \\
        &= \frac{\operatorname{tr}(P)I}{n} - \frac{1}{\hbar^2}\sum_{\alpha=1}^3\left( (PX_\alpha-X_\alpha P)X_\alpha - X_\alpha (PX_\alpha-X_\alpha P) \right)  = \\
        &= \frac{\operatorname{tr}(P)I}{n} - \frac{1}{\hbar^2}\sum_{\alpha=1}^3 \left(PX_\alpha^2 -X_\alpha PX_\alpha - X_\alpha PX_\alpha + X_\alpha^2 P \right)  = \\
        &= \frac{\operatorname{tr}(P)I}{n} - \frac{1}{\hbar^2}P\left(\sum_{\alpha=1}^3 X_\alpha^2\right) - \frac{1}{\hbar^2}\left(\sum_{\alpha=1}^3 X_\alpha^2\right)P +  \frac{2}{\hbar^2}\sum_{\alpha=1}^3 X_\alpha PX_\alpha.
    \end{align*}
    The results now follows since $X_\alpha^\dagger = -X_\alpha$ and $\sum_{\alpha=1}^3 X_\alpha^\dagger X_\alpha = I$.
\end{proof}

\begin{proof}[Proof of \cref{thm:positivity}]
    First, we note that the operator $P\mapsto X_\alpha^\dagger P X_\alpha$ is positive.
    Since a positive sum of positive operators is positive, it follows that 
    \[
        \Phi(P) = \sum_{\alpha=1}^3 X_\alpha^\dagger P X_\alpha
    \]
    is positive. 
    The eigenspaces of $\Phi$ are the same as the eigenspace of $\Delta_n$.
    Indeed, if $P\in V_\ell$, so that $\Delta_n P = -\ell(\ell+1)P$, then we get from \cref{lem:conjugate_form} that
    \[
        \Phi(P) = \frac{\hbar^2}{2}\Delta_n P + P = (1-\frac{\hbar^2\ell(\ell+1)}{2})P.
    \]
    It is then straightforward to verify, on each eigenspace $V_\ell$, that the inverse of $\mathcal{L}_n$ can be expressed as
    \begin{equation}\label{eq:inverseL}
        \mathcal{L}_n^{-1} = \frac{\operatorname{tr}(\cdot)I}{n} + \frac{\hbar^2}{2}\left(\mathrm{Id}-\Phi \right)^{-1}.
    \end{equation}
    Since $\hbar^2 = 4/(n^2-1)$ and since $$\ell(\ell+1) \leq n(n-1) < n^2-1$$ it follows that the spectrum of $\Phi$ is strictly contained in $(-1,1)$.
    Thus, $(\mathrm{Id}-\Phi)^{-1}$ can be expanded in a converging power series. 
    Consequently, it follows from equation~\cref{eq:inverseL} that the inverse of $\mathcal{L}_n$ is 
    \begin{equation*}
        \mathcal{L}_n^{-1} = \frac{\operatorname{tr}(\cdot)I}{n} + \frac{\hbar^2}{2}\sum_{k=0}^\infty \Phi^k .
    \end{equation*}
    The first term $\operatorname{tr}(\cdot)I/n$ is positive (see Bhatia~\cite[example 2.2.1]{Bh2015}). 
    For each $k$, the operator $\Phi^k$ is positive since $\Phi$ is positive.
    Thus, all the terms in the series for $\mathcal{L}_n^{-1}$ are positive, so it follows that $\mathcal{L}_n^{-1}$ is positive.
    Since $\Phi(I) = 0$ it also follows that $\mathcal{L}_n^{-1}$ is unital.
\end{proof}

\begin{proof}[Proof of \cref{thm:hoppe_yau_estimates}]
    For any positive unital operator, its operator norm relative to the spectral norm is 1 (the Russo--Dye theorem).
    Thus, it immediately follows that
    \[
        \lVert \mathcal{L}_n^{-1}W\rVert_\infty \leq \underbrace{\lVert \mathcal{L}_n^{-1} \rVert_\infty}_{=1} \lVert W\rVert_\infty 
        = \lVert W\rVert_\infty.
    \]
    For $p<\infty$ we proceed as follows.
    Since $\mathcal{L}_n^{-1}$ is positive unital, it follows that for any convex function $f\colon \mathbb{R}\to\mathbb{R}$ and any $W\in\mathrm{i}\mathfrak{u}(n)$ we have a Jensen-type inequality
    \begin{equation}\label{eq:jensen}
        f(\mathcal{L}_n^{-1} W ) \leq \mathcal{L}_n^{-1}f(W)
    \end{equation}
    where $f$ is applied to elements of $\mathrm{i}\mathfrak{u}(n)$ via diagonalization (the inequality $A\leq B$ between Hermitian matrices means that $B-A$ is positive).
    For the convex function $f(x) = \lvert x\rvert ^p$ we then get
    \[
        \lVert \mathcal{L}_n^{-1}W \rVert_p = \operatorname{tr}(f(\mathcal{L}_n^{-1}W))^{1/p}
        \leq \operatorname{tr}(\mathcal{L}_n^{-1}f(W))^{1/p} \leq
        \operatorname{tr}(f(W))^{1/p} = \lVert W\rVert_p
    \]
    where we first used the inequality \cref{eq:jensen} and then used that $f(W)$ is positive and $\mathcal{L}_n^{-1}$ is unital and positive.
    This concludes the proof of the main theorem since $\lVert W\rVert_p = \lVert \mathrm{i}W\rVert_p$.
\end{proof}

\bibliographystyle{siamplain}
\bibliography{references.bib}

@book{Bh2015,
  title = {Positive definite matrices},
  author = {Bhatia, R.},
  year = 2015,
  publisher = {Princeton University Press},
  address = {Princeton Oxford}
}

@article{MoVi2026b,
    author = {Modin, K. and Viviani, M.},
    journal = {J. Comput. Dyn.},
    pages = {17--35},
    title = {A brief introduction to matrix hydrodynamics},
    volume = {14},
    year = {2026},
}

@article{HoYa1998,
    author = {Hoppe, J. and Yau, S.-T.},
    journal = {Comm. Math. Phys.},
    number = {1},
    pages = {67--77},
    title = {Some properties of matrix harmonics on {$S^2$}},
    volume = {195},
    year = {1998},
}

@article{BaFaGr2013,
    title = {Existence and stability of ground states for fully discrete
             approximations of the nonlinear Schr\"odinger equation},
    author = {Bambusi, D. and Faou, E. and Gr{\'e}bert, B.},
    year = 2013,
    journal = {Numer. Math.},
    volume = {123},
    number = {3},
    pages = {461--492},
    doi = {10.1007/s00211-012-0491-7},
}

@article{FaGr2011,
    title = {Hamiltonian Interpolation of Splitting Approximations for Nonlinear
             PDEs},
    author = {Faou, E. and Gr{\'e}bert, B.},
    year = 2011,
    journal = {Found. Comp. Math.},
    volume = {11},
    number = {4},
    pages = {381--415},
    doi = {10.1007/s10208-011-9094-4},
}

@misc{FaMaSc2025arxiv,
    title = {Fully discrete backward error analysis for the midpoint rule
             applied to the nonlinear Schroedinger equation},
    author = {Faou, E. and Maierhofer, G. and Schratz, K.},
    year = 2025,
    number = {arXiv:2505.03271},
    eprint = {2505.03271},
    primaryclass = {math},
    publisher = {arXiv},
    doi = {10.48550/arXiv.2505.03271},
    archiveprefix = {arXiv},
}

@article{Ze1991,
    author = {Zeitlin, V.},
    journal = {Phys. D},
    number = {3},
    pages = {353--362},
    title = {Finite-mode analogs of {$2$}{D} ideal hydrodynamics: coadjoint
             orbits and local canonical structure},
    volume = {49},
    year = {1991},
}

@article{Ze2004,
    author = {Zeitlin, V.},
    journal = {Phys. Rev. Lett.},
    pages = {264501},
    title = {Self-Consistent Finite-Mode Approximations for the Hydrodynamics of
             an Incompressible Fluid on Nonrotating and Rotating Spheres},
    volume = {93},
    year = {2004},
}

@book{Le2018,
    author = {Le Floch, Y.},
    publisher = {Springer},
    title = {A brief introduction to {B}erezin-{T}oeplitz operators on compact {
             K}{\"a}hler manifolds},
    year = {2018},
}

@article{Ha1982,
    author = {Hamilton, Richard S.},
    journal = {Bull. Amer. Math. Soc. (N.S.)},
    number = {1},
    pages = {65--222},
    title = {The inverse function theorem of {N}ash and {M}oser},
    volume = {7},
    year = {1982},
}

@book{Ar1989,
    address = {New York},
    author = {Arnold, V. I.},
    edition = {Second},
    publisher = {Springer-Verlag},
    series = {Graduate Texts in Mathematics},
    title = {Mathematical Methods of Classical Mechanics},
    volume = {60},
    year = {1989},
}

@article{BeGi1994,
    author = {Benettin, Giancarlo and Giorgilli, Antonio},
    journal = {J. Statist. Phys.},
    number = {5-6},
    pages = {1117--1143},
    title = {On the {H}amiltonian interpolation of near-to-the-identity
             symplectic mappings with application to symplectic integration
             algorithms},
    volume = {74},
    year = {1994},
}

@article{bogfjellmoHamiltonianBseriesLie2017,
    title = {Hamiltonian {{B-series}} and a {{Lie}} Algebra of Non-Rooted Trees},
    author = {Bogfjellmo, Geir and Curry, Charles and Manchon, Dominique},
    year = {2017},
    journal = {Numerische Mathematik},
    shortjournal = {Numer. Math.},
    volume = {135},
    number = {1},
    pages = {97--112},
    issn = {0945-3245},
    doi = {10.1007/s00211-016-0792-3},
    url = {https://doi.org/10.1007/s00211-016-0792-3},
    urldate = {2025-05-27},
    langid = {english},
}

@article{butcherCoefficientsStudyRungeKutta1963,
    title = {Coefficients for the Study of {{Runge-Kutta}} Integration Processes
             },
    author = {Butcher, J. C.},
    year = {1963},
    journal = {Journal of the Australian Mathematical Society},
    volume = {3},
    number = {2},
    pages = {185--201},
    issn = {0004-9735},
    doi = {10.1017/S1446788700027932},
    url = {
           https://www.cambridge.org/core/journals/journal-of-the-australian-mathematical-society/article/coefficients-for-the-study-of-rungekutta-integration-processes/457E8B0413C29B9D7BBCD7D2A45A23D5
           },
    urldate = {2025-05-25},
    langid = {english},
}

@article{calaqueTwoInteractingHopf2011,
    title = {Two Interacting {{Hopf}} Algebras of Trees},
    author = {Calaque, Damien and Ebrahimi-Fard, Kurusch and Manchon, Dominique},
    year = {2011},
    journal = {Advances in Applied Mathematics},
    shortjournal = {Advances in Applied Mathematics},
    volume = {47},
    number = {2},
    eprint = {0806.2238},
    eprinttype = {arXiv},
    eprintclass = {math},
    pages = {282--308},
    issn = {01968858},
    doi = {10.1016/j.aam.2009.08.003},
    url = {http://arxiv.org/abs/0806.2238},
    urldate = {2025-04-24},
}

@article{EngoNIL01,
    title = {Numerical {{Integration}} of {{Lie--Poisson Systems While
             Preserving Coadjoint Orbits}} and {{Energy}}},
    author = {Engø, Kenth and Faltinsen, Stig},
    year = {2001},
    journal = {SIAM Journal on Numerical Analysis},
    shortjournal = {SIAM J. Numer. Anal.},
    volume = {39},
    number = {1},
    pages = {128--145},
    publisher = {{Society for Industrial and Applied Mathematics}},
    issn = {0036-1429},
    doi = {10.1137/S0036142999364212},
    url = {https://epubs.siam.org/doi/10.1137/S0036142999364212},
    urldate = {2025-06-10},
}

@article{hairerButcherGroupGeneral1974,
    title = {On the {{Butcher}} Group and General Multi-Value Methods},
    author = {Hairer, E. and Wanner, G.},
    year = {1974},
    journal = {Computing},
    shortjournal = {Computing},
    volume = {13},
    number = {1},
    pages = {1--15},
    issn = {1436-5057},
    doi = {10.1007/BF02268387},
    url = {https://doi.org/10.1007/BF02268387},
    urldate = {2025-05-25},
    langid = {english},
}

@book{HairerWannerGNI,
    title = {Geometric {{Numerical Integration}}: {{Structure-Preserving
             Algorithms}} for {{Ordinary Differential Equations}}},
    shorttitle = {Geometric {{Numerical Integration}}},
    author = {Hairer, Ernst and Lubich, Christian and Wanner, Gerhard},
    year = {2010},
    eprint = {ssrFQQAACAAJ},
    eprinttype = {googlebooks},
    publisher = {Springer Berlin Heidelberg},
    isbn = {978-3-642-05157-9},
    langid = {english},
    pagetotal = {644},
}

@article{ModinLPM20,
    title = {Lie–{{Poisson Methods}} for {{Isospectral Flows}}},
    author = {Modin, Klas and Viviani, Milo},
    year = {2020},
    journal = {Foundations of Computational Mathematics},
    shortjournal = {Found Comput Math},
    volume = {20},
    number = {4},
    pages = {889--921},
    issn = {1615-3383},
    doi = {10.1007/s10208-019-09428-w},
    url = {https://doi.org/10.1007/s10208-019-09428-w},
    urldate = {2025-06-03},
    langid = {english},
}

@article{Munthe-KaasHOR99,
    title = {High Order {{Runge-Kutta}} Methods on Manifolds},
    author = {Munthe-Kaas, Hans},
    date = {1999-01-01},
    journaltitle = {Applied Numerical Mathematics},
    shortjournal = {Applied Numerical Mathematics},
    series = {Proceedings of the {{NSF}}/{{CBMS Regional Conference}} on {{
              Numerical Analysis}} of {{Hamiltonian Differential Equations}}},
    volume = {29},
    number = {1},
    pages = {115--127},
    issn = {0168-9274},
    doi = {10.1016/S0168-9274(98)00030-0},
    url = {https://www.sciencedirect.com/science/article/pii/S0168927498000300},
    urldate = {2025-07-25},
}

@inproceedings{munthe-kaasLieButcherSeries2018,
    title = {Lie–{{Butcher Series}}, {{Geometry}}, {{Algebra}} and {{Computation
             }}},
    booktitle = {Discrete {{Mechanics}}, {{Geometric Integration}} and {{Lie}}–{
                 {Butcher Series}}},
    author = {Munthe-Kaas, Hans Z. and Føllesdal, Kristoffer K.},
    editor = {Ebrahimi-Fard, Kurusch and Barbero Liñán, María},
    year = {2018},
    pages = {71--113},
    publisher = {Springer International Publishing},
    location = {Cham},
    doi = {10.1007/978-3-030-01397-4_3},
    isbn = {978-3-030-01397-4},
    langid = {english},
}

@article{Munthe-KaasRML98,
    title = {Runge-{{Kutta}} Methods on {{Lie}} Groups},
    author = {Munthe-Kaas, Hans},
    date = {1998-03-01},
    journaltitle = {BIT Numerical Mathematics},
    shortjournal = {Bit Numer Math},
    volume = {38},
    number = {1},
    pages = {92--111},
    issn = {1572-9125},
    doi = {10.1007/BF02510919},
    url = {https://doi.org/10.1007/BF02510919},
    urldate = {2025-07-23},
    langid = {english},
}

@article{ModinTFM24,
    author = {Modin, Klas and Viviani, Milo},
    title = {Two-Dimensional Fluids Via Matrix Hydrodynamics},
    journal = {Archive for Rational Mechanics and Analysis},
    year = {2026},
    volume = {250},
    number = {1},
    pages = {10},
    issn = {1432-0673},
    doi = {10.1007/s00205-025-02154-4},
    url = {https://doi.org/10.1007/s00205-025-02154-4},
}

@article{Cayley1857,
    author = {Cayley, A.},
    title = {On the theory of the analytical forms called trees},
    journal = {Philosophical Magazine},
    volume = {13},
    pages = {172--176},
    year = {1857},
    doi = {10.1080/14786445708642866},
}

@article{MuruaHAR06,
    title = {The {{Hopf Algebra}} of {{Rooted Trees}}, {{Free Lie Algebras}},
             and {{Lie Series}}},
    author = {Murua, A.},
    year = 2006,
    month = nov,
    journal = {Foundations of Computational Mathematics},
    volume = {6},
    number = {4},
    pages = {387--426},
    issn = {1615-3383},
    doi = {10.1007/s10208-003-0111-0},
    urldate = {2026-02-12},
    langid = {english},
}

@article{ChartierNIB07,
    title = {Numerical {{Integrators Based}} on {{Modified Differential
             Equations}}},
    author = {Chartier, Philippe and Hairer, Ernst and Vilmart, Gilles},
    year = 2007,
    journal = {Mathematics of Computation},
    volume = {76},
    number = {260},
    eprint = {40234469},
    eprinttype = {jstor},
    pages = {1941--1953},
    publisher = {American Mathematical Society},
    issn = {0025-5718},
    urldate = {2026-05-20},
}

@article{ChartierASB10a,
    title = {Algebraic {{Structures}} of {{B-series}}},
    author = {Chartier, Philippe and Hairer, Ernst and Vilmart, Gilles},
    year = 2010,
    month = aug,
    journal = {Foundations of Computational Mathematics},
    volume = {10},
    number = {4},
    pages = {407--427},
    issn = {1615-3383},
    doi = {10.1007/s10208-010-9065-1},
    urldate = {2026-05-20},
    langid = {english},
}

@book{HairerSOD96,
    title = {Solving {{Ordinary Differential Equations II}}},
    author = {Hairer, Ernst and Wanner, Gerhard},
    year = 1996,
    series = {Springer {{Series}} in {{Computational Mathematics}}},
    volume = {14},
    publisher = {Springer},
    address = {Berlin, Heidelberg},
    doi = {10.1007/978-3-642-05221-7},
    urldate = {2025-11-12},
    copyright = {http://www.springer.com/tdm},
    isbn = {978-3-642-05220-0 978-3-642-05221-7},
}

@article{OtterNT48,
    title = {The {{Number}} of {{Trees}}},
    author = {Otter, Richard},
    year = 1948,
    journal = {Annals of Mathematics},
    volume = {49},
    number = {3},
    eprint = {1969046},
    eprinttype = {jstor},
    pages = {583--599},
    publisher = {[Annals of Mathematics, Trustees of Princeton University on
                 Behalf of the Annals of Mathematics, Mathematics Department,
                 Princeton University]},
    issn = {0003-486X},
    doi = {10.2307/1969046},
    urldate = {2026-06-09},
}

\end{document}